\documentclass{amsart}
\usepackage{color,url}

\ifx\pdfpageheight\undefined
   \usepackage[dvips,colorlinks=true,linkcolor=blue,citecolor=red,%
      urlcolor=green]{hyperref}
   \usepackage[dvips]{graphicx}
   \makeatletter
   \edef\Gin@extensions{\Gin@extensions,.mps}
   \makeatother
\else
 \usepackage[pdftex]{graphicx}
   \usepackage[bookmarksopen=false,pdftex=true,breaklinks=true,%
      backref=page,pagebackref=true,plainpages=false,%
      hyperindex=true,pdfstartview=FitH,colorlinks=true,%
      pdfpagelabels=true,colorlinks=true,linkcolor=blue,%
      citecolor=red,urlcolor=green,hypertexnames=false%
      ]%
   {hyperref}
\fi

\makeatletter
\let\orgdescriptionlabel\descriptionlabel
\renewcommand*{\descriptionlabel}[1]{%
  \let\orglabel\label
  \let\label\@gobble
  \phantomsection
  \edef\@currentlabel{#1}%
  \let\label\orglabel
  \orgdescriptionlabel{#1}%
}
\makeatother

\newcommand{\hide}[1]{}
\newtheorem{theorem}{Theorem}[section]
\newtheorem{lemma}[theorem]{Lemma}
\newtheorem{corollary}[theorem]{Corollary}
\newtheorem{proposition}[theorem]{Proposition}

\theoremstyle{definition}
\newtheorem{definition}[theorem]{Definition}

\newtheorem{notation}[theorem]{Notation}

\newtheorem{algorithm}{\sc Algorithm}

\theoremstyle{remark}
\newtheorem{remark}[theorem]{Remark}

\hide{
\usepackage{graphicx,color,url}
\documentclass{amsart}
\usepackage[english]{babel}
\usepackage{amssymb,graphicx,enumerate,
hyperref,makeidx,xcolor
,theorem
}
\usepackage{geometry}
\geometry{portrait}
}

\catcode`\<=\active \def<{
\fontencoding{T1}\selectfont\symbol{60}\fontencoding{\encodingdefault}}
\catcode`\>=\active \def>{
\fontencoding{T1}\selectfont\symbol{62}\fontencoding{\encodingdefault}}
\catcode`\|=\active \def|{
\fontencoding{T1}\selectfont\symbol{124}\fontencoding{\encodingdefault}}
\newcommand{\assign}{:=}
\newcommand{\nin}{\not\in}
\newcommand{\subindex}[2]{\index{#1!#2}}
\newcommand{\tmdummy}{$\mathrm{}$}
\newcommand{\tmem}[1]{{\em #1\/}}
\newcommand{\tmname}[1]{\textsc{#1}}

\newcommand{\tmrsub}[1]{\ensuremath{_{\textrm{#1}}}}
\newcommand{\tmstrong}[1]{\textbf{#1}}

\newcommand{\tmtextrm}[1]{{\rmfamily{#1}}}

\newcommand{\um}{-}
\newcommand {\D}     {\mbox{\rm D}}
\newcommand {\R}     {\mbox{\rm R}}
\newcommand {\ZZ}     {\mbox{\rm Zer}}
\newcommand {\C}     {\mbox{\rm C}}
\newcommand {\RM}     {\mbox{\rm RM}}
\newcommand {\eps}     {\varepsilon}
\newcommand {\Ext}     {\mbox{\rm Ext}}
\newcommand {\la}     {\langle}
\newcommand {\ra}     {\rangle}
\newcommand {\Def}      {\mbox{\rm Def}}
\newcommand {\PP}      {\mathbb{P}}
\newcommand {\Der}      {\mbox{\rm Der}}
\newcommand {\Cc}{\mathrm{Cc}}
\newcommand {\card}{\mathrm{card}}

\theoremstyle{remark}
\newenvironment{descriptioncompact}{\begin{description} }{\end{description}}

\begin{document}

\title{Divide and conquer roadmap for algebraic sets}

\author[Basu]{Saugata Basu}
\address{Department of Mathematics\\
Purdue University\\
West Lafayette, IN 47907\\
USA
}
\email{sbasu@math.purdue.edu}
\urladdr{\url{http://www.math.purdue.edu/~sbasu}}

\author[Roy]{Marie-Fran{\c c}oise Roy}
\address{IRMAR (UMR CNRS 6625)\\
Universite de Rennes 1\\
Campus de Beaulieu\\
35042 Rennes, cedex\\
France}
\email{marie-francoise.roy@univ-rennes1.fr}
\urladdr{\url{http://perso.univ-rennes1.fr/marie-francoise.roy/}}

\thanks{The first author was partially supported by NSF grants
  CCF-0915954,  CCF-1319080 and DMS-1161629.  }

{\maketitle}

\begin{abstract}
 Let $\R$ be a real closed field, and $\D \subset \R$ an ordered domain. We
describe an algorithm that given as input a polynomial $P \in \D [ X_{1} ,
\ldots ,X_{k} ]$, and a finite set, $\mathcal{A}= \{ p_{1} , \ldots ,p_{m}
\}$, of points contained in $V= \ZZ ( P, \R^{k})$ described by
real univariate representations, computes a roadmap of $V$ containing
$\mathcal{A}$. The complexity of the algorithm, measured by the number of
arithmetic operations in $\D$ is bounded by $\left( \sum_{i=1}^{m} D^{O (
\log^{2} ( k ) )}_{i} +1 \right) ( k^{\log ( k )} d )^{O ( k\log^{2} ( k
))}$, where $d= \deg ( P )$, and $D_{i}$ is the degree of the real
univariate representation describing the point $p_{i}$. The best previous
algorithm for this problem had complexity $\mathrm{card} ( \mathcal{A} )^{O (
1 )} d^{O ( k^{3/2} )}$ {\cite{BRMS10}}, where it is assumed that the
degrees of the polynomials appearing in the representations of the points in
$\mathcal{A}$ are bounded by $d^{O ( k )}$. As an application of our result
we prove that for any real algebraic subset $V$ of $\mathbb{R}^{k}$ defined
by a polynomial of degree $d$, any connected component $C$ of $V$
contained in the unit ball, and any two points of $C$, there exist a
semi-algebraic path connecting them in $C$, of length
at most $( k ^{\log (k )} d )^{O ( k\log
( k ) )}$,
consisting of at most $( k ^{\log (k )} d )^{O ( k\log ( k )
)}$ 
curve segments of degrees bounded by $( k ^{\log ( k )} d )^{O ( k \log ( k) )}$.
While it was known previously, by a result
of D'Acunto and Kurdyka {\cite{Kurdyka}}, that there always exists a path of
length $( O ( d ) )^{k-1}$ connecting two such points, there was no upper
bound on the complexity of such a path.
\end{abstract}

\tableofcontents

\section{Introduction}

Let $\R$ be a fixed real closed field and $\D \subset \R$ an ordered domain.
We will denote by $\C$ the algebraic closure of $\R$. We consider in this
paper the algorithmic problem of, given a polynomial $P \in \D [ X_{1} ,
\ldots ,X_{k} ]$, determining the number of semi-algebraically connected
components of the set, $\ZZ(P, \R^{k})$, of zeros of $P$ in
$\R^{k}$. Moreover, given two points $x,y \in \ZZ ( P, \R^{k})$,
described by real univariate representations (see 
 below for precise definition), we would like to
decide if $x,y$ belong to the same semi-algebraically connected component of
$\ZZ ( P, \R^{k} )$, and if so, to compute a semi-algebraic path
with image contained in $\ZZ ( P, \R^{k} )$, connecting them. We
measure the complexity of an algorithm by the number of arithmetic operations
performed in the ring $\D$.

The problem of designing an efficient algorithm for solving the problem
described in the previous paragraph is very well studied in algorithmic
semi-algebraic geometry. It follows from Collins' algorithm {\cite{Col}} for
computing cylindrical algebraic decomposition {\cite{Col}} that this problem
can be solved with complexity $d^{2^{O ( k )}}$, where $d= \deg ( P )$ \
{\cite{SS}}. Notice that this complexity is doubly exponential in $k$. Singly
exponential algorithms for solving this problem were introduced by Canny in
{\cite{Canny87}}, and successively completed and refined in {\cite{GV90}},
{\cite{GHRSV90}}, {\cite{HRS90}},
{\cite{HRS93}},{\cite{GR92}},{\cite{GV92}},{\cite{BPR99}}, the best complexity
bound being $d^{O ( k^{2} )}$ {\cite{BPR99}}. However, these results remained
unsatisfactory from the complexity point of view for the following reason. It
is a classical result due to Ole$\breve{\text{{\i}}}$nik and
Petrovski$\breve{\text{{\i}}}$ {\cite{OP}}, Thom {\cite{T}} and
Milnor{\cite{Milnor2}} that the number of semi-algebraically connected
components of a real algebraic variety in $\R^{k}$ defined by polynomials of
degree at most $d$ (in fact, the sum of all the Betti numbers of the variety)
is bounded by $d ( 2d-1 )^{k-1} =O ( d )^{k}$. Indeed, the Morse-theoretic
proof of this fact had inspired the so called ``critical point'' method, that
is at the base of many algorithms in semi-algebraic geometry. The best
algorithms using the critical point method often have complexity $d^{O ( k )}$
when applied to real algebraic varieties in $\R^{k}$ defined by polynomials of
degree $d$. It is the case for testing emptiness, computing at least one point
in every connected component, optimizing a polynomial and computing the
Euler-Poincar{\'e} characteristic (see for example, {\cite{BPRbook2}}). In
contrast, for counting the number of semi-algebraically connected components
and computing semi-algebraic paths, the best complexity bound remained $d^{O (
k^{2} )}$.

All known singly exponential algorithms for deciding connectivity of a
semi-algebraic set $S$ rely on computing a certain one dimensional
semi-algebraic subset, which is referred to as a {\tmem{roadmap}} 
of $S$. The definition of a roadmap of an arbitrary semi-algebraic set $S$
(not just a real variety) is as follows.

\begin{definition}
\label{def:roadmap-classical} A roadmap for $S$ is a semi-algebraic set $M$
of dimension at most one contained in $S$ such that $M$ satisfies the
following conditions:
\begin{itemize}
\item $\RM_{1}$ For every semi-algebraically connected component $D$ of
$S$, $D \cap M$ is non-empty and semi-algebraically connected.

\item $\RM_{2}$ For every $x \in \R$ and for every semi-algebraically
connected component $D'$ of $S_{x} =S \cap \pi_{1}^{-1} (
\{ x \} )$, $D' \cap M \neq \emptyset$, where $\pi_{1} : \R^{k}
\rightarrow \R$ is the projection on the first co-ordinate.
\end{itemize}
\end{definition}

Once roadmaps are computed with singly exponential complexity, questions about
connectivity are reduced to the same questions in a finite graph, and can be
answered with complexity no greater than polynomial in the size of the roadmap
itself.

All known algorithms for computing roadmaps follow a certain paradigm which
can be roughly described as follows. Given a semi-algebraic set $V \subset
\R^{k}$ (might be assumed to satisfy certain additional properties, such as
being a bounded, non-singular hypersurface), one defines
\begin{enumerate}
\item a certain semi-algebraic subset $V^{0} \subset V$, with dimension of
$V^{0}$ bounded by $p<k$,
\item a finite subset of points of $\mathcal{N} \subset \R^{p}$.
\end{enumerate}
The set $V^{0}$ and the finite set $\mathcal{N}$ are not arbitrary but must
satisfy certain intricate conditions. A crucial mathematical result is then
proved : for any semi-algebraically connected component $C$ of $V$, $C \cap (
V^{0} \cup V_{\mathcal{N}} )$ is non-empty and semi-algebraically connected,
where $V_{\mathcal{N}} =V \cap \pi_{[ 1,p ]}^{-1} ( \mathcal{N} )$, with
$\pi_{[ 1,p ]} : \R^{k} \rightarrow \R^{p}$ the projection on the first $p$
co-ordinates (see, for example, Proposition 15.7 in {\cite{BPRbook2}} for the
special case when $p=1$, Theorem 14 in {\cite{Mohab-Schost2010}}, \
Proposition 3 in {\cite{BRMS10}}, or Proposition \ref{prop:axiomatic-main} of
the current paper).

The actual algorithm then proceeds by reducing the problem of computing a
roadmap of $V$ to computing roadmaps of $V^{0}$ and of the fibers
$V_{\mathcal{N}}$, each such roadmap containing a well chosen set of points
including the intersection of $V^{0}$ and the fibers $V_{\mathcal{N}}$. The
roadmaps of fibers are then computed using a recursive call to the same
algorithm and the remaining problem is to compute a roadmap of $V^{0}$.

In the {\tmem{classical algorithm}} (see, for example, Chapter 15,
{\cite{BPRbook2}}), $p=1$, and thus $V^{0}$ has dimension at most one, and is
already a roadmap of itself. The complexity of this algorithm for computing
the roadmap of an algebraic set $V \subset \R^{k}$, defined by a polynomial of
degree $d$ in $k$ variables, is $d^{O ( k^{2} )}$ . The exponent $O ( k^{2}
)$ remained a very difficult obstacle to overcome for many years, and the
first progress was reported only very recently.

A fully general deterministic {\tmem{Baby-step Giant-step}} algorithm with
complexity $d^{O ( k^{3/2} )}$ for computing the roadmap of an algebraic set
$V \subset \R^{k}$, defined by a polynomial of degree $d$ in $k$ variables, is
given in {\cite{BRMS10}}. Its recursive scheme is similar to the one
introduced in {\cite{Mohab-Schost2010}} where a probabilistic algorithm of
complexity $d^{O ( k^{3/2} )}$ for computing roadmaps of smooth bounded
hypersurfaces of degree $d$ in $k$ variables is given. In {\cite{BRMS10}},
the parameter $p$ is chosen to be $\approx \sqrt{k}$, the roadmaps of the
fibers are computed recursively using the same algorithm, while that of
$V^{0}$ is computed using the classical algorithm. The main reason for having
such an unbalanced approach, and not using recursion to compute a roadmap of
$V^{0}$ as well, is that the good properties of $V$ under which the
mathematical connectivity result is proved, are not inherited by $V^{0}$. This
difficulty is avoided by making a call to the classical roadmap for $V^{0}$.
The classical roadmap algorithm can be modified so that its complexity is
$d^{O ( p k )}$ for special algebraic sets of dimension at most $p$. Having
an unbalanced approach where the dimension $p$ of $V^{0}$ is much smaller
(roughly $p= \sqrt{k}$) compared to the dimension of the various fibers
(roughly $k - \sqrt{k}$), the complexity of the algorithm in {\cite{BRMS10}}
can be bounded by $d^{O \left( \sqrt{k} k \right)}$.

It is reasonable to hope that a more balanced algorithm in which $p \approx
k/2$, and where the roadmaps of both $V^{0}$ and the $V_{\mathcal{N}}$ are
computed recursively using the same algorithm, by a
{\tmem{divide-and-conquer}} method, can compute a roadmap with a complexity
$d^{\tilde{O} ( k )}$ where we denote by $\tilde{O} ( k )$ any function of $k$
of the form $k \log^{O ( 1 )} ( k )$.

We prove the following theorem which is the main result of this paper
(definitions of real univariate representations are given in Subsection
\ref{prelim}).

\begin{theorem}
\label{thm:main} Let $\R$ be a real closed field and $\D \subset \R$ an
ordered domain. The following holds.
\begin{itemize}
\item There exists an algorithm that takes as input:
\begin{enumerate}
\item a polynomial $P \in \D [ X_{1} , \ldots ,X_{k} ]$, with $\deg ( P
) \leq d$;
\item a finite set, $A$, of real univariate representations
whose associated set of points, $\mathcal{A} = \{ p_{1} , \ldots ,p_{m}
\}$, is contained in $V= \ZZ ( P, \R^{k} )$, and such that
the degree of the real univariate representation representing $p_{i}$ is
bounded by $D_{i}$ for $1 \leq i \leq m$;
\end{enumerate}
and computes a roadmap of $V$ containing $\mathcal{A}$. The complexity of
the algorithm is bounded by 
\[\left( 1+ \sum_{i=1}^{m} D_{i}^{O ( \log^{2}
( k ) )} \right) ( k^{\log ( k )} d )^{O ( k\log^{2} ( k ) )}.
\] 
The
size of the output is bounded by $( \mathrm{card} ( \mathcal{A} ) + 1 ) (
k^{\log ( k )} d )^{O ( k\log ( k ) )}$, while the degrees of the
polynomials appearing in the descriptions of the curve segments and points
in the output are bounded by 
\[
( \max_{1 \leq i \leq m}D_{i} )^{O ( \log
( k ) )} ( k^{\log ( k )} d )^{O ( k\log ( k ) )}.
\]
\item There exists an algorithm that takes as input a polynomial $P \in \D [ X_{1} ,
\ldots ,X_{k} ]$, with $\deg ( P ) \leq d$, and computes
the number of semi-algebraically connected components of $V= \ZZ (
P, \R^{k} )$, with complexity bounded by 
\[( k^{\log ( k )} d )^{O
( k\log^{2} ( k ) )}.
\]

\item There exists an algorithm that takes as input:
\begin{enumerate}
\item a polynomial $P \in \D [ X_{1} , \ldots ,X_{k} ]$, with $\deg ( P
) \leq d$;

\item two real univariate representations whose associated points are
contained in $V= \ZZ ( P, \R^{k} )$, and whose degrees are
bounded by $D_{1}$ and $D_{2}$ respectively;
\end{enumerate}
and decides whether the two points belong to the same semi-algebraically
connected component of $V$, and if so computes a description of a
semi-algebraic path connecting them with image contained in $V$. The
complexity of the algorithm is bounded by 
\[
( D_{1}^{O ( \log^{2} ( k ) )}
+D_{2}^{O ( \log^{2} ( k ) )} +1 ) ( k^{\log ( k )} d )^{O ( k\log^{2}
( k ) )}.
\] 
The size of the output as well as the degrees of the
polynomials appearing in the descriptions of the curve segments and points
in the output are bounded by 
\[\max ( 1 ,D_{1} ,D_{2} )^{O ( \log
( k ) )} ( k^{\log ( k )} d )^{O ( k\log ( k ) )}.
\]
\end{itemize}
\end{theorem}

In fact we prove the following more technical result.

We need the following definition.
\begin{definition}
\label{def:strong-dimension}
A semi-algebraic set $S \subset \R^k$
is {\em strongly of dimension $\le \ell$}
if for every $y\in \R^\ell$,
$S_y=\{x\in S\mid \pi_{[1,\ell]}(x)=y\}$
is finite (possibly empty),
where $\pi_{[1,\ell]}$ denotes the projection to the first $\ell$ coordinates.
(Note that the notion of being strongly of dimension $\leq \ell$ is not invariant under
arbitrary change of coordinates. However, if a semi-algebraic set $S\subset \R^k$ is strongly
of dimension $\leq \ell$, then any semi-algebraic subset of $S$ is strongly of dimension $\leq \ell$.)
\end{definition}

\begin{theorem}
\label{thm:mainbis} Let $\R$ be a real closed field and $\D \subset \R$ an
ordered domain. Then the following holds.
 There exists an algorithm that takes as input:
\begin{enumerate}
\item a polynomial $P \in \D [ X_{1} , \ldots ,X_{k} ]$, with $\deg ( P
) \leq d$ such that $V= \ZZ ( P, \R^{k} )$ is bounded 
and strongly of dimension $\le k'$,
\item a finite set, $A$, of real univariate representations
whose associated set of points, $\mathcal{A} = \{ p_{1} , \ldots ,p_{m}
\}$, is contained in $V$, and such that
the degree of the real univariate representation representing $p_{i}$ is
bounded by $D_{i}$ for $1 \leq i \leq m$;
\end{enumerate}
and computes a roadmap of $V$ containing $\mathcal{A}$. The complexity of
the algorithm is bounded by 
\[
( 1+ \sum_{i=1}^{m} D_{i}^{O ( \log^{2}
( k') )}) ( k^{\log (k')} d )^{O ( k\log^{2} ( k' ) )}.
\] 
The
size of the output is bounded by $( \mathrm{card} ( \mathcal{A} ) + 1 ) (
k^{\log ( k' )} d )^{O ( k\log ( k' ) )}$, while the degrees of the
polynomials appearing in the descriptions of the curve segments and points
in the output are bounded by 
\[
( \max_{1 \leq i \leq m}D_{i} )^{O ( \log
( k') )} ( k^{\log ( k' )} d )^{O ( k\log ( k') )}.
\]
\hide{
\item There exists an algorithm that takes as input a polynomial $P \in \D [ X_{1} ,
\ldots ,X_{k} ]$, with $\deg ( P ) \leq d$,
such that $V= \ZZ ( P, \R^{k} )$ is bounded 
and strongly of dimension $\le k'$, and computes
the number of semi-algebraically connected components of $V= \ZZ \left(
P, \R^{k} \right)$, with complexity bounded by 
\[
( k^{\log ( k' )} d )^{O
( k\log^{2} ( k' ) )}.
\]

\item There exists an algorithm that takes as input:
\begin{enumerate}
\item a polynomial $P \in \D [ X_{1} , \ldots ,X_{k} ]$, with $\deg ( P
) \leq d$;

\item two real univariate representations whose associated points are
contained in $V= \ZZ ( P, \R^{k} )$, and whose degrees are
bounded by $D_{1}$ and $D_{2}$ respectively;
\end{enumerate}
and decides whether the two points belong to the same semi-algebraically
connected component of $V$, and if so computes a description of a
semi-algebraic path connecting them with image contained in $V$. The
complexity of the algorithm is bounded by 
\[
( D_{1}^{O ( \log^{2} ( k' ) )}
+D_{2}^{O ( \log^{2} ( k' ) )} +1 ) ( k^{\log ( k' )} d )^{O ( k\log^{2}
( k' ) )}.
\] 
The size of the output as well as the degrees of the
polynomials appearing in the descriptions of the curve segments and points
in the output are bounded by 
\[\max \left( 1 ,D_{1} ,D_{2} \right)^{O ( \log
( k') )} ( k^{\log ( k')} d)^{O ( k\log ( k') )}.
\]
\end{itemize}
}
\end{theorem}

The bounds on the complexity of the roadmap given in Theorem \ref{thm:main}
give an upper bound on the length of a semi-algebraic curve required to
connect two points in the same connected component of a real algebraic variety
in $\mathbb{R}^{k}$. In {\cite{Kurdyka}}, the authors proved that the
geodesic diameter of any connected component $C$ of a real algebraic variety
in $\mathbb{R}^{k}$ defined by a polynomial of degree $d$ and contained
inside the unit ball in $\mathbb{R}^{k}$, is bounded by $( O ( d )
)^{k-1}$. This result guarantees the existence of a semi-algebraic path
connecting any two points in $C$ of length bounded by $( O ( d ) )^{k-1}$.
Unfortunately, the complexity of this path (namely, the number and degrees of
the polynomials needed to define it) is not uniformly bounded as a function of
$k$ and $d$. We obtain a path of length 
bounded by $( k ^{\log (k )} d )^{O ( k\log( k ) )}$,
but
moreover with uniformly bounded complexity. We have the following theorem.

\begin{theorem}
\label{thm:geodesic} Let $V \subset \mathbb{R}^{k}$ be a real algebraic
variety defined by a polynomial of degree at most $d$, and let $C$ be a
connected component of $V$ contained in the unit ball centered at the
origin. Then, any two points $x,y \in C$, can be connected inside $C$ by a
semi-algebraic path of length at most $( k ^{\log (k )} d )^{O ( k\log
( k ) )}$ 
consisting of at most $( k ^{\log (k )} d )^{O ( k\log ( k )
)}$ 
curve segments of degrees bounded by $( k ^{\log ( k )} d )^{O ( k 
\log ( k) )}$.
\end{theorem}

Note that the algebraic case dealt with in this paper is usually the main
building block in designing roadmap algorithms for more general semi-algebraic
sets (see for example Chapter 16 in {\cite{BPRbook2}}). We believe that with
extra effort, the improvement in the algebraic case reported here could lead
to a corresponding improvement in the general semi-algebraic setting.

We prove Theorem \ref{thm:main} by giving a divide-and-conquer algorithm for
computing a roadmap based on two recursive calls to subvarieties whose
dimensions are at most half the dimension on the given variety $V$ (see
Algorithms \ref{alg:bounded} and \ref{alg:main} in Section \ref{sec:tree}
below).

Such a divide-and-conquer roadmap algorithm would be quite simple if it was
the case that the sub-varieties of $V$ obtained by iterating the following two
operations in any order:
\begin{enumerate}
\item taking the sub-variety consisting of the set of critical points of
$G$, for some polynomial $G \in \D [ X_{1} , \ldots ,X_{k} ]$, restricted to
the fibers, $V_{y} =V \cap \pi^{-1} ( \{ y \} )$, where $\pi$ is a
projection map to a subset of the coordinates (see Definition
\ref{def:critpqg} below for a precise definition of critical points of $G$
restricted to the fibers of $V$);

\item fixing a subset of coordinates (i.e.,  taking fibers of $V$);
\end{enumerate}
had good properties, e.g. the number of critical points of $G$ remains finite
as the parameters vary.

Suppose for simplicity that $k-1$ is a power of $2$. Then, the following
simple algorithm for constructing a roadmap would work, Namely, in the very
first step consider the projection map, $\pi$, to the first $p/2$ coordinates,
where $p=\dim ( V ) =k-1$. For every $y \in \R^{p/2}$, let
$V_{y} =V \cap \pi^{-1} ( \{ y \} )$ be the corresponding fiber and let
$V_{y}^{0} \subset V_{y}$ be the set of critical points of $G$ restricted to
$V_{y}$ and $V^{0} = \cup_{y \in \R^{p/2 }} V_{y}^{0}$. Let $\mathcal{M}
\subset V$ be the set of $G$-critical points of $V$, and $\mathcal{M}^{0}$ the
(assumed finite) $G$-critical points of $V^{0}$. Let $\mathcal{N} =\pi (
\mathcal{M} \cup \mathcal{M}^{0} )$. It can be proved that a roadmap of $V$
can be obtained by taking the union of
\begin{itemize}
\item a roadmap of $V^{0}$ containing $V^{0}_{\mathcal{N}}$,

\item and roadmaps of $V_{y}$, containing the points of $V^{0}$ above
$\mathcal{y}$, for $y \in \mathcal{N}$.
\end{itemize}
Both $V^{0}$ and the $V_{y}$ , $y \in \mathcal{N}$, are of dimension $p/2$. If
$p/2=1$, then the roadmaps of $V^{0}$ and the $V_{y}$, $y \in \mathcal{N}$
coincide with themselves. Otherwise, these roadmaps can then be computed by
recursive calls to the same algorithm.

The description given above, that we are using as a guide, is flawed in a
fundamental way. We know of no way to ensure that all the intermediate
varieties that occur in the course of the algorithm have good properties even
if the original variety $V$ has them.

In order to get around this difficulty we use perturbation techniques, in the
spirit of several other prior work on computing roadmaps. The main difficulty
is to ensure that good properties are preserved for the variety $V^{0}$ as we
go down in the recursion.

In the divide-and-conquer scheme pursued in this paper, it is imperative, for
complexity reasons, that $V^{0}$ and the fibers $V_{y}$ have the same
dimension (namely,  $\tfrac{1}{2} \dim ( V )$). So we cannot resort to the
classical roadmap algorithm for $V^{0}$ any more and we need to ensure good
properties for $V^{0}$ (which is no more an hypersurface even if $V$ is) as
well.

While the general principle -- that of making perturbations to reach an ideal
situation -- is similar to that used in {\cite{BRMS10}} for the Baby-step
Giant-step algorithm for computing roadmaps, there are many new
ideas involved which we list below.

We start the construction with an algebraic hypersurface $V$, defined as the
zero set of one single polynomial $P$.

\begin{enumerate}
\item
We make a deformation $\tilde{P}$ of $P$ using an
infinitesimal, and consider the algebraic set $\tilde{V}$ defined by
$\tilde{P}$ with coefficients in a new field 
$\tilde{\R}$ consisting of algebraic Puiseux series (with coefficients in
$\R$) in this infinitesimal. \

\item
Instead of considering critical points of the projection map on
to a fixed coordinate, we consider critical points of a well chosen fixed
polynomial $G $. This is done to ensure more genericity. Geometrically, we
sweep using the level surfaces of the polynomial $G$.

\item For every $y \in \tilde{\R}^{p/2 }$, let
$\tilde{V}^{0}_{y} \subset \tilde{V}_{y}$ be the set of critical points of $G$
restricted to $\tilde{V}_{y}$ and $\tilde{V}^{0} = \bigcup_{y \in
\tilde{\R}^{p/2 }} \tilde{V}_{y}^{0}$. The closed semi-algebraic set
$\tilde{V}^{0}$ is naturally described as the projection of some variety
involving extra variables. This causes a problem, since we need an explicit
description of $\tilde{V}^{0}$ in order to be able to make a recursive call.
We are able to express $\tilde{V}^{0}$ as the union of several pieces
(charts), each described as a basic constructible set of the form
\[ \bigwedge_{P \in \mathcal{P}} ( P=0 ) \wedge \left( Q \neq 0 \right) .
\]

\item
\label{item:preceding} 
The preceding decomposition of $\tilde{V}^{0}$ into open charts
is not very easy to use, so we modify the description using instead closed
sets (by shrinking slightly the constructible sets). We are able to cover (an
approximation of) $\tilde{V}^{0}$ by basic semi-algebraic sets of the form
\[ \bigwedge_{P \in \mathcal{P}} ( P=0 ) \wedge ( Q \geq 0 ) . \]

\item
This necessitates that in our recursive calls we accept as
inputs not just varieties, but basic semi-algebraic sets of a certain special
form having only a few inequalities in their definitions.

\item
The Morse-theoretical connectivity results needed to prove the
correctness of the new algorithms have to be extended to take into account the
two new features mentioned above. The first new feature is that instead of
considering projection map to a fixed coordinate, we are using the polynomial
$G$ as the ``Morse function''. Secondly, instead of varieties we need to deal
with more general semi-algebraic sets. We define a new variant of the notion
of ``pseudo-critical values'' introduced in {\cite{BPRbook2}} which is
applicable to the semi-algebraic case and which takes into account the
polynomial $G$, and prove the required Morse theoretical lemmas in this new
setting.

\item
 The covering mentioned in \eqref{item:preceding} above means that we are replacing
each semi-algebraic set, by several basic semi-algebraic sets, the union of
whose limits co-incides with the given set. In order that the union of the
limits of the roadmaps computed for each of the new sets gives a roadmap of
the original one, we need to make sure that the roadmaps of the new sets
contain certain carefully chosen points. Very roughly speaking these points
will correspond to a finite number of pairs of closest points realizing the
locally minimal distance between any two semi-algebraically connected
components of the new sets.

\item
 The construction involves a perturbation using four
infinitesimals at each level of the recursion. Since, there will be at most $O
( \log ( k' ) )$ levels, at the end we will be doing computations in a ring
with $O ( \log ( k' ) )$ infinitesimals. At the end of the algorithm we will
need to compute descriptions of the limits of the semi-algebraic curves
computed in the previous steps of the algorithm. We show that these limits can
be computed within the claimed complexity bound. For this the fact that we
have only $O ( \log ( k' ) )$ infinitesimals, and not more, is crucial.

\end{enumerate}

The rest of the paper is organized as follows. In Section
\ref{sec:critical}, we state some basic results of Morse theory for higher
co-dimensional non-singular varieties, including definitions of critical
points on basic semi-algebraic sets and their properties.

In Section \ref{sec:connectivity}, we prove the connectivity results that we
will require. We introduce a set of axioms (to be satisfied by a basic
semi-algebraic set $S$ and certain subsets of $S$) and prove an abstract
connectivity result (Proposition \ref{prop:axiomatic-main}) which forms the
basis of the roadmap algorithm in this paper. The main differences between
Proposition \ref{prop:axiomatic-main} and a similar result in {\cite[Proposition 3]{BRMS10}}
are that Proposition \ref{prop:axiomatic-main} applies to
basic semi-algebraic sets (not just to algebraic hypersurfaces), and that
there is an auxiliary polynomial $G$ which plays the role of the
$X_{1}$-co-ordinate in {\cite{BRMS10}}.

In Section \ref{sec:deformation}, we discuss certain specific infinitesimal
deformations that we will use in order to ensure that the properties defined
in Section \ref{sec:axiomatics} hold. In Section
 \ref{sec:grp}, we
explain a deformation technique to reach general position and prove that the
set of $G$-critical points is finite for a certain well chosen polynomial $G$.
The techniques used in this
section are adapted from {\cite{JPT2012}}. In Section
\ref{sec:G-special-values}, we define a new notion of pseudo-critical values
for semi-algebraic sets with respect to a given polynomial $G$ and state their
connectivity properties, generalizing to this new context results from
{\cite{BPRbook2}}. In Section \ref{sec:deformation}, we discuss how the
deformations are used to ensure the connectivity properties defined in Section
\ref{sec:axiomatics}.

Section \ref{sec:minors} is devoted to a description of the set of
$G$-critical points using minors of certain Jacobian matrices and the
properties of the set of $G$-critical points.

Section \ref{sec:tree} is devoted to the description of the Divide and Conquer
Roadmap Algorithm. We first define the tree that is computed, explain how it
gives a roadmap, and finally describe the Divide and Conquer Algorithm first
for the bounded case (Algorithm \ref{alg:bounded_refined}), and then in general
(Algorithm \ref{alg:main}).

In the Annex (Section \ref{sec:aux-proofs}), we include certain technical
proofs of propositions on critical and pseudo-critical values stated in
Section \ref{subsec:criticalsa} and Section \ref{sec:G-special-values} and
used in the paper.

\section{Critical points of algebraic and basic semi-algebraic
sets}\label{sec:critical}

In this section we define critical points of a polynomial first on an
algebraic set and then on a basic semi-algebraic set and discuss their
properties.

\subsection{Critical points of algebraic sets}\label{subsec:critical}

\begin{definition}
\label{def:G-critical-point}Let $G \in \R [ X_{1} , \ldots ,X_{k} ]$ and
$\mathcal{P} =\{ P_{1} , \ldots ,P_{m} \} \subset \R [ X_{1} , \ldots
,X_{k} ]$ be a finite family of polynomials.

We say that $x \in \ZZ ( \mathcal{P}, \R^{k} )$ is a
{\tmem{$G$-critical point of $\ZZ ( \mathcal{P}, \R^{k} )$}}, if
there exists $\lambda = ( \lambda_{0} , \cdots , \lambda_{m} ) \in
\R^{m+1}$ satisfying the system of equations $\mathrm{CritEq} ( \mathcal{P} ,G
)$
\begin{eqnarray}
\label{eqn:prop-critical}
 P_{j} & = & 0, j=1, \ldots ,m, \nonumber\\
\sum^{m}_{j=1} \lambda_{j} \frac{\partial P_{j}}{\partial X_{i}} -
\lambda_{0} \frac{\partial G}{\partial X_{i}} & = & 0,
i=1, \ldots ,k,\label{eqn:criteqnpg}\\
\sum_{j=0}^{m} \lambda_{j}^{2} -1 & = & 0. \nonumber
\end{eqnarray}
The set $\mathrm{Crit} ( \mathcal{P} ,G ) \subset \R^{k}$ is the set of
$G$-critical points of $\ZZ ( \mathcal{P}, \R^{k} )$, i.e., the
projection on $\R^{k}$ of $\mathrm{Zer} \left( \mathrm{CritEq} ( \mathcal{P} ,G
) , \R^{k+m+1} \right)$. Note that geometrically, in the case
the polynomials $\mathcal{P}$ define a non-singular complete intersection,
$\mathrm{Crit} ( \mathcal{P} ,G )$ is the set of points $x \in \ZZ (
\mathcal{P}, \R^{k} )$, such that the tangent space at $x$ of $\ZZ
( \mathcal{P}, \R^{k} )$ is orthogonal to $\mathrm{grad} ( G ) ( x
)$. In case $\ZZ ( \mathcal{P}, \R^{k} )$ in singular, then the
set of $G$-critical points includes the set of singular points of $\ZZ
( \mathcal{P}, \R^{k} )$, which is clear from
\eqref{eqn:criteqnpg}.
\end{definition}

\subsection{Critical points of basic semi-algebraic
sets}\label{subsec:criticalsa}

\begin{notation}
\label{not:S-PQ}Given two finite families of polynomials
$\mathcal{P},\mathcal{Q}\subset \R [ X_{1} , \ldots ,X_{k} ]$, we denote
by $\mathrm{Bas} ( \mathcal{P} , \mathcal{Q} )$ the basic semi-algebraic
set defined by
\[ \mathrm{Bas} ( \mathcal{P} , \mathcal{Q} )=\left\{ x\in \R^{k} \mid
 \bigwedge_{P \in \mathcal{P}}P ( x ) =0 \wedge \bigwedge_{Q
 \in \mathcal{Q}} Q ( x ) \geq 0 \right\} . \]
\end{notation}

\begin{definition}
\label{def:critpqg} Let $G \in \R [ X_{1} , \ldots ,X_{k} ]$. We define
$\mathrm{Crit} ( \mathcal{P} , \mathcal{Q} ,G )$, the set of
\emph{$G$-critical points of $\mathrm{Bas} ( \mathcal{P} ,
\mathcal{Q} )$}, by
\[ \mathrm{Crit} ( \mathcal{P} , \mathcal{Q} ,G ) = \mathrm{Bas} ( \mathcal{P}
 , \mathcal{Q} ) \bigcap \left( \bigcup_{\mathcal{Q' \subset Q
 }} \mathrm{Crit} ( \mathcal{P} \cup \mathcal{Q}' ,G ) \right) . \]
\end{definition}

\begin{definition}
\label{def:general-position}
We say that the pair
 $ \mathcal{P},\mathcal{Q}$ is {\tmem{in general position with
respect to $G \in \R [ X_{1} , \ldots ,X_{k} ]$}} if $\ZZ
( \mathcal{P}, \R^{k} )$ is bounded, and for any subset
$\mathcal{Q}' \subset \mathcal{Q,}$ $\mathrm{Crit} ( \mathcal{P \cup Q'} ,G
) \subset \R^{k}$ is empty or finite.
\end{definition}

\begin{remark}
\label{rem:general-position}
Note that in this case $\ZZ ( \mathcal{P} , \R^{k})$ has only a finite number of singular points; moreover if $
\mathrm{card}(\mathcal{P})=k$, $\ZZ ( \mathcal{P} , \R^{k})$ is finite (possibly empty).
\end{remark}

The properties of $G$-critical points used later in the paper are now given in
the following two Morse-theoretic lemmas. The proofs, which
are slight variants of the classical proofs, are included in the Annex (Section
\ref{auxproof}).

\begin{notation}
Let $T \subset \R^{k}$, $G$ a function $\R^{k} \longrightarrow \R$, and
suppose that $a \in \R$. We denote
\begin{eqnarray*}
T_{G=a} & = & \{ x \in T \mid G ( x ) =a \} ,\\
T_{G \leq a} & = & \{ x \in T \mid G ( x ) \leq a \} ,\\
T_{G<a} & = & \{ x \in T \mid G ( x ) <a \} .
\end{eqnarray*}
\end{notation}

Let $\mathcal{P},\mathcal{Q}\subset \R [ X_{1} , \ldots ,X_{k} ]$,
$S = \mathrm{Bas} ( \mathcal{P} , \mathcal{Q} )$, $S$ bounded, and
$\mathcal{M} = \mathrm{Crit} ( \mathcal{P} , \mathcal{Q} ,G )$.

\begin{lemma}
 \label{lemmafora} Suppose that $b \nin \mathcal{D} =G
( \mathcal{M} )$. Let $C$ be a semi-algebraically connected component of
$S_{G \leq b}$. If $a<b$ and $(a,b] \cap \mathcal{D}$ is empty, then
$C_{G \le a}$ is semi-algebraically connected.
\end{lemma}

Now assume that $\mathcal{P},\mathcal{Q}$ are in general position with respect
to $G $ (cf. Definition \ref{def:general-position}).

\begin{lemma}
\label{lemmaforb}Let $C$ be a semi-algebraically connected component of
$S_{G \le b}$, such that $C_{G=b}$ is not empty.
\begin{enumerate}
\item \label{item:lemmaforb1}
If $\dim (C) =0$, $C$ is a point contained in $\mathcal{M}$.

\item \label{item:lemmaforb2} If $\dim ( C ) \neq 0$, then $C_{G<b}$ is non-empty. Let $B_{1} ,
\ldots ,B_{r}$ be the semi-algebraically connected components of
$C_{G<b}$. Then,
\begin{enumerate}
\item \label{item:lemmaforb2a} for each $i,1 \leq i \leq r$, $\overline{B_{i}} \cap
\mathcal{M} \neq \emptyset$;

\item \label{item:lemmaforb2b} if there exist $i,j,1 \leq i<j \leq r$ such that $\overline{B_{i}}
\cap \overline{B_{j}} \neq \emptyset$, then $\overline{B_{i}} \cap
\overline{B_{j}} \subset \mathcal{M}$;

\item \label{item:lemmaforb2c} $\cup_{i=1}^{r} \overline{B_{i}} =C$, and hence $\cup_{i=1}^{r}
\overline{B_{i}}$ is semi-algebraically connected.
\end{enumerate}
\end{enumerate}
\end{lemma}

\section{Axiomatics for connectivity
}\label{sec:connectivity}\label{sec:axiomatics}

In this subsection we identify a set of properties, to be satisfied by a basic
semi-algebraic set $\mathrm{Bas} ( \mathcal{P} , \mathcal{Q} )$, a polynomial
$G$, and certain finite subsets of points contained in $\mathrm{Bas} (
\mathcal{P} , \mathcal{Q} )$, and prove a key connectivity result (Proposition
\ref{prop:axiomatic-main} below) for such a situation, which plays a key role
in our recursive algorithm later. In Section \ref{sec:deformation} we will
explain how to use a perturbation technique to reach the ideal situation
described here.

\begin{notation}
Let $\pi_{[ 1, \ell ]}$ be the projection map from $\R^{k}$ to $\R^{\ell}$
forgetting the last $k- \ell$ coordinates. For every $T \subset \R^{k}$
and $A \subset \R^{\ell}$, we denote $T_{A} = T \cap \pi_{[ 1, \ell
]}^{-1} ( A )$. 
For $ w \in \R^{\ell}$, we denote 
$T_{w} = T \cap \pi_{[ 1, \ell ]}^{-1} ( \{ w \} )$.
\end{notation}

\begin{definition}
\label{def:property-special} 
Let $1\leq \ell <k$, $G\in \R [ X_{1} ,
\ldots ,X_{k} ]$, and let $\mathcal{P},\mathcal{Q} \subset \R [ X_{1} ,
\ldots ,X_{k} ]$ be in general position with respect to $G$. Let $S = 
\mathrm{Bas} ( \mathcal{P} , \mathcal{Q} )$, and suppose that $S$ is bounded.

 We say that a tuple $ ( S, \mathcal{M,} \ell ,S^{0} ,\mathcal{D}^{0}
,\mathcal{M}^{0} )$ is {\tmem{special}} if it satisfies the following Properties 
 \ref{item:special0}, \ref{item:special1}, \ref{item:special2}, and
\ref{item:special3}. 
\begin{enumerate}
 \item \label{item:special0} $\mathcal{M} = \mathrm{Crit} ( \mathcal{P} , \mathcal{Q} ,G )$ is the
finite set of critical points of $G$ on $S$.
\item
\label{item:special1} $S^{0} \subset S$ is a semi-algebraic set 
strongly of dimension $\leq \ell$
such that for
every $w\in \R^{\ell}$, 
$S^{0}_{w}$
meets every semi-algebraically connected component of $S_{w}$,
and for each semi-algebraically connected component $C$ of $S_{w}$,
$S^{0}_{w}$ contains a minimizer of $G$ over $C$.

\item
\label{item:special2} $\mathcal{D}^{0} \subset \R$ is a finite set of values
satisfying for every interval $[ a,b ] \subset \R$ and $c\in [ a,b ]$,
with $\{ c \} \supset \mathcal{D}^{0} \cap [ a,b ]$, if $D$ is a
semi-algebraically connected component of 
$S^{0}_{a \leq G \leq b}$, then $D_{G=c}$ is a semi-algebraically connected component
of $S^{0}_{G=c}$.

\item
\label{item:special3} $\mathcal{M}^{0} \subset S^{0}$ is a finite set of points
satisfying the following properties:
\begin{enumerate}
\item
\label{item:special3a}
$G ( \mathcal{M}^{0} )$=$\mathcal{D}^{0}$,

\item 
\label{item:special3b}
$\mathcal{M}^{0}$ meets every semi-algebraically connected
component of $S^{0}_{G=a}$ for all $a\in \mathcal{D}^{0}$.
\end{enumerate}
\end{enumerate}

\end{definition}

\begin{definition}
\label{def:connected-component}For a semi-algebraic subset $S \subset T$, we
say that $S$ has {\tmem{good connectivity property}} with respect to $T$, if
the intersection of $S$ with every semi-algebraically connected component of
$T$ is non-empty and semi-algebraically connected.
\end{definition}

With the definition introduced above we have the following key result which
generalizes Proposition 3 in {\cite{BRMS10}} (see also Theorem 14 in
{\cite{Mohab-Schost2010}}).

\begin{proposition}
\label{prop:axiomatic-main}Let $ ( S, \mathcal{M, \ell ,} S^{0}
,\mathcal{D}^{0} ,\mathcal{M}^{0} )$ be a special tuple. Then, for every
finite 
$\mathcal{N }\supset \pi_{[ 1, \ell ]} ( \mathcal{M} \cup \mathcal{M}^{0} )$,
the
semi-algebraic set $S^{0} \cup S_{\mathcal{N}}$ has good connectivity
property with respect to $S$.
\end{proposition}

In the proof of Proposition \ref{prop:axiomatic-main} we will use the
following notation.

\begin{notation}
\label{not:ccofx}If $S \subset \R^{k}$ is semi-algebraic set and $x \in S$,
then we denote by $\Cc(x,S)$ the semi-algebraically connected
component of $S$ containing $x$.
\end{notation}

 \begin{notation}
\label{not:puiseux}Given a real closed field $\R$ and a variable $\eps$, we
denote by $\R \la \eps \ra$ the real closed field of
algebraic Puiseux series (see {\cite{BPRbook2}}). In the ordered field $\R
\la \eps \ra$, $\eps$ is positive and infinitesimal, i.e.,
smaller than any positive element of $\R$. We denote by $\lim_{\eps}$ the
mapping which sends a bounded Puiseux series to its constant term.
\end{notation}

\begin{notation}
If $\R'$ is a real closed extension of a real closed field $\R$, and $S
\subset \R^{k}$ is a semi-algebraic set defined by a first-order formula
with coefficients in $\R$, then we will denote by $\Ext \left( S, \R'
\right) \subset \R'^{k}$ the semi-algebraic subset of $\R'^{k}$ defined by
the same formula. It is well-known that $\Ext \left( S, \R' \right)$ does
not depend on the choice of the formula defining $S$ (see {\cite{BPRbook2}}
for example).
\end{notation}

\begin{proof}[Proof of Proposition \ref{prop:axiomatic-main}] Let $S^{1}
=S_{\pi_{[ 1, \ell ]} ( \mathcal{M} \cup \mathcal{M}_{0} )}$. We are going
to prove that $S^{0} \cup S^{1}$ has good connectivity property with
respect to $S$, which implies the proposition.

For $a$ in $\R$, we say that property $\mathrm{GCP} (a)$ holds if
$ ( S^{0} \cup S^{1} )_{G \leq a}$ has good connectivity property with
respect to $S$.

We prove that for all $a$ in $\R$, $\mathrm{GCP} (a)$ holds. Since
$S$ is assumed to be bounded, the proposition follows immediately from
this claim, since it is clear that the proposition follows from $\mathrm{GCP}
( a )$ for any $a \ge \max_{x \in S}G (x)$.

The proof uses two intermediate results:

\begin{description}
\item[Step 1 \label{item:axiomatic-main1}]
Let $\mathcal{D} = G(\mathcal{M})$.
 For every $a \in \mathcal{D} \cup
\mathcal{D}^{0}$, and for every $b \in \R$ with 
$(a,b] \cap (\mathcal{D}
\cup \mathcal{D}^{0} ) = \emptyset$, $ \mathrm{GCP} (a)$ implies $\mathrm{GCP}
( b )$.

 \item[Step 2 \label{item:axiomatic-main2}]
 For every $b \in \mathcal{D} \cup \mathcal{D}^{0}$, if
$ \mathrm{GCP} (a)$ holds for all $a<b$, then $
\mathrm{GCP} (b)$ holds.
\end{description}

The combination of \ref{item:axiomatic-main1} and \ref{item:axiomatic-main2} implies by an easy induction that the property 
$\mathrm{GCP} (a)$ holds for all $a$ in $\R$, 
since for $a< \min_{x \in S} ( G ( x) )$, the property $ \mathrm{GCP} (a)$ holds vacuously. So the
proposition follows from \ref{item:axiomatic-main1} and \ref{item:axiomatic-main2}.

We now prove the two steps.

\begin{itemize}
\item[\ref{item:axiomatic-main1}]
We suppose that $a \in \mathcal{D} \cup
\mathcal{D}^{0}$ and $ \mathrm{GCP} (a)$ holds, take $b \in \R$, $a<b$
with $(a,b] \cap (\mathcal{D} \cup \mathcal{D}^{0}) = \emptyset$, and prove that
$\mathrm{GCP} (b)$ holds. Let $C$ be a semi-algebraically
connected component of $S_{G \leq b}$. We have to prove that $C \cap ( S^{0}
\cup S^{1} )$ is semi-algebraically connected.

Since $(a,b] \cap (\mathcal{D} \cup \mathcal{D}^{0}) = \emptyset$, it follows that
$\mathcal{M}_{a< G \leq b} = \emptyset$, and it follows from  Lemma
\ref{lemmafora} that $C_{G \le a}$ is a semi-algebraically connected
component of $S_{G \le a}$. Now, using property $\mathrm{GCP}
(a)$, we see that $C_{G \le a} \cap ( S^{0} \cup S^{1} )$ is non-empty and
semi-algebraically connected.

Let $x \in C \cap ( S^{0} \cup S^{1} )$. We prove that $x$ can be
semi-algebraically connected to a point in $C_{G \le a} \cap S^{0}$ by a
semi-algebraic path in $C \cap ( S^{0} \cup S^{1} )$, which is enough to
prove that $C \cap ( S^{0} \cup S^{1} )$ is semi-algebraically connected.

There are three cases to consider.

Case 1: $x \in S^{1}$. In this case, consider $\Cc ( x,S_{\pi_{[ 1,
\ell ]} ( x )} )= \Cc ( x,S^{1}_{\pi_{[ 1, \ell ]} ( x )} )$. Then,
by Definition \ref{def:property-special}, Part (\ref{item:special1}), there exists $x' \in
\Cc ( x,S_{\pi_{[ 1,q ]} ( x )} ) \cap S^{0}$ such that $x'$ is a
minimizer of $G$ over 
$\Cc ( x,S_{\pi_{[ 1, \ell ]} ( x )} )$ i.e., 
\[
G( x' ) = \min_{x'' \in \Cc( x,S_{\pi_{[ 1, \ell ]} ( x )} )} G ( x'').
\] 
In particular, $x' \in \Cc( x,S^{0}_{G
\leq b} ) \subset C$. Connecting $x$ to $x'$ by a semi-algebraic path inside
$\Cc ( x,S^{1}_{\pi_{[ 1, \ell ]} ( x )} )$ we reduce either to Case 2
or Case 3 below.

Case 2: $x \in S^{0} $, $G ( x ) \leq a$. In this
case there is nothing to prove.

Case 3: $x \in S^{0}$, $G ( x )>a$. By Definition
\ref{def:property-special}, Part (\ref{item:special2}) applied to $\Cc
(x,S^{0}_{a \leq G \leq b} )$ we have that $a \in G ( \Cc
(x,S^{0}_{a \leq G \leq b} ))$ and $\Cc(x,S^{0}_{a \leq G
\leq b} )_{G=a}$ is non-empty. Hence,  there exists a semi-algebraic path
connecting $x$ to a point in $\Cc (x;S^{0}_{a \leq G \leq b}
)_{G=a}$ inside $\Cc (x,S^{0}_{a \leq G \leq b} )$. Since
$\Cc (x,S^{0}_{a \leq G \leq b} ) \subset S^{0}$ and
$\Cc (x,S^{0}_{a \leq G \leq b} ) \subset C$, it follows
that $\Cc (x,S^{0}_{a \leq G \leq b} ) \subset C \cap S^{0}$
and we are done.

This finishes the proof of \ref{item:axiomatic-main1}.
\item[\ref{item:axiomatic-main2}]
We suppose that $b \in \mathcal{D} \cup \mathcal{D}^{0}$, and $ \mathrm{GCP} (a)$ holds for all $a<b$, and prove
that $ \mathrm{GCP} (b)$ holds.

Let $C$ be a semi-algebraically connected component of $S_{G \le b}$.

If $\dim (C) =0$, $C$ is a point belonging to $\mathcal{M} \subset (
S^{0} \cup S^{1} )$ by Lemma \ref{lemmaforb}. So $C \cap ( S^{0} \cup S^{1}
)$ is semi-algebraically connected.

Hence, we can assume that $\dim (C) >0$. If $C_{G=b} = \emptyset$ there is
nothing to prove. Suppose that $C_{G=b}$ is non-empty, so that $C_{G<b}$ is
non-empty by Lemma \ref{lemmaforb}.

Our aim is to prove that $C \cap ( S^{0} \cup S^{1} )$ is semi-algebraically
connected. We do this in two steps. We prove the following statements:
\begin{description}
\item[(a) \label{item:a}]
if $B$ is a semi-algebraically connected component of $C_{G<b}$,
then $\overline{B} \cap ( S^{0} \cup S^{1} )$ is non-empty and
semi-algebraically connected, and
\item
[(b) \label{item:b}]
using 
\ref{item:a}
$C \cap ( S^{0} \cup S^{1} )$ is semi-algebraically
connected.
\end{description}
{\noindent}{\bf Proof of \ref{item:a}.} We prove that if $B$ is a
semi-algebraically connected component of $V_{G<b}$, then $\overline{B} \cap
( S^{0} \cup S^{1} )$ is non-empty and semi-algebraically connected.

Since $\overline{B}$ contains a point of $\mathcal{M}$ it follows that
$\overline{B} \cap ( S^{0} \cup S^{1} )$ is not empty.

Note that if $\overline{B} \cap ( S^{0} \cup S^{1} ) =B \cap ( S^{0} \cup
S^{1} )$, then there exists $a$ with
\[ \max (\{G(x) \mid x \in B \cap ( S^{0} \cup S^{1} ) \}) <a<b, \]
such that $B \cap ( S^{0} \cup S^{1} ) = (B \cap ( S^{0} \cup S^{1} ) )_{G
\le a}$,  and using Lemma \ref{lemmafora}, $B_{G \le a}$ is
semi-algebraically connected. So $B \cap ( S^{0} \cup S^{1} )$ is
semi-algebraically connected since $\mathrm{GCP} (a)$ holds.

We now suppose that $( \overline{B} \setminus B) \cap ( S^{0} \cup S^{1} )$
is non-empty. Taking $x \in ( \overline{B} \setminus B) \cap ( S^{0} \cup
S^{1} )$, we are going to show that $x$ can be connected to a point $z$ in
$B \cap S^{0}$ by a semi-algebraic path $\gamma$ inside $\overline{B} \cap (
S^{0} \cup S^{1} )$. Notice that $G (x) =b$.

We first prove that we can assume without loss of generality that $x \in
S_{0}$. Otherwise, since $x \in S^{0} \cup S^{1}$, we must have that $x \in
S_{w}$ with $w= \pi_{[1, \ell ]} (x)$, and $S_{w} \subset S^{ 1}$. Let $A=
\Cc (x,S_{w} \cap \overline{B} )$. We now prove that $A \cap
S^{0}_{w} \neq \emptyset$. Using the curve section lemma, choose a
semi-algebraic path $\gamma : [0, \eps ] \rightarrow \Ext ( \overline{B} ,
\R \left \la \eps \ra \right )$ such that $\gamma (0) =x$, $\lim_{\eps}
\gamma ( \eps ) =x$ and $\gamma ((0, \eps ]) \subset \Ext (B, \R \left \la
\eps \ra \right )$. Let $w_{\eps} = \pi_{[1, \ell ]} ( \gamma ( \eps ))$
and
\[ A_{\eps} = \Cc ( \gamma ( \eps ), \Ext (B, \R \left \la
 \eps \ra \right )_{w_{\eps}} ) . \]
Note that $x \in \lim_{\eps}A_{\eps} \subset A$.

By the Tarski-Seidenberg transfer principle {\cite{BPRbook2}}, $\Ext (B, \R
\left \la \eps \ra \right )$ is a semi-algebraically connected component
of $ \Ext (S_{G<a} , \R \left \la \eps \ra \right )$ which
implies that $A_{\eps}$ is a semi-algebraically connected component of $\Ext
(S, \R \left \la \eps \ra \right )_{w_{\eps}}$. By Definition
\ref{def:property-special}, Part (\ref{item:special1}), and the Tarski-Seidenberg transfer
principle, 
\[
\Ext (S^{0} , \R \left \la \eps \ra \right )_{w_{\eps}} \cap
A_{\eps} \neq \emptyset.
\]
 Then, since $\Ext (S^{0} , \R \left \la \eps \ra
\right )_{w_{\eps}} \cap A_{\eps}$ is bounded over $\R$, 
\[\lim_{\eps} (
\Ext (S^{0} , \R \left \la \eps \ra \right )_{w_{\eps}} \cap A_{\eps} )
\]
is a non-empty subset of $S^{0}_{w} \cap A$.

Now connect $x$ to a point in $x' \in S^{0}_{w}$ by a semi-algebraic path
whose image is contained in $A \subset \overline{B}_{w} \subset (
\overline{B} \setminus B) \cap ( S^{0} \cup S^{1} )$ such that $x'$ is a
minimizer of $G$ on $A$. If $G ( x' ) <b$, take $z=x'$. Otherwise, replacing
$x$ by $x'$ if necessary we can assume that $x \in S^{0}$ as announced.

There are four cases -- namely, 
\begin{enumerate}
\item
\label{item:axiomatic-main-case1}
  $x \in \mathcal{M} \cup \mathcal{M}^{0}$;

\item 
\label{item:axiomatic-main-case2}
$x \nin \mathcal{M} \cup \mathcal{M}^{0}$ and
$\Cc (x,S^{0}_{G=b} )
\not\subset
\overline{B}$;

\item 
\label{item:axiomatic-main-case3}
$x \nin \mathcal{M} \cup \mathcal{M}^{0}$, $\Cc
(x,S^{0}_{G=b} ) \subset \overline{B}$ and $b \in \mathcal{D}^{0}$;

\item 
\label{item:axiomatic-main-case4}
$x \nin \mathcal{M} \cup \mathcal{M}^{0}$, $\Cc
(x,S^{0}_{G=b} ) \subset \overline{B}$ and $b \nin \mathcal{D}^{0}$;
\end{enumerate}
that we consider now.
\begin{enumerate}
\item  $x \in \mathcal{M} \cup \mathcal{M}^{0}$:\\
Define $w= \pi_{[1, \ell ]} (x)$, and note that 
$S_{w} \subset ( S^{0}\cup S^{1} )$. 
Since $x \in \overline{B}$, and $B$ is bounded, $w \in
\pi_{[1, \ell ]} ( \overline{B} ) = \overline{\pi_{[1, \ell ]} (B)}$. Now
let $\eps >0$ be an infinitesimal. By applying the curve selection lemma
to the set $B$ and $x \in \overline{B}$, we obtain that there exists
$x_{\eps} \in \Ext \left( B, \R \langle \eps \rangle \right)$ with
$\lim_{\eps} x_{\eps} =x$, $G (x_{\eps} ) <G (x)$ and $x \in \lim_{\eps}
\Ext (S, \R \langle \eps \rangle )_{w_{\eps}}$, where $w_{\eps} = \pi_{[1,
\ell ]} \left( x_{\eps} \right)$. By Definition
\ref{def:property-special}, Part (\ref{item:special1}), and the Tarski-Seidenberg transfer
principle, we have that $\Ext (S^{0} , \R \langle \eps \rangle
)_{w_{\eps}}$ is non-empty, and contains a minimizer of $G$ over
$\Cc \left( x_{\eps} , \Ext \left( S, \R \la \eps \ra
\right)_{w_{\eps}} \right)$. Let
\[ x'_{\eps} \in \Ext (S^{0} , \R \langle \eps \rangle )_{w_{\eps}} \cap
 \Cc(x_{\eps} , \Ext (B, \R \left \la \eps \ra \right
 )_{w_{\eps}} ) \]
be such a minimizer and let $x' = \lim_{\eps}x'_{\eps}$. Notice that $G
\left( x_{\eps} \right)< G ( x )$. Since $\lim_{\eps}x_{\eps} =x$ and
$\lim_{\eps} \Cc (x_{\eps} , \Ext (B, \R \left \la \eps \ra \right
)_{w_{\eps}} )$ is semi-algebraically connected,
\[ \lim_{\eps} \Cc (x_{\eps} , \Ext (B, \R \left \la \eps
 \ra \right )_{w_{\eps}} ) \subset \Cc (x,
 \overline{B}_{w} ) . \]
Now choose a semi-algebraic path $\gamma_{1}$ connecting $x$ to $x'$
inside $\Cc (x, \overline{B}_{w} )$ (and hence inside
$S^{0} \cup S^{1}$ since $\Cc (x, \overline{B}_{w} )
\subset S_{w} \subset S^{0} \cup S^{1}$), and a semi-algebraic path
$\gamma_{2} ( \eps )$ joining $x'$ to $x'_{\eps}$ inside $\Ext (S^{0} , \R
\langle \eps \rangle )$. The concatenation of $\gamma_{1} , \gamma_{2} (
\eps )$ gives a semi-algebraic path $\gamma$ having the required property,
after replacing $\eps$ in $\gamma_{2} ( \eps )$ by a small enough positive
element of $t \in \R$. Now take $z =\gamma_{2} ( t )$.

\item $x \nin \mathcal{M} \cup \mathcal{M}^{0}$ and
$\Cc (x,S^{0}_{G=b} )
\not\subset
\overline{B}$:\\
There exists $x' \in \Cc (x,S^{0}_{G=b} )$, $x' \nin
\overline{B}$ and a semi-algebraic path $\gamma : [0,1] \rightarrow
\Cc (x,S^{0}_{G=b} )$, with $\gamma (0) =x, \gamma (1) =x'$. Since
$x' \nin \overline{B}$, it follows from Lemma \ref{lemmaforb} (\ref{item:lemmaforb2}) that for
$t_{1} = \max \{0 \leq t<1 \hspace{0.75em} \mid \hspace{0.75em} \gamma
(t) \in \overline{B} \}$, $\gamma (t_{1} ) \in \mathcal{M}$. We can now
connect $x$ to a point in $z\in B \cap S^{0}$ by a semi-algebraic path
inside $\overline{B} \cap ( S^{0} \cup S^{1} )$ using what has been
already proved in Case (\ref{item:axiomatic-main-case1}) above.

\item $x \nin \mathcal{M} \cup \mathcal{M}^{0}$, $\Cc
(x,S^{0}_{G=b} ) \subset \overline{B}$ and $b \in \mathcal{D}^{0}$:\\
Since $b \in \mathcal{D}^{0}$, by Definition \ref{def:property-special},
Part 
(\ref{item:special3b})
there exists $x' \in \Cc (x,S^{0}_{G=b} ) \cap
\mathcal{M}^{0}$. Thus, there exists a semi-algebraic path connecting
$x$ to $x' \in \mathcal{M}_{0}$ with image contained in $\overline{B} \cap
( S^{0} \cup S^{1} )$. We can now connect $x'$ to a point in $z \in B \cap
S^{0}$ by a semi-algebraic path inside $\overline{B} \cap ( S^{0} \cup
S^{1} )$ using what has been already proved in Case 
(\ref{item:axiomatic-main-case1}) above.

\item $x \nin \mathcal{M} \cup \mathcal{M}^{0}$, $\Cc
(x,S^{0}_{G=b} ) \subset \overline{B}$ and $b \nin \mathcal{D}^{0}$:\\
Since $b \nin \mathcal{D}_{0}$, for all $a<b$ such that $[a,b] \cap
\mathcal{D}^{0} = \emptyset$, $\Cc (x,S^{0}_{a \leq G \leq
b} )_{G=b} = \Cc (x,S^{0}_{G=b} )$ and
$\Cc (x,S^{0}_{a \leq G \leq b} )_{G=a} \neq \emptyset$ by
Definition \ref{def:property-special}, Part (\ref{item:special2}). Let $x' \in
\Cc (x,S^{0}_{a \leq G \leq b} )_{G=a}$. We can choose a
semi-algebraic path $\gamma : [0,1] \rightarrow \Cc
(x,S^{0}_{a \leq G \leq b} )$ with $\gamma (0) =x, \gamma (1) =x'$. Let
$t_{1} = \max \{0 \leq t<1 \hspace{0.75em} \mid \hspace{0.75em} \gamma
(t) \in S^{0}_{G=b} \}$. Then, either $\gamma (t_{1} ) \in \mathcal{M}$
and we can connect $\gamma (t_{1} )$ to a point in $B \cap ( S^{0} \cup
S^{1} )$ by a semi-algebraic path inside $\overline{B} \cap ( S^{0} \cup
S^{1} )$ using 
Case (\ref{item:axiomatic-main-case1});
otherwise, by Lemma \ref{lemmaforb} (\ref{item:lemmaforb2b}), for all
small enough $r>0$, $B_{k} ( \gamma (t_{1} ),r) \cap C_{G<b}$ is non-empty
and contained in $B$. Then, there exists $t_{2} \in (t_{1} ,1]$ such that
$z= \gamma (t_{2} ) \in B \cap S^{0}$, and the semi-algebraic path $\gamma
|_{[0,t_{2} ]}$ gives us the required path in this case.
\end{enumerate}
Taking $x$ and $x'$ in $\overline{B} \cap ( S^{0} \cup S^{1} )$, they can be
connected to points $z$ and $z'$ in $B \cap S_{0}$ by semi-algebraic paths
$\gamma$ and $\gamma'$ inside $\overline{B} \cap ( S^{0} \cup S^{1} )$ such
that, without loss of generality, $G (z) =G (z' ) =a$. Using $\mathrm{GCP} (a)$, we conclude that $\mathrm{GCP} (b)$ holds.

{\noindent}{\bf Proof of \ref{item:b}}. We have to prove that $C \cap ( S^{0}
\cup S^{1} )$ is semi-algebraically connected.

Let $x$ and $x'$ be in $C \cap ( S^{0} \cup S^{1} )$. We prove that it is
possible to connect them by a semi-algebraic path inside $C \cap ( S^{0}
\cup S^{1} )$.

Since we suppose that $\dim (C) >0$, $C_{G<b}$ is non-empty by Lemma
\ref{lemmaforb} (\ref{item:lemmaforb2c}). Using Lemma \ref{lemmaforb} (\ref{item:lemmaforb2c}), let $B_{i}$ (resp.
$B_{j}$) be a semi-algebraically connected component of $C_{<b}$ such that
$x \in \overline{B_{i}}$ (resp. $x' \in \overline{B_{j}}$).

If $i=j$, $x$ and $x'$ both lie in $\overline{B}_{i} \cap ( S^{0} \cup S^{1}
)$ which is semi-algebraically connected by (a). Hence, they can be
connected by a semi-algebraically connected path in $\overline{B}_{i} \cap (
S^{0} \cup S^{1} ) \subset C \cap ( S^{0} \cup S^{1} )$.

So let us suppose that $i \neq j$. Note that:
\begin{itemize}
\item by Lemma \ref{lemmaforb} (\ref{item:lemmaforb2a}), $\overline{B_{i}} \cap
\mathcal{M}$ and $\overline{B_{j}} \cap \mathcal{M}$ are not empty,

\item by (a) $\overline{B_{i}} \cap ( S^{0} \cup S^{1} )$ and
$\overline{B_{j}} \cap ( S^{0} \cup S^{1} )$ are semi-algebraically
connected,

\item by definition of $S^{0} \cup S^{1}$, $\mathcal{M} \subset S^{0}
\cup S^{1}$.
\end{itemize}
Then, one can connect $x$ (resp. $x'$) to a point in $\overline{B}_{i} \cap
\mathcal{M}$ (resp. $\overline{B}_{j} \cap \mathcal{M}$), so that one can
suppose without loss of generality that $x \in \overline{B}_{i} \cap
\mathcal{M}$ and $x' \in \overline{B}_{j} \cap \mathcal{M}$.

Let $\gamma : [0,1] \to C$ be a semi-algebraic path that connects $x$ to
$x'$, and let $I= \gamma^{-1}(C \cap \mathcal{M} )$ and $H= [0,1]
\setminus I$.

Since $\mathcal{M}$ is finite, we can assume without loss of generality that
$I$ is a finite set of points, and $H$ is a union of a finite number of open
intervals.

Since $\gamma (I) \subset \mathcal{M} \subset S^{0} \cup S^{1}$, it
suffices to prove that if $t$ and $t'$ are the end points of an interval in
$H$, then $\gamma (t)$ and $\gamma (t' )$ are connected by a semi-algebraic
path inside $C \cap ( S^{0} \cup S^{1} )$.

Notice that $\gamma ((t,t' )) \cap \mathcal{M} = \emptyset$, so that $\gamma
(t)$ and $\gamma (t' )$ belong to the same $\overline{B}_{\ell}$ by Lemma
\ref{lemmaforb} (\ref{item:lemmaforb2b}) .
Hence,
$\gamma(t), \gamma (t' )$ both belong to  $\overline{B}_{\ell} \cap (
S^{0} \cup S^{1} )$, and we know that $\overline{B}_{\ell} \cap ( S^{0} \cup S^{1} )$
is semi-algebraically connected by \ref{item:a}. Consequently, $\gamma (t)$ and
$\gamma (t' )$ can be connected by a semi-algebraic path in
$\overline{B}_{\ell} \cap ( S^{0} \cup S^{1} ) \subset C \cap ( S^{0} \cup
S^{1} )$.
\end{itemize}
\end{proof}

\section{Good rank property}\label{sec:grp}

In this section, we introduce matrices having the ``good rank property'' and
derive two geometric consequences of this property which will be important for
us later.

\begin{notation}
\label{not:good-rank-property} Let $m \geq 0,B =( b_{i, j} )_{0 \leq i
\leq m,1\leq j \leq k} \in \R^{( m+1 ) \times k}$, such that every $j
\times j$ sub-matrix of $B$ with $1 \leq j \leq m+1 $, has rank $j$.
We say that the matrix $B$ has {\tmem{good rank property}}.
\end{notation}

\subsection{A deformation of several equations to general
position}\label{def:severaldef}

Our first application of matrices having good rank property is to use such a
matrix to define a deformation of a finite set of polynomials with the
property of being in general position which is what we describe now (see
Proposition \ref{prop:parametrized-dimension} below). We discuss first how to
deform a given system of equation, following an idea introduced in
{\cite{JPT2012}}, so that the number of critical points of a certain well
chosen polynomial $G$ is guaranteed to be finite.

\begin{notation}
\label{not:deformation}Let $Q \in \R [ X_{1} , \ldots ,X_{k} ]$, $b= ( b_{0}
,b_{1} , \ldots ,b_{k} ) \in \R^{k+1}$, and $d\geq 0$. Let
$\zeta$ be a new variable. We denote
\begin{eqnarray}
\Def ( Q, \zeta ,b,d ) & = & ( 1- \zeta ) Q^{2}- \zeta ( b_{0} +b_{1}
X_{1}^{d} + \cdots +b_{k} X_{k}^{d} ) .\label{eqn:not:Def}
\end{eqnarray}
In the special case when $b= ( 1, \ldots ,1 )$, and $d=2\deg ( Q )$+2, we
denote
\begin{eqnarray}
\Def ( Q, \zeta ) & = & \Def ( Q, \zeta ,b,d ) .\label{eqn:not:def2}
\end{eqnarray}
\end{notation}

\begin{notation}
\label{not:goodrank} Let $m \geq 0,B =( b_{i, j} )_{0 \leq i \leq m,0 
\leq j \leq k} \in \R^{( m+1 ) \times ( k+1 )}$, a matrix having good rank
property and $b_{0} = ( 1,2, \ldots ,k )$. For $i=0, \ldots ,m$, let
$b_{i} =( b_{i, 0} , \ldots ,b_{i ,k} )$ denote the $i$-th row of $B$.

Let $\mathcal{P} = \{ P_{1} , \ldots ,P_{m} \} \subset \R [ X_{1} , \ldots
,X_{k} ]$, and $\bar{\zeta} = ( \zeta_{1} , \ldots , \zeta_{m} )$
new variables.

For any $d \geq 0$, we denote by $\Def ( \mathcal{P} , \bar{\zeta}
,B,d )$ the polynomials
\begin{equation}
\Def ( P_{1} , \zeta_{1} ,b_{1} ,d ) , \ldots , \Def ( P_{m} , \zeta_{m}
,b_{m} ,d ) , \label{deformfullrank}
\end{equation}

and denote
\begin{eqnarray}
G_{d} & = & b_{0,0} + \sum_{j=1}^{k} b_{0,j}X_{j}^{d} . 
\label{eqn:not:G}
\end{eqnarray}
\end{notation}

The following proposition and its proof are similar to results in
{\cite{JPT2012}}. We include it here for the sake of completeness.

\begin{proposition}
\label{prop:parametrized-dimension}Suppose that $B =( b_{i, j} )_{0 \leq i
\leq m,0\leq j \leq k} \in \R^{( m+1 ) \times k}$ has good rank property.
Let $0 \leq \ell \leq k$, and $d > 2\max_{1 \leq i \leq m} \deg ( P_{i}
)$. Then, for each $w \in \R^{\ell}$, and $\bar{\zeta} = ( \zeta_{1} ,
\ldots , \zeta_{m} ) \in \left( \R \setminus \{ 0 \} \right)^{m}$, $\Def (
\mathcal{P} , \bar{\zeta} ,B,d ) ( w, \cdot )$ is in general position with
respect to $G_{d} ( w, \cdot )$.
\end{proposition}

\begin{proof} Fix $w \in \R^{\ell}$, and $\bar{\zeta} \in \left( \R \setminus \{ 0 \}
\right)^{m}$. We prove that $\mathrm{Crit} \left( \Def ( \mathcal{P},
\bar{\zeta} ,B,d ) ( w, \cdot ) ,G \right)$ is finite (possibly empty).

Consider the following system of bi-homogeneous equations defining a
sub-variety $W \subset \PP^{k - \ell}_{\C} \times \PP^{m}_{\C}$:
\begin{eqnarray}
\left( \Def ( P_{i} , \zeta_{i} ,b_{i} ,d ) ( w, \cdot ) \right)^{h} &
= & 0, i=1, \ldots ,m, \nonumber\\
\sum_{i=1}^{m} \lambda_{i} \frac{\partial \left( \Def ( P_{i} ,
\zeta_{i} ,b_{i} ,d ) ( w, \cdot ) \right)^{h}}{\partial X_{j}} & = &
\lambda_{0}\frac{\partial G_{d} ( w, \cdot )^{h}}{\partial X_{j}} ,j=
\ell +1, \cdots ,k.\label{eqn:G-pseudocriticalhomogenized}
\end{eqnarray}
Let $\pi : \PP^{k - \ell}_{\C} \times \PP^{m}_{\C} \longrightarrow
\PP^{m}_{\C}$ be the projection map to the second factor. 

It follows from
the definition of $\mathrm{Crit} \left( \Def ( \mathcal{P}, \bar{\zeta} ,B,d )
( w, \cdot ) ,G_{d} \right)$ that
this set 
is 
contained in the real
affine part of $\pi ( W )$, and thus in order to prove that 
\[
\mathrm{Crit}
\left( \Def ( \mathcal{P}, \bar{\zeta} ,B,d ) ( y, \cdot ) ,G_{d} \right)
\]
is finite (possibly empty), it suffices to show that the complex projective
variety $\pi ( W )$ is a finite number of points (possibly empty). So, we
prove that the projective variety $\pi ( W ) \subset \mathbb{P}_{\C}^{m}$
has an empty intersection with the hyperplane at infinity defined by $X_{0}
=0$.

Substituting, $X_{0} =0$ in the system
(\ref{eqn:G-pseudocriticalhomogenized}), we get,
\begin{eqnarray}
\zeta_{i} ( b_{i, \ell +1} X_{\ell +1}^{d} + \cdots +b_{i,k} X_{k}^{d}
) & = & 0, i=1, \ldots ,m, 
\label{eqn:G-pseudocritical-parametrized-firstset}\\
d \left( \sum_{i=1}^{m} \zeta_{i} \lambda_{i} b_{i,j} - \lambda_{0 }
b_{0,j} \right)X_{j}^{d-1} & = & 0, j= \ell +1, \cdots ,k. 
\label{eqn:G-pseudocritical-parametrized-secondse}
\end{eqnarray}
There are two cases to consider.

\begin{descriptioncompact}
\item[Case 1] $m \geq k- \ell$: In this case, since the matrix of
coefficients in the first set of equations
\[ \left(\begin{array}{cccc}
 b_{1, \ell +1} & \cdot & \cdot & b_{1,k}\\
 \cdot & \cdot & \cdot & \cdot\\
 b_{m, \ell +1} & \cdot & \cdot & b_{m,k}
 \end{array}\right) \]
has rank $k- \ell$ which follows from the given property of the matrix
$B$, we get that $X_{\ell +1} = \cdots =X_{k} =0$, which is impossible.

\item[Case 2] $m<k- \ell$: Consider the second set of equations
 involving the
Lagrangian variables $\lambda_{0} , \ldots , \lambda_{m}$. Since, the
matrix $B$ has the property that every $( m+1 ) \times ( m+1 )$
sub-matrix has rank $( m+1 )$, we have for every choice $J \subset [ \ell
+1,k ] , \mathrm{card} ( J ) =m+1$, the system of
equations
\[ \sum_{i=1}^{p} \zeta_{i} \lambda_{i} b_{i,j} - \lambda_{0 } b_{0,j} =0
,j \in J \]
has an empty solution in $\PP^{m}_{\C}$, and hence at least $k-m- \ell$
amongst the variables $X_{\ell +1} , \ldots ,X_{k}$ must be equal to $0$.
Suppose that $X_{m+ \ell +1} = \cdots =X_{k } =0$. Now, from the property
that the all $m \times m$ sub-matrices of $B$ have full rank we obtain
that the only solution to system
(\ref{eqn:G-pseudocritical-parametrized-firstset}) with $X_{m+ \ell +1} =
\cdots =X_{k} =0$, is the one with $X_{\ell +1} = \cdots =X_{k} =0$, which
is impossible. 
\end{descriptioncompact}

This proves that in both cases the projective variety $\pi ( W ) \subset
\mathbb{P}_{\C}^{m}$ has an empty intersection with the hyperplane at
infinity defined by $X_{0} =0$, and hence $\pi ( W )$ is a finite number of
points which finishes the proof.
\end{proof}

\subsection{$( B,G )$-pseudo-critical values\label{sec:G-special-values}}

We now describe a second application of matrices with the good rank property.
Given a finite family of polynomials $\mathcal{P}= \{ P_{1} , \ldots ,P_{s} \}
\subset \R [ X_{1} , \ldots ,X_{k} ]$, a matrix $B =( b_{i, j} )_{1 \leq i
\leq s,0\leq j \leq k} \in \R^{s \times ( k+1 )}$ having good rank
property (see Notation \ref{not:goodrank}), and a polynomial $G \in \R [ X_{1}
, \ldots ,X_{k} ]$, we define a finite set $\mathcal{D} ( \mathcal{P},B,G )
\subset \R$ which we call the {\tmem{$( B,G )$-pseudo-critical values of the
family $\mathcal{P}$}}.

These $( B,G )$-pseudo-critical values are used to ensure good connectivity
properties in the case of basic closed semi-algebraic sets.

\begin{definition}
\label{def:Bpseudocritical} Let $\mathcal{P}= \{ P_{1} , \ldots
,P_{s} \} \subset \R [ X_{1} , \ldots ,X_{k} ]$, $B =( b_{i, j} )_{1 \leq
i \leq s,0\leq j \leq k} \in \R^{s \times ( k+1 )}$ a matrix having good
rank property (see Notation \ref{not:goodrank}), and let $G \in \R [ X_{1} ,
\ldots ,X_{k} ]$. We denote for $1 \leq i \leq s$,
\begin{eqnarray*}
H_{i} & = & b_{i,0} + \sum_{j=1}^{k} b_{i,j} X_{j}^{d} ,
\end{eqnarray*}
where $d$ is the least even number greater than $\max_{P \in \mathcal{P}}
\deg ( P )$. For $I \subset [ 1,s ] $, and $\sigma \in \{ -1,1
\}^{I}$, we denote
\begin{eqnarray*}
\tilde{\mathcal{P}}_{I, \sigma ,B} & = & \{ P_{i} +\gamma \sigma ( i )
H_{i} \mid i \in I \} .
\end{eqnarray*}
We say that $c \in \R$, is a $( B,G )$-pseudo-critical value of
$\mathcal{P}$, if there exists $I \subset [ 1,s ] $ with
$\mathrm{card} ( I ) \leq k$, $\sigma \in \{ -1,1 \}^{I}$, $( x, \lambda )
\in \R \langle \gamma \rangle^{k} \times \R \langle \gamma \rangle^{(
\mathrm{card} ( I ) +1 )}$ bounded over {\R}, such that
\begin{eqnarray*}
c & = & \lim_{\gamma}G ( x ) ,
\end{eqnarray*}
and $( x, \lambda ) \in \ZZ \left( \mathrm{CritEq} ( \tilde{\mathcal{P}}_{I,
\sigma ,B} ,G ) , \R \langle \gamma \rangle^{k+ \mathrm{card} ( I ) +1}
\right)$ (see Definition \ref{def:G-critical-point}, 
\eqref{eqn:criteqnpg}). We denote set of all $( B,G )$-pseudo-critical
values of $\mathcal{P}$ by $\mathcal{D} ( \mathcal{P},B,G )$.
\end{definition}

The property of $( B,G )$-pseudo-critical values used in the paper is the
following result. Its proof is postponed to the Annex (Section
\ref{auxproof}).

\begin{proposition}
\label{prop:properties-of-special} Let $\mathcal{P},\mathcal{Q}
\subset \R [ X_{1} , \ldots ,X_{k} ]$, $B =( b_{i, j} )_{1 \leq i \leq
\mathrm{card} ( \mathcal{P} ) + \mathrm{card} ( \mathcal{Q} ) ,0\leq j \leq k}
\in \R^{( \mathrm{card} ( \mathcal{P} ) + \mathrm{card} ( \mathcal{Q} ) ) \times
( k+1 )}$, a matrix having good rank property, and $G\in \R [ X_{1} ,
\ldots ,X_{k} ]$, where $d$ is the least even number greater than $\max_{P
\in \mathcal{P}} \deg ( P )$. Suppose that $S= \mathrm{Bas} (
\mathcal{P},\mathcal{Q} )$ is bounded. Then,
\begin{enumerate}
\item \label{item:properties-of-special1} 
the set $\mathcal{D}=\mathcal{D} ( \mathcal{P} \cup \mathcal{Q},B,G
)$ is finite;
\item \label{item:properties-of-special2} for any interval $[ a,b ] \subset \R$ and $c\in [ a,b ]$, with
$\{ c \} \supset \mathcal{D} \cap [ a,b ]$, if $D$ is a
semi-algebraically connected component of $S_{a\leq G \leq b}$, then
$D_{G=c}$ is a semi-algebraically connected component of $S_{G=c}$.
\end{enumerate}
\end{proposition}

\section{Deformation to the special case }\label{sec:deformation}

Our aim in this section is to associate to a basic semi-algebraic set $S=
\mathrm{Bas} ( \mathcal{P} , \mathcal{Q} )$ a deformation
$\tilde{S}$ of $S$, and a special tuple $ ( \tilde{S} , \tilde{\mathcal{M}} ,
\ell , \tilde{S}^{0} ,\mathcal{D}^{0} ,\mathcal{M}^{0} )$ (cf. Definition
\ref{def:property-special}).

\begin{notation}
\label{convention}
We fix for the remainder of this section:
\begin{enumerate}
\item $p \in \mathbb{N} ,1 \leq p \leq k;$

\item two finite sets of polynomials $\mathcal{P}\subset \R [ X_{1} ,
\ldots ,X_{k} ]$,
$ \mathcal{Q}= \{ Q_{1} , \ldots ,Q_{q} \} \subset \R [ 
X_{1} , \ldots ,X_{k} ]$;

\item $d= \max_{P \in \mathcal{P} \cup \mathcal{Q}} \deg ( P )$;

\item $G=G_{2d+2}$=$1+ \sum^{k}_{i=1} i X_{i}^{2d+2}$.
\end{enumerate}
\end{notation}

\subsection{Deformation of $\mathrm{Bas} ( \mathcal{P} , \mathcal{Q} )$ to
$\mathrm{Bas} ( \tilde{\mathcal{P}} , \tilde{\mathcal{Q}})$}

\begin{notation}
\label{not:cauchy}Let $\mathcal{H}_{N,k} =( h_{i j} )_{0 \leq i \leq N,0
\leq j \leq k}$, be an $( N +1 ) \times ( k+1 )$ matrix with integer entries
defined by $h_{i,j} =j^{i+1}$ and for each $i,0\leq i\leq N,0
\leq j \leq k$.

Notice that the matrix $\mathcal{H}_{N,k}$ has good rank property (see
Notation \ref{not:good-rank-property}), since every submatrix of
$\mathcal{H}_{N,k}$ is a generalized Vandermonde matrix (see for example
{\cite{Polya-Szego}}, page 43).
\end{notation}

\begin{notation}
\label{not:order}Given a finite list of variables $\zeta = ( \zeta_{1} ,
\ldots , \zeta_{t} )$, we denote by $\R \langle \zeta \rangle = \R \langle
\zeta_{1} , \ldots , \zeta_{t} \rangle$ the field $\R \la \zeta_{1} \ra
\cdots \la \zeta_{t} \ra$ and for any $\xi \in \R \langle \zeta_{1} , \ldots
, \zeta_{t} \rangle$ bounded over $\R \langle \zeta_{1} , \ldots , \zeta_{i}
\rangle$, $i<t$, we denote by $\lim_{\zeta_{i+1}} ( \xi )$ the element $(
\lim_{\zeta_{i+1}} \circ \cdots \circ \lim_{\zeta_{t}} ) ( \xi )$ of $\R
\langle \zeta_{1} , \ldots , \zeta_{i} \rangle$. 
For an element $f= \sum_{\alpha} c_{\alpha} \zeta^{\alpha} \in \D [ \zeta_{1} , \ldots ,
\zeta_{t} ]$,
 we will denote by $o_{\zeta} ( f ) = \alpha_{0} \in
\mathbb{N}^{t}$, such that $\zeta^{\alpha_{0}}$ is the largest element of
$\mathrm{supp} ( f )=\{ \zeta^{\alpha} \mid c_{\alpha} \neq 0 \}$ in
the unique ordering of the real closed field $\R \langle \zeta \rangle$. For
$\alpha , \beta \in \mathbb{N}^{t}$, we denote $\alpha \geq \beta$, if
$\zeta^{\alpha} \geq \zeta^{\beta}$.
\end{notation}

We now define families $\tilde{\mathcal{P}}$ and $\tilde{\mathcal{Q}}$ such
that $\mathrm{Bas} ( \tilde{\mathcal{P}} , \tilde{\mathcal{Q}} )$ is a
deformation of $\mathrm{Bas} ( \mathcal{P} , \mathcal{Q} )$. 

\begin{notation}
\label{not:notationtilde} Let
\[ H_{i} = h_{i,0} + \sum^{k}_{j=1} h_{i j}X_{j}^{2d+2} , \]
where $d= \max_{P \in \mathcal{P} \cup \mathcal{Q}} \deg ( P )$,  and the 
$h_{i j}$'s are the entries in the matrix $\mathcal{H}_{k-p+q,k}$.
For $1\leq j \leq q $, let
\[ \tilde{Q}_{j} =( 1- \zeta )Q_{j} +\zeta H_{k-p+j}, \]
and define
\[ \tilde{\mathcal{Q}} = \{ \tilde{Q}_{1} , \ldots , \tilde{Q}_{q} \} . \]

Let
\begin{eqnarray*}
P_{1}^{\star} & = & ( 1- \eps ) \sum_{P \in \mathcal{P}} P^{2} + \eps
( X_{p+1}^{2d+2} + \ldots +X_{k}^{2d+2} +X_{p+1}^{2} + \ldots +X_{k}^{2} )
 ,\\
P_{i}^{\star} & = & \dfrac{\partial P_{1}^{\star}}{\partial X_{p+i}},2 \leq i
\leq k-p,\\
\mathcal{P}^{\star} & = & \{ P_{1}^{\star} , \ldots ,P^{\star}_{k-p} \} .
\end{eqnarray*}
For $1 \leq i \leq k-p$, let
\[ \tilde{P}_{i}=\left( 1- \delta \right) P^{\star}_{i} -\delta
 H_{i} , \]
finally, define
\[ \tilde{\mathcal{P}} = \{ \tilde{P}_{1} , \ldots , \tilde{P}_{k-p} \} , \]
\end{notation}

\begin{proposition}
\label{prop:limtilde}
Suppose that $\mathrm{Bas} ( \mathcal{P} , \mathcal{Q} )$
is bounded, and that for each $y \in \R^{p}$, $\ZZ ( \mathcal{P},
\R^{k} )_{y}$ is a finite number of points (possibly empty). Let
$\mathrm{Bas} (\tilde{\mathcal{P}} , \tilde{\mathcal{Q}} )
\subset \R \la \zeta , \eps , \delta \ra^{k}$. Then,
\[ \mathrm{Bas} ( \mathcal{P} , \mathcal{Q} ) = \lim_{\zeta} ( \mathrm{Bas} (
 \tilde{\mathcal{P}} , \tilde{\mathcal{Q}} ) ) . \]
\end{proposition}

\begin{proof}
It is clear that $\lim_{\zeta} ( \mathrm{Bas} (\tilde{\mathcal{P}}
, \tilde{\mathcal{Q}} ) ) \subset \mathrm{Bas} ( \mathcal{P} , \mathcal{Q}
)$. We now prove that $\mathrm{Bas} ( \mathcal{P} , \mathcal{Q} ) \subset
\lim_{\zeta} ( \mathrm{Bas} ( \tilde{\mathcal{P}} ,
\tilde{\mathcal{Q}}) )$. Let $x=( y,z ) \in \mathrm{Bas} ( \mathcal{P}
, \mathcal{Q} )$, where $y \in \R^{p}$ and $z \in \R^{k-p}$. For each (of
the finitely many) $z \in \R^{k-p}$ such that $x= ( y,z ) \in \mathrm{Bas} (
\mathcal{P} , \mathcal{Q} )$, there exists a bounded semi-algebraically
connected component $C_{z}$ of the non-singular hypersurface 
$\ZZ (P_{1}^{\star} ( y, \cdot ) , \R \langle \zeta,\eps \rangle^{k-p} )$ 
such
that 
$\lim_{\eps} ( C_{z} )= z$.
Now, the system $\mathcal{P}^{\star} ( y, \cdot ) $ has only simple
zeros in 
$\R \langle \zeta,\eps \rangle^{k-p}$ 
(see {\cite{BPRbook2}},
Proposition 12.44) and contains the non-empty set of $X_{p+1}$-extremal
points of $C_{z}$. 
Let 
$z' \in \R \langle \zeta, \eps \rangle^{k-p}$ 
be an
$X_{p+1}$-extremal point of $C_{z}$. Then, since $z'$ is a simple zero of
the system $\mathcal{P}^{\star} ( y, \cdot )$, 
there must exist $z'' \in \ZZ
( \tilde{\mathcal{P}} , \R \la \zeta , \eps , \delta
\ra^{k-p} )$ such that $\lim_{\eps} (z'' )= z'$.
Moreover, it is clear that $x'' = ( y,z'' ) \in \mathrm{Bas} (
\tilde{\mathcal{P} , \tilde{\mathcal{Q}}} )$, and that
$\lim_{\zeta} ( x'' ) =x$, which finishes the proof.
\end{proof}

\subsection{General position and definition of $\tilde{\mathcal{M}}$}

Suppose now that
$\ZZ( \mathcal{P},\R^k)$ is strongly of dimension $\leq p$, and let
$S= \mathrm{Bas} ( \mathcal{P} , \mathcal{Q} ) \subset \R^{k}$.
Let
\[ \tilde{S} = \mathrm{Bas} ( \tilde{\mathcal{P}} ,
 \tilde{\mathcal{Q}} ) \subset \R \la \zeta , \eps , \delta
 \ra^{k} \]
following Notation \ref{not:notationtilde}.

\begin{proposition}
\label{prop:general-position} For every $\ell \leq p$ and for every $w \in
\R \la \zeta , \eps , \delta \ra^{\ell}$,
$\tilde{\mathcal{P}} ( w,- ) , \tilde{\mathcal{Q}} ( w,- )$ is in general
position with respect to $G ( w,- )$.
\end{proposition}

\begin{proof}
Follows from Definition \ref{def:general-position} and Proposition
\ref{prop:parametrized-dimension} noting that 
$\zeta,\delta \neq 0$
in $\R
\la \zeta , \eps , \delta \ra$.
\end{proof}

\begin{corollary}
\label{cor:M-finite}
The set $\tilde{\mathcal{M}} = \mathrm{Cr} ( \tilde{\mathcal{P}} ,
\tilde{\mathcal{Q}} ,G )$ is finite.
\end{corollary}

\begin{corollary}
\label{cor:strong-dimension}
$\ZZ(\tilde{\mathcal{P}},\R^k)$, 
and
$\mathrm{Bas} ( \tilde{\mathcal{P}} ,
 \tilde{\mathcal{Q}} )$ are strongly of dimension $\le p$.
\end{corollary}

\begin{proof}
Applying Proposition \ref{prop:general-position} with $\ell = p$, and noting that 
$\mathrm{card}(\mathcal{P}) = k-p$, we get that
 for every $w \in\R \la \zeta , \eps , \delta \ra^{p}$,
$\ZZ(\tilde{\mathcal{P}}, \R\la \zeta , \eps , \delta \ra^{k})_{w}$ 
is finite (possibly empty) by Remark \ref{rem:general-position}.
 It then follows from 
 Definition \ref{def:strong-dimension} that,
 $\ZZ(\tilde{\mathcal{P}}, \R\la \zeta , \eps , \delta \ra^{k})$ is strongly of dimension
 $\leq p$. The same then holds for
 $\mathrm{Bas} ( \tilde{\mathcal{P}} ,
 \tilde{\mathcal{Q}} )$, since 
 $\mathrm{Bas} ( \tilde{\mathcal{P}} ,
 \tilde{\mathcal{Q}} )\subset 
 \ZZ(\tilde{\mathcal{P}}, \R\la \zeta , \eps , \delta \ra^{k})$.
 \end{proof}

\subsection{Definition of $\tilde{\mathcal{A}}$}
Since we have replaced $\mathrm{Bas} ( \mathcal{P} , \mathcal{Q} )$ by
$\mathrm{Bas} (\tilde{\mathcal{P}} , \tilde{\mathcal{Q}} )$, we need
to associate to any given finite set of points $\mathcal{A} \subset
\mathrm{Bas} (\mathcal{P} , \mathcal{Q})$, a corresponding finite
set of points $\tilde{\mathcal{A}} \subset \mathrm{Bas} (
\tilde{\mathcal{P}} , \tilde{\mathcal{Q}} )$ whose limits contain
$\mathcal{A}$, and which moreover ensures certain connectivity properties (see
Proposition \ref{prop:propertyofAtilde}).

In our constructions we will often require to choose a finite subset of a
given semi-algebraic set $S$ which meets every semi-algebraically connected
component of $S$. Since the relevant connectivity properties of the constructions 
will not depend on how these
points are chosen it is convenient to have the following notation. Later in
the descriptions of our algorithms we will specify precisely how these points
are chosen.

\begin{notation}
\label{not:sample} For any closed and bounded semi-algebraic subset $S
\subset \R^{k}$, we denote by $\mathrm{Samp} ( S )$ some finite subset of $S$
which meets every semi-algebraically connected component of $S$. 
\end{notation}

\begin{notation}
\label{not:closest} We associate to two closed and bounded semi-algebraic
sets $S_{1} ,S_{2} \subset \R^{k}$ a finite set of points $\mathrm{MinDist}
( S_{1} ,S_{2} ) \subset S_{1}$ defined as follows. Let $M$ be the set of
local minimizers of the polynomial function $F ( X,Y ) = \sum_{i=1}^{k} (
X_{i} -Y_{i} )^{2}$ on the set $S_{1} \times S_{2}$ and let $\pi_{1} ,
\pi_{2} : \R^{k} \times \R^{k} \longrightarrow \R^{k}$ be the projections on
the first and second components respectively. Let
\begin{eqnarray*}
\mathrm{MinDi} ( S_{1} ,S_{2} ) & = & \pi_{1} ( \mathrm{Samp} ( M ) ) \cup
\pi_{2} ( \mathrm{Samp} ( M ) )
\end{eqnarray*}
using Notation \ref{not:sample}.
\end{notation}

\begin{proposition}
\label{prop:closesta}Let $T \subset \R \langle \zeta \rangle^{k}$ be a
closed semi-algebraic set bounded over $\R$, and $x \in \lim_{\zeta} ( T )$.
Then, $\mathrm{MinDi} ( T, \{ x \} ) \neq \emptyset$, and $x \in \lim_{\zeta} 
( \mathrm{MinDi} ( T, \{ x \} ) )$.
\end{proposition}

\begin{proof}
Let $C$ be the semi-algebraically connected component of $\lim_{\zeta} ( T
)$ containing $x$. Then, there exists semi-algebraically connected
components $C_{1} , \ldots ,C_{m}$ of $T$, such that $C =\bigcup_{i=1}^{m}
\lim_{\zeta} ( C_{i} )$. Hence, there exists $i,1 \leq i \leq m$, such that
$x \in \lim_{\zeta} ( C_{i} )$. Since, $C_{i}$ is bounded over $\R$, the
subset $M_{i,x } \subset C_{i}$ of points which achieve the minimum distance
from $x$ to $C_{i}$ is non-empty. Every semi-algebraically connected
component of $M_{i,x }$ is a semi-algebraically connected component of the
set $M_{x} \subset T$ of points which achieve the minimum distance from $x$
to $T$. Hence, $M_{i,x }$ contains one point, $\tilde{x}$, which is included
in $\mathrm{MinDi} ( T, \{ x \} )$. It is now clear that $\mathrm{MinDi} ( T, \{
x \} ) \neq \emptyset$, and that $x= \lim_{\zeta} ( \tilde{x} ) \in
\lim_{\zeta} ( \mathrm{MinDi} ( T, \{ x \} ) \neq \emptyset )$.
\end{proof}

\begin{proposition}
\label{prop:closestb}Let $T_{1} ,T_{2} \subset \R \langle \zeta \rangle^{k}$
be closed semi-algebraic sets bounded over $\R$. Then, for every
$\tilde{C} , \tilde{D}$ semi-algebraically connected components of
$T_{1}$ and $T_{2}$ respectively, such that $\lim_{\zeta} ( \tilde{C} )
\cap \lim_{\zeta} ( \tilde{D} )$ is non-empty, $\lim_{\zeta} ( \tilde{C}
\cap \mathrm{MinDi} ( T_{1} ,T_{2} ) ) \cap \lim_{\zeta} ( \tilde{D}
\cap \mathrm{MinDi} ( T_{1} ,T_{2} ) )$ is non-empty, and meets every
semi-algebraically connected component of $\lim_{\zeta} ( \tilde{C} )
\cap \lim_{\zeta} ( \tilde{D} )$.
\end{proposition}

\begin{proof}
Let $M$ denote the semi-algebraic subset of $\R \langle \zeta \rangle^{k}
\times \R \langle \zeta \rangle^{k}$ consisting of the local minimizers of
the polynomial function $F ( X,Y ) = \sum_{i=1}^{k} ( X_{i} -Y_{i} )^{2}$ on
$T_{1} \times T_{2}$. Also, note that the function $F$ is proportional to
the square of the distance to the diagonal $\Delta \subset \R \langle \zeta
\rangle^{k} \times \R \langle \zeta \rangle^{k}$.

Let $B$ be a semi-algebraically connected component of $\lim_{\zeta} (
\tilde{C} ) \cap \lim_{\zeta} ( \tilde{D} )$. Notice that $( B \times B )
\cap \Delta$ is a semi-algebraically connected component of $( \lim_{\zeta}
( \tilde{C} ) \times \lim_{\zeta} ( \tilde{D} ) ) \cap \Delta$. Let $(
\tilde{u}_{0} , \tilde{v}_{0} ) \in \tilde{C} \times \tilde{D}$ such
that $\lim_{\zeta} ( \tilde{u}_{0} ) = \lim_{\zeta} ( \tilde{v}_{0}
) \in B$. Notice that $\lim_{\zeta} ( F ( \tilde{u}_{0} , \tilde{v}_{0} ) )
=F ( \lim_{\zeta} ( \tilde{u}_{0} ) , \lim_{\zeta} ( \tilde{v}_{0} ) ) =0$,
and hence $F ( \tilde{u}_{0} , \tilde{v}_{0} )$ is infinitesimally small.
Let
\begin{eqnarray*}
U & = & \{ ( \tilde{u} , \tilde{v} ) \in \tilde{C} \times \tilde{D} \mid
F ( \tilde{u} , \tilde{v} ) <F ( \tilde{u}_{0} , \tilde{v}_{0} )
\} .
\end{eqnarray*}
Since the image under $\lim_{\zeta}$ of a bounded, semi-algebraically
connected set is semi-algebraically connected (see Proposition 12.43 in
{\cite{BPRbook2}}), for any semi-algebraically connected component $V$
of $U$, $\lim_{\zeta} ( V )$ is either contained in $( B \times B ) \cap
\Delta$ or disjoint from $( B \times B ) \cap \Delta$. Denote by $U'$ the
union of semi-algebraically connected components $V$ of $U$ such that
$\lim_{\zeta} ( V ) \subset ( B \times B ) \cap \Delta$, and denote by
$\overline{U'} \subset \tilde{C} \times \tilde{D}$ the closure of $U'$. If
$U'$ is empty then $( \tilde{u}_{0} , \tilde{v}_{0} )$ is a local
minimizer of $F$ on $\tilde{C} \times \tilde{D}$ and we are done. Otherwise,
the minimum of $F$ on $\overline{U'}$ is strictly smaller than $F (
\tilde{u}_{0} , \tilde{v}_{0} )$, and it must be realized at a point of $U'$,
since $F ( \tilde{u} , \tilde{v} ) =F ( \tilde{u}_{0} , \tilde{v}_{0} )$ for
all $( \tilde{u} , \tilde{v} ) \in \overline{U'} \setminus U'$, and we are
done.
\end{proof}

We now let $\mathcal{A} \subset S$ be a fixed finite set of points contained
in $S$.

Let (using Notation \ref{not:closest})
\begin{eqnarray*}
\tilde{\mathcal{A}} & = & \mathrm{MinDi} ( \tilde{S} ,\mathcal{A} ) \cup
\mathrm{MinDi} ( \tilde{S} , \tilde{S} ) .
\end{eqnarray*}
\begin{proposition}
\label{prop:propertyofAtilde}The finite set
$\tilde{\mathcal{A}} \subset \tilde{S}$ has the
following properties.
\begin{enumerate}
\item \label{item:propertyofAtilde1}
$\lim_{\zeta} (\tilde{\mathcal{A}}) \supset
\mathcal{A}$;

\item \label{item:propertyofAtilde2}
for every pair of semi-algebraically connected components
$\tilde{C} , \tilde{D}$ of $\tilde{S}$ such that $\lim_{\zeta} (
\tilde{C} ) \cap \lim_{\zeta} ( \tilde{D} )$ is non-empty,
$\lim_{\zeta} ( \tilde{C} \cap \tilde{\mathcal{A}} ) \cap
\lim_{\zeta} ( \tilde{D} \cap \tilde{\mathcal{A}} )$ is non-empty,
and meets every semi-algebraically connected component of $\lim_{\zeta} (
\tilde{C} ) \cap \lim_{\zeta} ( \tilde{D} )$;

\item\label{item:propertyofAtilde3}
 $\tilde{\mathcal{A}}$ meets every semi-algebraically connected
component of $\tilde{S}$.
\end{enumerate}
\end{proposition}

\begin{proof}
Part (\ref{item:propertyofAtilde1}) 
follows from Proposition \ref{prop:closesta} after observing
that $\mathrm{MinDist} ( \tilde{S} ,\tilde{\mathcal{A}} ) = \bigcup_{x \in
\mathcal{A}} \mathrm{MinDi} ( \tilde{S} , \{ x \} )$ (see Notation
\ref{not:closest}), and the fact that $\tilde{\mathcal{A}}$ contains
$\mathrm{MinDist} ( \tilde{S} ,\mathcal{A} )$.

 Part (\ref{item:propertyofAtilde2}) follows directly from Proposition \ref{prop:closestb} with
$T_{1} = T_{2} = \tilde{S}$ and the fact that $\tilde{A}$ contains
$\mathrm{MinDi} ( \tilde{S} , \tilde{S} )$.

Part (\ref{item:propertyofAtilde3}) is a special case of Part (\ref{item:propertyofAtilde2}), with $T_{1} =T_{2} = \tilde{S}$, and $\tilde{C} = \tilde{D}$.
\end{proof}

\begin{corollary}
\label{cor:propertyofAtilde}Let $x,x' \in \tilde{S}$ such that $\lim_{\zeta}
( x' ) \in \Cc ( \lim_{\zeta} ( x ) ,S )$. Then, there exist elements
$\tilde{x}_{0} =x, \ldots , \tilde{x}_{2n+1} =x'$ of $\tilde{S}$ such that
\begin{enumerate}
\item for all $i=1, \ldots ,n,$ $\lim_{\zeta_{t}} ( \tilde{x}_{2i-1} ) =
\lim_{\zeta} ( \tilde{x}_{2i} )$,

\item for all $i=0, \ldots ,n$, $\tilde{x}_{2i+1} \in \Cc (
\tilde{x}_{2i} , \tilde{S} )$,

\item for all $i=1, \ldots ,2n$, $\tilde{x}_{i} \in
\tilde{\mathcal{A}}$.
\end{enumerate}
\end{corollary}

\begin{proof} Follows clearly from Part (\ref{item:propertyofAtilde2}) of Proposition \ref{prop:propertyofAtilde}.
\end{proof}

\subsection{Definition of $\tilde{S}^{0}$}

We want to consider $G$-critical points parametrized by $\R^{\ell}$.

\begin{notation}
\label{not:lagrange-equations} Let $G \in \R [ X_{1} , \ldots ,X_{k} ]$ and
$\mathcal{P} =\{ P_{1} , \ldots ,P_{m} \} \subset \R [ X_{1} , \ldots
,X_{k} ]$ be a finite family of polynomials.

Let $0 \leq \ell \leq k$ and consider the system of equations
$\mathrm{CritEq}_{\ell} ( \mathcal{P} ,G )$
\begin{eqnarray*}
P_{j} & = & 0,j=1, \ldots ,m,\\
\sum^{m}_{j=1} \lambda_{j} \frac{\partial P_{j}}{\partial X_{i}} -
\lambda_{0} \frac{\partial G}{\partial X_{i}} & = & 0,i= \ell +1, \ldots
,k,\\
\sum_{j=0}^{m} \lambda_{j}^{2} -1 & = & 0.
\end{eqnarray*}
The set $\mathrm{Crit}_{\ell} ( \mathcal{P} ,G ) \subset \R^{k}$ is the
projection to  $\R^{k}$ of 
\[
\ZZ \left( \mathrm{CritEq}_{\ell} (
\mathcal{P} ,G ) , \R^{k} \times \R^{\max ( m,k ) +1} \right).
\]

Note that for every $w \in \R^{\ell}$,
\[ \mathrm{Crit}_{\ell} ( \mathcal{P} ,G )_{w} = \mathrm{Crit} ( \mathcal{P} (
 w, \cdot ) ,G ( w, \cdot ) ) . \]
\end{notation}

We now fix $\ell ,1 \leq \ell <p$.

\begin{notation}
\label{notationStilde0} Let $\tilde{\mathcal{Q}}' \subset
\tilde{\mathcal{Q}}$. Define
\[ \tilde{S}^{0} ( \tilde{\mathcal{Q}}' ) = \mathrm{Cr}_{\ell} (
 \tilde{\mathcal{P}} \cup \tilde{\mathcal{Q}}' ,G ) \cap \tilde{S} ,
\]
\[ \tilde{S}^{0} = \bigcup_{\tilde{\mathcal{Q}}' \subset
 \tilde{\mathcal{Q}}} \tilde{S}^{0} ( \tilde{\mathcal{Q}}' ) . \]
\end{notation}

\begin{proposition}
\label{prop:properties-of-S0}For each $w\in \R \la \zeta , \eps ,
\delta \ra^{\ell}$:
\begin{enumerate}
\item\label{item:properties-of-S01}
$\tilde{S}^{0}_{w}$ is a finite set;

\item \label{item:properties-of-S02} 
$\tilde{S}^{0}_{w}$ meets every semi-algebraically connected
component of $\tilde{S}_{w}$, and contains for every semi-algebraically
connected component $C$ of $\tilde{S}_{w}$ a minimizer of $G$ over $C$.
\end{enumerate}
\end{proposition}

\begin{proof}
 Part (\ref{item:properties-of-S01}) is immediate from Proposition
\ref{prop:general-position}.

Part (\ref{item:properties-of-S02}) follows from the fact that for each semi-algebraically
connected component $C$ of $\tilde{S}_{w}$, there exists some
$\tilde{\mathcal{Q}}'$ such that the minimizer of $G$ over $C$ is a local
minimizer $x \in \left( \ZZ ( \tilde{\mathcal{P}} \cup
\tilde{\mathcal{Q}}' , \R \la \zeta , \eps , \delta
\ra ) \right)_{w}$ of $G$ over $\left( \ZZ (
\tilde{\mathcal{P}} \cup \tilde{\mathcal{Q}}' , \R \la \zeta ,
\eps , \delta \ra ) \right)_{w}$, and then $x$ clearly
belongs to $\mathrm{Cr}_{\ell} (\tilde{\mathcal{P}} \cup \tilde{\mathcal{Q}}')
,G ) \cap \tilde{S}$. Since, $\tilde{S}_{w}$ is closed and bounded, every
semi-algebraically connected component $C$ of $\tilde{S}_{w}$ must contain a
minimizer of $G$ over $C$, and hence $\tilde{S}^{0}_{w}$ meets every
semi-algebraically connected component of $\tilde{S}_{ w}$.
\end{proof}

\subsection{Definition of $\mathcal{D}^{0}$, $\mathcal{M}^{0}$}

\begin{notation}
\label{not:D0}Let
\begin{eqnarray*}
\tilde{S} & = & \mathrm{Bas} ( \tilde{\mathcal{P}} ,
\tilde{\mathcal{Q}} ) ,\\
F & =& \prod_{\tilde{\mathcal{Q}}' \subset \tilde{\mathcal{Q}}} F (
\tilde{\mathcal{Q}}' ) ,
\end{eqnarray*}
where
\begin{eqnarray*}
F ( \tilde{\mathcal{Q}}' ) & = & \sum_{P \in \mathrm{CrEq}_{\ell} (
\tilde{\mathcal{P}} \cup \tilde{\mathcal{Q}}' ,G )} P^{2} .
\end{eqnarray*}
We denote (see Definition \ref{def:Bpseudocritical})
\begin{eqnarray}
\label{eqn:defD0}
\mathcal{D}^{0} & = & \mathcal{D} ( \{ F \} \cup \tilde{\mathcal{Q}}
,\mathcal{H}_{\mathrm{card} ( \mathcal{Q} ) +1,2k-p+ \mathrm{card} (
\mathcal{Q} ) +1} ,G )
\end{eqnarray}
considering the polynomials in $\{ F \} \cup \tilde{\mathcal{Q}}$ as
elements of 
\[\R \left[ \zeta , \eps , \delta \right] [ X_{1} , \ldots ,X_{k}
, \lambda_{0} , \ldots , \lambda_{k-p+ \mathrm{card} ( \mathcal{Q} )} ].
\]
\end{notation}

Let $\mathcal{M}^{ 0} = \mathrm{Samp} \left( \bigcup_{c \in \mathcal{D}^{0}}
\tilde{S}^0_{G=c} 
\right)$ be a finite set of points meeting every
semi-algebraically connected component of $\bigcup_{c \in \mathcal{D}^{0}}
\tilde{S}^0_{G=c}$.

\begin{lemma}
\label{lem:properties-of-M0-and-D0} The sets $\mathcal{D}^{0}$ and
$\mathcal{M}^{0}$ have the following properties:
\begin{enumerate}
\item \label{item:properties-of-M0-and-D01}
 for every interval $[ a,b ] \subset \R \la \zeta , \eps ,
\delta \ra$ and $c\in [ a,b ]$, with $\{ c \} \supset
\mathcal{D}^{0} \cap [ a,b ]$, if $D$ is a semi-algebraically connected
component of $( \tilde{S}^{0} )_{a\leq G \leq b}$, then $D_{G=c}$ is a
semi-algebraically connected component of $( \tilde{S}^{0} )_{G=c}$;

\item \label{item:properties-of-M0-and-D02} $\mathcal{M}^{0}$ meets every semi-algebraically connected component
of $( \tilde{S}^{0} )_{G=a} $ for all $a\in \mathcal{D}^{0}$.
\end{enumerate}
\end{lemma}

\begin{proof}
Part (\ref{item:properties-of-M0-and-D01}): notice that $\mathcal{D}^{0}$ is the finite set of $( B,G
)$-pseudo-critical values of the family $\{ F \} \cup
\tilde{\mathcal{Q}}$, for the matrix $B=\mathcal{H}_{\mathrm{card} (
\mathcal{Q} ) +1,2k-p+ \mathrm{card} ( \mathcal{Q} ) +1}$ which has the good rank
property. Hence, using Part (\ref{item:properties-of-special2}) of Proposition
\ref{prop:properties-of-special} we have that for every interval $[ a,b ]
\subset \R \la \zeta , \eps , \delta \ra$ and $c\in [
a,b ]$, with $\{ c \} \supset \mathcal{D}^{0} \cap [ a,b ]$, if $D$ is
a semi-algebraically connected component of $( \mathrm{Bas} ( \{ F \} ,
\tilde{\mathcal{Q}} ) )_{a\leq G \leq b}$, then $D_{G=c}$ is a
semi-algebraically connected component of $( \mathrm{Bas} ( \{ F \} ,
\tilde{\mathcal{Q}} ) )_{G=c}$.

 To finish the proof of Part (\ref{item:properties-of-M0-and-D01}) observe that $\tilde{S}^{0}$ is the
image of $\mathrm{Bas} ( \{ F \} , \tilde{\mathcal{Q}} ) \subset \R \la
\zeta , \eps , \delta \ra^{k} \times \R \la \zeta , \eps , \delta \ra^{k-p+
\mathrm{card} ( \mathcal{Q} ) +1}$ under projection to $\R \la \zeta , \eps ,
\delta \ra^{k}$, and the fibers of this projection are intersections of
linear subspaces with the unit sphere in $\R \la \zeta , \eps ,
\delta \ra^{k-p+ \mathrm{card} ( \mathcal{Q} ) +1}$, and the
polynomial $G$ is independent of the $\lambda$'s. Hence, the
semi-algebraically connected components of $( \tilde{S}^{0} )_{a\leq G
\leq b}$ and $( \tilde{S}^{0} )_{G=c}$, are in correspondence with those of
$( \mathrm{Bas} ( \{ F \} , \tilde{\mathcal{Q}} ) )_{a\leq G \leq b}$
and $( \mathrm{Bas} ( \{ F \} , \tilde{\mathcal{Q}} ) )_{G=c}$
respectively.

Part (\ref{item:properties-of-M0-and-D02}) is clear from the definition of $\mathcal{M}^{0}$. 
\end{proof}

\begin{remark}
\label{rem:BasFQtilde} Note that the elements of $\mathcal{D}^{0}$ are the
$( B,G )$-pseudo-critical values of the family $\{ F \} \cup
\tilde{\mathcal{Q}}$, for the matrix $B=\mathcal{H}_{\mathrm{card} (
\mathcal{Q} ) +1,2k-p+ \mathrm{card} ( \mathcal{Q} ) +1}$, and thus satisfy
the properties of Proposition \ref{prop:properties-of-special} with respect
to the level sets of the polynomial $G$ restricted to $\mathrm{Bas} ( \{ F \}
, \tilde{\mathcal{Q}} ) \subset \R \la \zeta , \eps , \delta
\ra^{k} \times \R \la \zeta , \eps , \delta
\ra^{k-p+ \mathrm{card} ( \mathcal{Q} ) +1}$. Part (\ref{item:properties-of-M0-and-D01}) of Lemma
\ref{lem:properties-of-M0-and-D0} implies that the same properties also
hold for $\tilde{S}^{0}$ with respect to the values $\mathcal{D}^{0}$
(recall that $\tilde{S}^{0}$ is defined in Notation \ref{notationStilde0} as
the projection of $\mathrm{Bas} ( \{ F \} , \tilde{Q} )$ to $\R \la
\zeta , \eps , \delta \ra^{k}$).
\end{remark}

\subsection{Definition of $\mathcal{N} , \tilde{S}^{1}$, and
$\mathcal{B}$\label{sec:goodconnectivity}}

We will use the two following propositions which use the definitions given above.
\begin{proposition}
\label{prop:special}The tuple $ ( \tilde{S} , \tilde{\mathcal{M}} , \ell ,
\tilde{S}^{0} ,\mathcal{D}^{0} ,\mathcal{M}^{0} )$ is special (cf.
Definition \ref{def:property-special}).
\end{proposition}

\begin{proof}
Follows from Lemma \ref{lem:properties-of-M0-and-D0} and Definition
\ref{def:property-special}.
\end{proof}

We denote $\mathcal{N} = \pi_{[ 1, \ell ]} ( \tilde{\mathcal{M}} \cup
\mathcal{M}^{0} \cup \tilde{\mathcal{A}} )$, $\tilde{S}^{
1} = \tilde{S}_{\mathcal{N}}$ and $\mathcal{B} = (
\tilde{S}^{0} )_{\mathcal{N}}$. Note that $\tilde{S}^{0} \cap \tilde{S}^{
1} = \mathcal{B} .$

\begin{proposition}
\label{prop:good-connectivity}
The semi-algebraic set $\tilde{S}^{0} \cup \tilde{S}^{1}$
has good connectivity property with respect to $\tilde{S}$.
\end{proposition}

\begin{proof} Follows from Proposition \ref{prop:special} and Proposition
\ref{prop:axiomatic-main}.
\end{proof}

\section{Critical points and minors}\label{sec:minors}

In the previous section, $\tilde{S}^{0}$ is described as the image of a
projection applied to the basic semi-algebraic set $\mathrm{Bas} ( \{ F \} ,
\tilde{\mathcal{Q}} )$ (see Remark \ref{rem:BasFQtilde}). This means that
we cannot hope to compute a roadmap of $\tilde{S}^{0}$ by a divide-and-conquer
algorithm directly since the input to such an algorithm should be a basic
semi-algebraic set. In this section, we give an alternative description of
$\tilde{S}^{0}$ (see Proposition \ref{prop:coverS0} below) as a (limit of)
union of basic semi-algebraic sets which allows us to get past this problem.

\subsection{Description of critical points}\label{sec:morsetheory}

In the case when $\mathcal{P} = \{ P_{1} , \ldots ,P_{m} \}$, $m<k$, is in
general position with respect to {\tmem{$G \in \R [ X_{1} , \ldots ,X_{k}
]$}}, we can describe $\mathrm{Crit} ( \mathcal{P} ,G ) \subset \R^{k}$ as
follows.

Define the Jacobian matrix
\begin{eqnarray*}
\mathcal{\mathrm{Jac}} & = & \left(\begin{array}{ccccc}
\frac{\partial G}{\partial X_{1}} & \frac{\partial P_{1}}{\partial X_{1}}
&\cdots & \frac{\partial P_{m}}{\partial X_{1}} & \\
\vdots & \vdots && \vdots & \\
\frac{\partial G}{\partial X_{k}} & \frac{\partial P_{1}}{\partial X_{k}}
& \cdots & \frac{\partial P_{m}}{\partial X_{k}} & 
\end{array}\right)
\end{eqnarray*}
whose rows are indexed by $[ 1,k ]$ and columns by $[ 0,m ]$.

For $J \subset [ 1,k ]$ and $J' \subset [ 0,m ]$, let $\mathrm{Jac} (
J,J' )$ the matrix obtained from $\mathrm{Jac}$ by extracting the rows numbered by elements of
$J$, and the columns numbered by elements of  $J'$.

We use the following convenient notation in what follows.
For any finite set $X$, and any integer $r\geq0$, we will denote by  $\binom{X}{r}$ the set of 
all subsets of $X$ of cardinality $r$.

For each $0 \leq r \leq m$, and each $J \in \binom{[ 1,k ]}{r}, J' \in
\binom{[ 0,m ]}{r}$, let
\[ \mathrm{jac} ( J,J' ) = \det ( \mathrm{Jac} ( J,J' ) ) . \]
For every $i \in [ 1,k ] \setminus J$, and $i' \in [ 0,m ] \setminus J'$, let
\begin{eqnarray*}
\mathrm{Eq} ( J,J' ) & = & \mathcal{P} \cup \bigcup_{i \in [ 1,k ]
\setminus J,i' \in [ 0,m ] \setminus J'} \mathrm{jac} ( J \cup \{ i \} ,J'
\cup \{ i' \} ) ,
\end{eqnarray*}
and
\begin{eqnarray*}
\mathrm{Cons} ( J,J' ) & = & \left\{ x \in \ZZ \left( \mathrm{Eq} (
J,J' ) , \R^{k} \right)\mid \mathrm{jac} (
J,J' ) ( x ) \neq 0 \right\} .
\end{eqnarray*}
\begin{proposition}
\label{prop:Cramer} If $\mathcal{P}= \{ P_{1} , \ldots ,P_{m} \}$ is in
general position with respect to {\tmem{$G \in \R [ X_{1} , \ldots ,X_{k}
]$}}, the finite variety $\mathrm{Crit} ( \mathcal{P},G )$ is the union of
the various 
\[
\mathrm{Cons} ( J,J' ), 0 \leq r \leq m,  J \in \binom{[1,k]}{r} ,J'
\in \binom{[ 0,m ]}{r}.
\]
\end{proposition}

\begin{proof}
 We first prove that $\mathrm{Crit} ( \mathcal{P},G )$ is contained in
the union of the various $\mathrm{Cons} ( J,J' ) ,0 \leq r \leq m, J \in
\binom{[ 1,k ]}{r} ,J' \in \binom{[ 0,m ]}{r}$. It follows from Definition
\ref{def:G-critical-point} that each $x \in \mathrm{Crit} ( \mathcal{P},G )$
is contained in the projection to $\R^{k}$ of the set of solutions to the
system of equations, $\mathrm{CritEq} ( \mathcal{P} ,G )$ (cf.
\eqref{eqn:criteqnpg}).

Substituting, $X=x$ in the above system, we obtain the following system of
homogeneous linear equations in $\lambda = ( \lambda_{0} , \ldots ,
\lambda_{m} )$.
\begin{eqnarray}
\lambda_{0} \frac{\partial G}{\partial X_{i}} ( x ) + \sum^{m}_{j=1}
\lambda_{j} \frac{\partial P_{j}}{\partial X_{i}} ( x ) & = & 0,i=1,
\ldots ,k.\label{eqn:intermediate5}
\end{eqnarray}
Let the rank of the matrix of coefficients of the above system be $r_{x}$.
Then, $r_{x} \leq m$, since there must exist a $\lambda = ( \lambda_{0} ,
\ldots , \lambda_{m} )$ satisfying (\ref{eqn:intermediate5}) and $\lambda
\neq ( 0, \ldots ,0 )$ since it has to satisfy also the equation
\begin{eqnarray*}
\sum_{j=0}^{m} \lambda_{j}^{2} -1 & = & 0.
\end{eqnarray*}
Then there exists $J \subset \binom{[ 1,k ]}{r_{x}} ,J' \subset \binom{[ 0,m
]}{r_{x}}$ such that the $r_{x} \times r_{x}$ sub-matrix of the matrix of
coefficients with rows indexed by $J$ and columns indexed by $J'$ has full
rank and hence $\mathrm{jac} ( J,J' ) ( x ) \neq 0$. Then, clearly for every
$i \in [ 1,k ] \setminus J$, and $i' \in [ 0,m ] \setminus J'$,
\[ \mathrm{jac} ( J \cup \{ i \} ,J' \cup \{ i' \} ) ( x ) =0. \]
Hence, $x \in \mathrm{Cons} ( J,J' )$ using the definition of the set
$\mathrm{Cons} ( J,J' )$. This completes the proof that $\mathrm{Crit} (
\mathcal{P},G )$ is contained in union of the various $\mathrm{Cons} ( J,J' )
,0 \leq r \leq m,J \in \binom{[ 1,k ]}{r} ,J' \in \binom{[
0,m ]}{r}$.

To prove the reverse inclusion fix, $r$, $0 \leq r \leq m$, $J \in \binom{[
1,k ]}{r} ,J' \in \binom{[ 0,m ]}{r}$, and let $x \in \mathrm{Cons} ( J,J' )$.
Then, $\mathrm{jac} ( J,J' ) ( x ) \neq 0$, and for each $i \in [ 1,k ]
\setminus J$, and $i' \in [ 0,m ] \setminus J'$, $\mathrm{jac} ( J \cup \{ j
\} ,J' \cup \{ i' \} ) ( x )= 0$. We now show that there exists
$\lambda = ( \lambda_{0} , \ldots , \lambda_{m} )$ such that $( x, \lambda
)$ satisfy the system of equations (\ref{eqn:prop-critical}). It follows
from Cramer's rule that for each $i \in J$, the equation
\begin{eqnarray*}
\lambda_{0} \frac{\partial G}{\partial X_{i}} ( x ) + \sum^{m}_{j=1}
\lambda_{j} \frac{\partial P_{j}}{\partial X_{i}} ( x ) & = & 0 
 ,
\end{eqnarray*}
is satisfied after making the substitution
\begin{eqnarray}
\lambda_{j} & = & - \sum_{j' \in J' \setminus \{ j \}} \dfrac{\mathrm{jac}
( J,J' \setminus \{ j \} \cup \{ j' \} ) ( x )}{\mathrm{jac} ( J,J' )
( x )} \lambda_{j'} ,\label{eqn:intermediate6}
\end{eqnarray}
for each $j \in J'$.

Moreover, substituting the expressions in (\ref{eqn:intermediate6}) in the
equations indexed by $i \in [ 1,k ] \setminus J$ in (\ref{eqn:prop-critical}), clearing the denominator
$\mathrm{jac} ( J,J' ) ( x )$, we have that the coefficient of $\lambda_{i'}$ for
$i' \in [ 0,m ] \setminus J'$ equals $\mathrm{jac} ( J \cup \{ i \} ,J' \cup
\{ i' \} ) ( x )$,  and hence equal to $0$. Thus, the equations indexed by $i
\in [ 1,k ] \setminus J$ in (\ref{eqn:prop-critical}) are satisfied as well.
Finally since, $r \leq m<m+1$, we can assume, that there exists $\lambda = (
\lambda_{0} , \ldots , \lambda_{m} ) \in \R \la \zeta , \eps , \delta
\ra^{m+1}$ with not all coordinates equal to $0 $, such that $( x,
\lambda )$ satisfy all but the last equation in (\ref{eqn:criteqnpg}), and
it follows that there exists $\lambda$ such that $( x, \lambda )$ satisfy
(\ref{eqn:criteqnpg}), and hence $x \in \mathrm{Crit} ( \mathcal{P},G )$. This
proves the reverse inclusion.
\end{proof}

\subsection{Description of $\tilde{S}^{0}$ using minors}\label{descriptionS0}

\begin{notation}
\label{not:covering} Following Notation \ref{convention} and Notation
\ref{not:notationtilde}:
\begin{enumerate}
\item Let $\ell <p \leq k$, $\tilde{\mathcal{Q}}' \subset \tilde{Q}$ and $
\tilde{\mathcal{P}} \cup \tilde{\mathcal{Q}}' = \{ F_{1} , \ldots ,F_{m}
\}$.

\item Define the matrix
\begin{eqnarray*}
\mathrm{Jac} ( \ell , \tilde{\mathcal{Q}}' ) & = &
\left(\begin{array}{ccccc}
\frac{\partial G}{\partial X_{\ell +1}} & \frac{\partial
F_{1}}{\partial X_{\ell +1}} & \cdots & \frac{\partial F_{m}}{\partial
X_{\ell +1}} & \\
\vdots & \vdots && \vdots & \\
\frac{\partial G}{\partial X_{k}} & \frac{\partial F_{1}}{\partial
X_{k}} & \cdots & \frac{\partial F_{m}}{\partial X_{k}} & 
\end{array}\right)
\end{eqnarray*}
whose rows are indexed by $[ \ell +1,k ]$ and columns by $[ 0,m ]$.

For each $\alpha = ( \tilde{\mathcal{Q}}' ,r,J,J' )$ with
$\tilde{\mathcal{Q}}' \subset \tilde{Q}$, $0 \leq r \leq m$, $J \in
\binom{[ \ell +1,k ]}{r}, J' \in \binom{[ 0,m ]}{r}$ denote by
\[ \mathrm{jac} ( \alpha ) = \det ( \mathrm{Jac} ( \ell , \tilde{\mathcal{Q}}'
 ) ( J,J' ) ) . \]

Moreover, for each $i \in [ \ell +1,k ] \setminus J$, $i' \in [ 0,m ]
\setminus J'$, let
\[ \mathrm{jac} ( \alpha ,i,i' ) = \det ( \mathrm{Jac} ( \ell ,
 \tilde{\mathcal{Q}}' ) ( J \cup \{ i \} ,J' \cup \{ i' \} ) ) . \]

Let
\begin{eqnarray}
\mathcal{P}^{0} ( \alpha ) & = &\tilde{\mathcal{P}} \cup
\tilde{\mathcal{Q}}' \cup \bigcup_{i \in [ \ell +1,k ] \setminus J,i'
\in [ 0,m ] \setminus J'} \{ \mathrm{jac} ( \alpha ,i,i' ) \} , 
\label{eqn:P0}\\
\mathcal{Q}^{0} ( \alpha ) & = & \tilde{\mathcal{Q}} \cup \{ \mathrm{jac}
( \alpha )^{2} - \gamma \}\label{eqn:Q0}
\end{eqnarray}
where $\gamma$ is a new variable.

\item Define
\begin{eqnarray*}
S^{0} ( \alpha ) & = & \mathrm{Bas} ( \mathcal{P}^{0} ( \alpha )
,\mathcal{Q}^{0} ( \alpha ) ) \subset \R \la \zeta , \eps , \delta ,
\gamma \ra^{k} .
\end{eqnarray*}
\end{enumerate}
\end{notation}

\begin{notation}
\label{not:indices} Fixing $\tilde{\mathcal{P}},
\tilde{\mathcal{Q}} , \ell$ with $0 \leq \ell <p \leq k $, we
denote by $\mathcal{I}( \tilde{P} , \tilde{Q} , \ell )$ the set of
quadruples $\alpha = ( \tilde{\mathcal{Q}}' ,r,J,J' )$ with
$\tilde{\mathcal{Q}}' \subset \tilde{\mathcal{Q}}$, $0 \leq r \leq m$, $J \in \binom{[
\ell +1,k ]}{r}, J' \in \binom{[ 0,m ]}{r}$. 
\end{notation}

\begin{proposition}
\label{prop:coverS0}{\tmdummy}

\[ \begin{array}{lll}
 \tilde{S}^{0} & = & \lim_{\gamma} \left( \bigcup_{\alpha \in \mathcal{I}
 ( \tilde{P} , \tilde{Q} , \ell )} S^{0} ( \alpha ) \right) .
 \end{array} \]
\end{proposition}

\begin{proof}
 We first prove that
\begin{eqnarray}
\tilde{S}^{0} & = & \bigcup_{\alpha \in \mathcal{I} ( \tilde{P} , \tilde{Q}
, \ell )} \{ x \in \mathrm{Bas} ( \mathcal{P}^{0} ( \alpha ) ,
\tilde{\mathcal{Q}} ) \mid \mathrm{jac} ( \alpha ) ( x ) \neq 0\} . 
\label{eqn:intermediate}
\end{eqnarray}

Using Notation \ref{notationStilde0}, notice that for each
$\tilde{\mathcal{Q}}' \subset \tilde{\mathcal{Q}}$, $\tilde{S}^{0} (
\tilde{\mathcal{Q}}' )$ is the set of $G$-critical points of $\ZZ (
\tilde{\mathcal{P}} \cup \tilde{\mathcal{Q}}' ( w, \cdot ) , \R
\la \zeta , \eps , \delta \ra^{k} )$ contained in
$\tilde{S}$, as $w$ varies over $\R \la \zeta , \eps , \delta
\ra^{\ell}$, and $\tilde{S}^{0}=\bigcup_{\tilde{\mathcal{Q}}'
\subset \tilde{Q}} \tilde{S}^{0} ( \tilde{\mathcal{Q}}' )$. It follows from
Proposition \ref{prop:Cramer} that for each $\tilde{\mathcal{Q}}' \subset
\tilde{Q}$,
\begin{eqnarray}
\tilde{S}^{0} ( \tilde{\mathcal{Q}}' ) & = & _{\mathcal{}} \{ x \in
\mathrm{Bas} ( \mathcal{P}^{0} ( \alpha ) , \tilde{\mathcal{Q}} ) \mid
\mathrm{jac} ( \alpha ) ( x ) \neq 0\} ,\label{eqn:intermediate3}
\end{eqnarray}
and this proves (\ref{eqn:intermediate}).

Noticing that all the sets $\{ x \in \mathrm{Bas} ( \mathcal{P}^{0} (
\alpha ) , \tilde{\mathcal{Q}} ) \mid \mathrm{jac} ( \alpha ) ( x ) \neq 0\}
\subset \R \la \zeta , \eps , \delta \ra^{k}$ are bounded, it follows from
the definition of $S^{0} ( \alpha )$ that
\begin{eqnarray}
\lim_{\gamma}S^{0} ( \alpha ) & = & \overline{_{\mathcal{}} \{ x \in
\mathrm{Bas} ( \mathcal{P}^{0} ( \alpha ) , \tilde{\mathcal{Q}} ) \mid
\mathrm{jac} ( \alpha ) ( x ) \neq 0\}} ,\label{eqn:intermediate4}
\end{eqnarray}
using {\cite{BPRbook2}} Proposition 11.56.

Also, since $\tilde{S}^{0}$ is closed it follows from
(\ref{eqn:intermediate3}) that
\begin{eqnarray*}
\tilde{S}^{0} & = & \bigcup_{\alpha \in \mathcal{I} ( \tilde{P} , \tilde{Q}
, \ell )} \overline{\{ x \in \mathrm{Bas} ( \mathcal{P}^{0} ( \alpha ) ,
\tilde{\mathcal{Q}} ) \mid \mathrm{jac} ( \alpha ) ( x ) \neq 0 \}} .
\end{eqnarray*}
The proposition now follows from (\ref{eqn:intermediate4}).
\end{proof}

\begin{remark}
Note that if the description of $S$ does not involve any inequality, this is
the first time that an inequality appears in the construction. 
\end{remark}

\subsubsection{Definition of $\mathcal{A} ( \alpha
)$}\label{sec:definitionB}

Since we have covered $\tilde{S}^{0}$ by the (limit of the) union of the
$S^{0} ( \alpha )$, we need to choose a finite set of points ensuring
connectivity properties.

\begin{notation}
\label{def:Azero}

For each $\alpha \in \mathcal{I} ( \tilde{\mathcal{P}} ,
\tilde{\mathcal{Q}} , \ell )$, we denote (using Notation
\ref{not:closest})
\begin{eqnarray*}
\mathcal{A} ( \alpha ) & = & \mathrm{MinDi} ( S^{0} ( \alpha ) ,
\mathcal{B} ) \cup \left( \bigcup_{\beta \in \mathcal{I} ( \tilde{P} ,
\tilde{Q} , \ell )} \mathrm{MinDi} ( S^{0} ( \alpha ) ,S^{0} ( \beta ) )
\right) .
\end{eqnarray*}
\end{notation}

We have the following property of the finite sets $\mathcal{A} ( \alpha )
 , \alpha \in \mathcal{I} ( \tilde{\mathcal{P}} ,
\tilde{\mathcal{Q}} , \ell )$.

\begin{proposition}
\label{prop:alpha-beta}
For every $\alpha , \beta$ in $\mathcal{I} ( \tilde{P} , \tilde{Q} , \ell )$
the following are true.
\begin{enumerate}
\item\label{item:alpha-beta1} $\bigcup_{\alpha \in \mathcal{I} ( \tilde{\mathcal{P}} ,
\tilde{\mathcal{Q}} , \ell )} \lim_{\gamma} ( \mathcal{A ( \alpha
)} ) \supset \mathcal{B}$.

\item\label{item:alpha-beta2} For $C$ and $D$ semi-algebraically connected components (not
necessarily distinct) of $S^{0} ( \alpha )$ and $S^{0} ( \beta )$ such
that $\lim_{\gamma} ( C ) \cap \lim_{\gamma} ( D )$ is
non-empty, $\lim_{\gamma} ( C \cap \mathcal{A} ( \alpha ) ) \cap
\lim_{\gamma} ( D \cap \mathcal{A} ( \beta ) )$ is non-empty,  and meets
every semi-algebraically connected component of $\lim_{\gamma} ( C )
\cap \lim_{\gamma} ( D )$.

\item \label{item:alpha-beta3} $\mathcal{A ( \alpha )}$ meets every semi-algebraically connected
component of $S^{0} ( \alpha )$.
\end{enumerate}
\end{proposition}

\begin{proof}

Part (\ref{item:alpha-beta1}) follows from Proposition \ref{prop:closesta}.

Part (\ref{item:alpha-beta2}) follows from Proposition \ref{prop:closestb}.

Part (\ref{item:alpha-beta3}) is a special case of Part (\ref{item:alpha-beta2}).
\end{proof}

\section{Divide and conquer algorithm}\label{sec:tree}

\subsection{Description of the tree $\mathrm{Tree} ( V,\mathcal{A} )$, and its
associated roadmap}

We first describe the tree which is going to be constructed in the algorithm,
using the definitions in the two former sections. We then prove that the
(limits of the) union of the leaves of the tree give a roadmap.

Since new infinitesimals will be added at each level of the tree, we need the
following notation.

\begin{notation}
\label{not:rtdt}We consider an ordered domain $\D$ contained in a real
closed field $\R$. We denote by $\D_{t}$ the polynomial ring $\D [ \eta ]$
and we denote by $\R_{t}$ the real closed field $\R \langle \eta \rangle$
where $\eta = ( \eta_{1} , \ldots , \eta_{t} )$ and $\eta_{i} = \left(
\zeta_{i} , \eps_{i} , \delta_{i} , \gamma_{i} \right)$. By convention
$\R_{0} = \R$ and $\D_{0} = \D$.
\end{notation}

\subsubsection{Description of the tree $\mathrm{Tree} ( V,\mathcal{A}
)$}
\label{subsec:tree}
We start with a bounded real algebraic variety $V = \ZZ ( P, \R^{k}
)$, 
strongly 
of dimension $\leq k'$ (assumed to be a power of $2$ for
simplicity), and suppose that $\mathcal{A} \subset V$ is a finite set of
points meeting every semi-algebraically connected component of $V$. The
algorithm constructs a rooted tree, which we denote by $\mathrm{Tree} (
V,\mathcal{A} )$.

More precisely, the root, $\mathfrak{r}$, of $\mathrm{Tree} ( V,\mathcal{A} )$
has level $0$, contains the empty string $s ( \mathfrak{r} )$, the real
algebraic variety $\mathrm{Bas} ( \mathfrak{r} )= V$, and the finite set of
points $\mathcal{A} ( \mathfrak{r} )= \mathcal{A}$. A node $\mathfrak{n}$ of
the tree $\mathrm{Tree} ( V,\mathcal{A} )$ at level $t \neq 0$ contains a
string $s ( \mathfrak{n} ) \in \{ 0,1 \}^{t}$, a basic semi-algebraic set
$\mathrm{Bas} ( \mathfrak{n} ) = \{ w ( \mathfrak{n} ) \} \times \mathrm{Bas} (
\mathcal{P} ( \mathfrak{n} ) ,\mathcal{Q} ( \mathfrak{n} ) ) \subset
\R_{t}^{k}$, 
such that $\mathrm{Bas} ( \mathfrak{n} )$ is strongly of dimension $\leq k' /2^{t}$, 
$w ( \mathfrak{n} ) \in \R_{t}^{\mathrm{Fix} ( \mathfrak{n} )}$,
defining
\[ \mathrm{Fix} ( \mathfrak{n} ) = \sum_{i=1}^{t} s ( \mathfrak{n} )_{i} k'
 /2^{i} , \]
and a finite number of points $\mathcal{A} ( \mathfrak{n} ) \subset \mathrm{Bas}
( \mathfrak{n} )$ meeting every semi-algebraically connected components of
$\mathrm{Bas} ( \mathfrak{n} )$. 
A node $\mathfrak{n}$ of the tree $\mathrm{Tree}
( V,\mathcal{A} )$ of level $t \neq 0$ is either a left child, if the last
bit of $s ( \mathfrak{n} )$ is $0$, or a right child if the last bit of $s (
\mathfrak{n} )$ is $1$.

If the level of the node $\mathfrak{n}$ is $< \log ( k' )$, we construct the
left children and right children of $\mathfrak{n}$ as follows. 
We replace
$\mathrm{Bas} ( \mathfrak{n} )$ by a semi-algebraic set $\widetilde{\mathrm{Bas}}
( \mathfrak{n} ) = \{ w ( \mathfrak{n} ) \} \times \mathrm{Bas} (
\tilde{\mathcal{P}} ( \mathfrak{n} ) , \tilde{\mathcal{Q}} ( \mathfrak{n} )
) \subset \R_{t}^{k}$ (see Notation \ref{not:notationtilde}) such that
\begin{enumerate}
\item
$\lim_{\zeta_{t}} ( \widetilde{\mathrm{Bas}} ( \mathfrak{n} ) ) =\mathrm{Bas}
( \mathfrak{n} )$ 
(using Proposition \ref{prop:limtilde}), and
\item
$\widetilde{\mathrm{Bas}} ( \mathfrak{n} )$ is strongly of dimension 
$\leq k' /2^{t}$ (using Corollary \ref{cor:strong-dimension}).
\end{enumerate}
We define semi-algebraic subsets $\widetilde{\mathrm{Bas}} (
\mathfrak{n} )^{0} , \widetilde{\mathrm{Bas}} ( \mathfrak{n} )^{1}$ of
$\widetilde{\mathrm{Bas}} ( \mathfrak{n} )$, 
with $\widetilde{\mathrm{Bas}} (
\mathfrak{n} )^{ 1}$ strongly of dimension $\leq k'/2^{t+1}$,
by the method
described in Section \ref{sec:deformation}, with 
$p= k'/2^{t}$, $\ell =p/2= k'/2^{t+1}$.

We define $\mathcal{N ( \mathfrak{n} )}$, $\tilde{\mathcal{A}} (
\mathfrak{n} )$, and $\mathcal{\mathcal{B}} ( \mathfrak{n} )$ as in Section
\ref{sec:deformation}. For every $w \in \mathcal{N ( \mathfrak{n} )}$ we have
a right child $\mathfrak{m}$ of the node $\mathfrak{n}$, with
\begin{eqnarray}
s ( \mathfrak{m} ) & = & s ( \mathfrak{n} ) 1, \nonumber\\
w ( \mathfrak{m} ) & = & ( w ( \mathfrak{n} ) ,w ) , \nonumber\\
\mathrm{Bas} ( \mathfrak{m} ) & =& 
\Ext \left( \widetilde{\mathrm{Bas}} ( \mathfrak{n} )_{w ( \mathfrak{m} )} , \R_{t+1} \right) , \nonumber\\
\mathcal{A} ( \mathfrak{m} ) & = & ( \tilde{\mathcal{A}} (
\mathfrak{n} ) \cup \mathcal{B} ( \mathfrak{n} ) )_{w (
\mathfrak{m} )} .\label{eqn:Aofm}
\end{eqnarray}
Recall that $\widetilde{\mathrm{Bas}} ( \mathfrak{n} )^{0} \subset \R_{t} \la
\zeta_{t+1} , \eps_{t+1} , \delta_{t+1} \ra^{k}$ is defined as an image of a
certain semi-algebraic set under a projection along Lagrangian variables. We
are able to identify (by the method of Section \ref{descriptionS0}, using
Notation \ref{not:indices}) a finite family $( \mathrm{Bas} ( \mathfrak{n} )^{0}
( \alpha ) )_{\alpha \in \mathcal{I} ( \mathfrak{n} )}$ (with $\mathcal{I} (
\mathfrak{n} ) =\mathcal{I} ( \tilde{\mathcal{P}} ( \mathfrak{n} ) ,
\tilde{\mathcal{Q}} ( \mathfrak{n} ) ,  k'/2^{t+1} )$ of basic
semi-algebraic subsets of $\Ext \left( \widetilde{\mathrm{Bas}} (
\mathfrak{n} )^{0} , \R_{t+1} \right)$, with each
\[ \begin{array}{lll}
 \mathrm{Bas} ( \mathfrak{n} )^{0} ( \alpha ) & := & \{ w ( \mathfrak{n} )
 \} \times \mathrm{Bas} ( \mathcal{P}^{0} ( \alpha ) ,\mathcal{Q}^{0} (
 \alpha ) )
 \end{array} \subset \R_{t+1}^{k} , \]
 such that
\[ \bigcup_{\alpha \in \mathcal{I} ( \mathfrak{n} )} \lim_{\gamma_{t+1}} (
 \mathrm{Bas} ( \mathfrak{n} )^{0} ( \alpha ) ) =\widetilde{\mathrm{Bas}} (
 \mathfrak{n} )^{0} \]
using Proposition \ref{prop:coverS0}.

For each $\alpha \in \mathcal{I}( \mathfrak{n} )$ we include a left child
node $\mathfrak{n} ( \alpha )$, with
\begin{eqnarray*}
s ( \mathfrak{\mathfrak{n} ( \alpha )} ) & = & s ( \mathfrak{n} ) 0,\\
w ( \mathfrak{\mathfrak{n} ( \alpha )} ) & = & w ( \mathfrak{n} ) ,\\
\mathrm{Bas} ( \mathfrak{n} ( \alpha ) ) & = & \mathrm{Bas} ( \mathfrak{n} )^{0}
( \alpha ) ,\\
\mathcal{A} ( \mathfrak{n} ( \alpha ) ) & = & \mathcal{A} ( \alpha ) \mbox{ (using the definition in Section \ref{sec:definitionB})}.
\end{eqnarray*}

If the level of the node $\mathfrak{n}$ is $\log ( k' )$, then $\mathfrak{n}$
is a leaf of $\mathrm{Tree} ( V,\mathcal{A} )$.  Note that the basic
semi-algebraic set $\mathrm{Bas} ( \mathfrak{n} )$ contained in a leaf
$\mathfrak{n}$ is 
strongly 
of dimension $\leq 1$.

\begin{definition}
\label{def:leftandrightleaves}We denote by $\mathrm{Leav} ( V,\mathcal{A} )
$ the set of leaf nodes of $\mathrm{Tree} ( V,\mathcal{A} )$. When
$\mathfrak{n}$ is a node in $\mathrm{Tree} ( V,\mathcal{A} )$, we denote by
\begin{enumerate}
\item $\mathrm{Leav} ( \mathfrak{n} )$ the set of leaves of the subtree of
$\mathrm{Tree} ( V,\mathcal{A} )$ rooted at $\mathfrak{n}$,

\item $\mathrm{Leav}^{1} ( \mathfrak{n} )$ the set  of leaves $\mathfrak{l}$
of the subtree of $\mathrm{Tree} ( V,\mathcal{A} )$ rooted at $\mathfrak{n}$
with $s ( \mathfrak{l} )$ consisting of $s ( \mathfrak{n} )$ followed only
by 1,

\item $\mathrm{Leav}^{0} ( \mathfrak{n} )$ the set  of leaves $\mathfrak{l}$
of the subtree of $\mathrm{Tree} ( V,\mathcal{A} )$ rooted at $\mathfrak{n}$
with $s ( \mathfrak{l} )$ consisting of $s ( \mathfrak{n} )$ followed only
by 0.
\end{enumerate}
\end{definition}

Note that for every node $\mathfrak{n}$ of level $t$ of $\mathrm{Tree} (
V,\mathcal{A} )$ we have
\begin{equation}
\mathcal{A} ( \mathfrak{n} ) \subset \bigcup_{\mathfrak{l} \in
\mathrm{Leav}^{1} ( \mathfrak{n} )} \lim_{\zeta_{t+1}} ( \mathcal{A} (
\mathfrak{l} ) ) , \label{eqn:rightmostA}
\end{equation}
\begin{equation}
\mathcal{B} ( \mathfrak{n} ) \subset \bigcup_{\mathfrak{l} \in
\mathrm{Leav}^{1} ( \mathfrak{n} )} \lim_{\gamma_{t+1}} ( \mathcal{A} (
\mathfrak{l} ) ) . \label{eqn:rightmostB}
\end{equation}
A useful fact is the following.

\begin{proposition}
\label{prop:additive}Suppose that $\mathcal{A}_{1} , \mathcal{A}_{2}$ are
finite subsets of $V$, and $\mathcal{A} =\mathcal{A}_{1} \cup
\mathcal{A}_{2}$.

Then,
\begin{eqnarray*}
\bigcup_{\mathfrak{l} \in \mathrm{Leav} ( V,\mathcal{A}_{1} ) \cup
\mathrm{Leav} ( V, \mathcal{A}_{2})}\mathrm{Bas}
( \mathfrak{l} ) & = & \bigcup_{\mathfrak{l} \in\mathrm{Leav} ( V,
\mathcal{A} )} \mathrm{Bas} ( \mathfrak{l} ) .
\end{eqnarray*}
\end{proposition}

\begin{proof}
We prove by induction on the level $t$ the two following statements.
\begin{enumerate}
\item For each node $\mathfrak{m}$ of $\mathrm{Tree} ( V,\mathcal{A}_{1} )$
(respectively $\mathrm{Tree} ( V, \mathcal{\mathcal{A}_{2}} )$) with
$\mathrm{level} ( \mathfrak{m} ) =t$, there exists a node $\mathfrak{m}'$ of
$\mathrm{Tree} ( V, \mathcal{A} )$ with level $t$ such that
\begin{eqnarray}
( w ( \mathfrak{m}' ) ,s ( \mathfrak{m}' ) , \tilde{\mathcal{P}} (
\mathfrak{m}' ) , \tilde{\mathcal{Q}} ( \mathfrak{m}' ) ) & = & ( w (
\mathfrak{m} ) ,s ( \mathfrak{m} ) , \tilde{\mathcal{P}} ( \mathfrak{m}
) , \tilde{\mathcal{Q}} ( \mathfrak{m} ) ) ,\label{eqn:monotone1}\\
\mathcal{A} ( \mathfrak{m}' ) & \supset & \mathcal{A} ( \mathfrak{m} ) .
\nonumber
\end{eqnarray}
\item For each node $\mathfrak{m}$ of $\mathrm{Tree} ( V, \mathcal{A} )$ \
with $\mathrm{level} ( \mathfrak{m} ) =t$, either there exists a node
$\mathfrak{m}_{1}$ of $\mathrm{Tree} ( V,\mathcal{A}_{1} )$ with level $t$
such that
\begin{eqnarray}
( w ( \mathfrak{m}_{1} ) ,s ( \mathfrak{m}_{1} ) , \tilde{\mathcal{P}} (
\mathfrak{m}_{1} ) , \tilde{\mathcal{Q}} ( \mathfrak{m}_{1} ) ) & = & (
w ( \mathfrak{m} ) ,s ( \mathfrak{m} ) , \tilde{\mathcal{P}} (
\mathfrak{m} ) , \tilde{\mathcal{Q}} ( \mathfrak{m} ) ) , 
\label{eqn:monotone2}\\
\mathcal{A} ( \mathfrak{m}_{1} ) & = & \mathcal{A} ( \mathfrak{m} ) ,
\nonumber
\end{eqnarray}
or there exists a node $\mathfrak{m}_{2}$ of $\mathrm{Tree} ( V,
\mathcal{A}_{2} )$ with level $t$ such that
\begin{eqnarray}
( w ( \mathfrak{m}_{2} ) ,s ( \mathfrak{m}_{2} ) , \tilde{\mathcal{P}} (
\mathfrak{m}_{2} ) , \tilde{\mathcal{Q}} ( \mathfrak{m}_{2} ) ) & = & (
w ( \mathfrak{m} ) ,s ( \mathfrak{m} ) , \tilde{\mathcal{P}} (
\mathfrak{m} ) , \tilde{\mathcal{Q}} ( \mathfrak{m} ) ) , 
\label{eqn:monotone3}\\
\mathcal{A} ( \mathfrak{m}_{2} ) & = & \mathcal{A} ( \mathfrak{m} ) ,
\nonumber
\end{eqnarray}
or there exists a node $\mathfrak{m}_{1}$ of $\mathrm{Tree} (
V,\mathcal{A}_{1} )$ and a node $\mathfrak{m}_{2}$ of $\mathrm{Tree} ( V,
\mathcal{A}_{2} )$, both with level $t$ such that
\begin{eqnarray}
( w ( \mathfrak{m}_{1} ) ,s ( \mathfrak{m}_{1} ) , \tilde{\mathcal{P}} (
\mathfrak{m}_{1} ) , \tilde{\mathcal{Q}} ( \mathfrak{m}_{1} ) ) & = & (
w ( \mathfrak{m}_{2} ) ,s ( \mathfrak{m}_{2} ) , \tilde{\mathcal{P}} (
\mathfrak{m}_{2} ) , \tilde{\mathcal{Q}} ( \mathfrak{m}_{2} ) ) 
\label{eqn:monotone4}\\
& = & ( w ( \mathfrak{m} ) ,s ( \mathfrak{m} ) , \tilde{\mathcal{P}} (
\mathfrak{m} ) , \tilde{\mathcal{Q}} ( \mathfrak{m} ) ) , \nonumber\\
\mathcal{A} ( \mathfrak{m}_{1} ) \cup\mathcal{A} ( \mathfrak{m}_{2}
) & = & \mathcal{A} ( \mathfrak{m} ) . \nonumber
\end{eqnarray}
\end{enumerate}
The base case is when $t=0$, and in this case the claim is obviously true.
We now prove the inductive step from $t-1$ to $t$.
\begin{enumerate}
\item Suppose that the node $\mathfrak{m}$ is a child of $\mathfrak{n}$.
Since, $\mathrm{level} ( \mathfrak{n} ) =t-1$, by induction hypothesis there
exists a node $\mathfrak{n}'$ in $\mathrm{Tree} ( V, \mathcal{A} )$ with
level $t-1$ such that
\begin{eqnarray*}
( w ( \mathfrak{n}' ) ,s ( \mathfrak{n}' ) , \tilde{\mathcal{P}} (
\mathfrak{n}' ) , \tilde{\mathcal{Q}} ( \mathfrak{n}' ) ) & = & ( w (
\mathfrak{n} ) ,s ( \mathfrak{n} ) , \tilde{\mathcal{P}} ( \mathfrak{n}
) , \tilde{\mathcal{Q}} ( \mathfrak{n} ) ) ,\\
\mathcal{A} ( \mathfrak{n}' ) & \supset & \mathcal{A} ( \mathfrak{n} ) .
\end{eqnarray*}
The existence of a child $\mathfrak{m}'$ of the node $\mathfrak{n}'$
satisfying (\ref{eqn:monotone1}) is now clear from the definition of
$\mathrm{Tree} ( V, \mathcal{A} )$ (following the description given in
Section \ref{subsec:tree}). This completes the induction in this case.

\item Suppose again that the node $\mathfrak{m}$ is a child of
$\mathfrak{n}$. Since, $\mathrm{level} ( \mathfrak{n} ) =t-1$, by induction
hypothesis there exists either a node $\mathfrak{n}_{1}$ of $\mathrm{Tree} (
V,\mathcal{A}_{1} )$ with level $t-1$ such that
\begin{eqnarray}
( w ( \mathfrak{n}_{1} ) ,s ( \mathfrak{n}_{1} ) , \tilde{\mathcal{P}} (
\mathfrak{n}_{1} ) , \tilde{\mathcal{Q}} ( \mathfrak{n}_{1} ) ) & = & (
w ( \mathfrak{n} ) ,s ( \mathfrak{n} ) , \tilde{\mathcal{P}} (
\mathfrak{n} ) , \tilde{\mathcal{Q}} ( \mathfrak{n} ) ) , 
\label{eqn:monotone2a}\\
\mathcal{A} ( \mathfrak{n}_{1} ) & = & \mathcal{A} ( \mathfrak{n} ) ,
\nonumber
\end{eqnarray}
or there exists a node $\mathfrak{n}_{2}$ of $\mathrm{Tree} ( V,
\mathcal{A}_{2} )$ with level $t-1$ such that
\begin{eqnarray}
( w ( \mathfrak{n}_{2} ) ,s ( \mathfrak{n}_{2} ) , \tilde{\mathcal{P}} (
\mathfrak{n}_{2} ) , \tilde{\mathcal{Q}} ( \mathfrak{n}_{2} ) ) & = & (
w ( \mathfrak{n} ) ,s ( \mathfrak{n} ) , \tilde{\mathcal{P}} (
\mathfrak{n} ) , \tilde{\mathcal{Q}} ( \mathfrak{n} ) ) , 
\label{eqn:monotone3a}\\
\mathcal{A} ( \mathfrak{n}_{2} ) & = & \mathcal{A} ( \mathfrak{n} ) ,
\nonumber
\end{eqnarray}
or there exists a node $\mathfrak{n}_{1}$ of $\mathrm{Tree} (
V,\mathcal{A}_{1} )$ and a node $\mathfrak{n}_{2}$ of $\mathrm{Tree} ( V,
\mathcal{A}_{2} )$ both with level $t-1$ such that
\begin{eqnarray}
( w ( \mathfrak{n}_{1} ) ,s ( \mathfrak{n}_{1} ) , \tilde{\mathcal{P}} (
\mathfrak{n}_{1} ) , \tilde{\mathcal{Q}} ( \mathfrak{n}_{1} ) ) & = & (
w ( \mathfrak{n}_{2} ) ,s ( \mathfrak{n}_{2} ) , \tilde{\mathcal{P}} (
\mathfrak{n}_{2} ) , \tilde{\mathcal{Q}} ( \mathfrak{n}_{2} ) )
\nonumber\\
& = & ( w ( \mathfrak{n} ) ,s ( \mathfrak{n} ) , \tilde{\mathcal{P}} (
\mathfrak{n} ) , \tilde{\mathcal{Q}} ( \mathfrak{n} ) ) , 
\label{eqn:monotone4a}\\
\mathcal{A} ( \mathfrak{n}_{1} ) \cup\mathcal{A} ( \mathfrak{n}_{2}
) & = & \mathcal{A} ( \mathfrak{n} ) . \nonumber
\end{eqnarray}
It follows from the description of $\mathrm{Tree} ( V,\mathcal{A}_{1} )$
given in Section \ref{subsec:tree} that\eqref{eqn:monotone2} implies
that there exists a child $\mathfrak{m}_{1}$ of the node
$\mathfrak{n}_{1}$ satisfying  \eqref{eqn:monotone2a}. Similarly, it
is clear that \eqref{eqn:monotone3}  implies that there exists a child
$\mathfrak{m}_{2}$ of the node $\mathfrak{n}_{2}$ satisfying 
\eqref{eqn:monotone3a}. Now suppose that \eqref{eqn:monotone4} hold.
If $\mathfrak{m}$ is a left child of $\mathfrak{n}$, then clearly there
exists a child $\mathfrak{m}_{1}$ of the node $\mathfrak{n}_{1}$, and a
child $\mathfrak{m}_{2}$ of the node $\mathfrak{n}_{2}$ satisfying 
\eqref{eqn:monotone4a}. If $\mathfrak{m}$ is a right child of
$\mathfrak{n}$, then there are several cases.
\begin{enumerate}
\item Case $w ( \mathfrak{m} ) \in \pi_{[ 1, \mathrm{Fix} ( \mathfrak{m} )
]} ( \mathcal{A} ( \mathfrak{n}_{1} ) \cap \mathcal{A} (
\mathfrak{n}_{2} ) )$: In this case there exists a child
$\mathfrak{m}_{1}$ of the node $\mathfrak{n}_{1}$, and a child
$\mathfrak{m}_{2}$ of the node $\mathfrak{n}_{2}$ satisfying
\eqref{eqn:monotone4}.

\item Case $w ( \mathfrak{m} ) \in \pi_{[ 1, \mathrm{Fix} ( \mathfrak{m} )
]} ( \mathcal{A} ( \mathfrak{n}_{1} ) \setminus \mathcal{A} (
\mathfrak{n}_{1} ) \cap \mathcal{A} ( \mathfrak{n}_{2} ) )$: In this
case there exists a child $\mathfrak{m}_{1}$ of the node
$\mathfrak{n}_{1}$ satisfying \eqref{eqn:monotone2}.

\item Case $w ( \mathfrak{m} ) \in \pi_{[ 1, \mathrm{Fix} ( \mathfrak{m} )
]} ( \mathcal{A} ( \mathfrak{n}_{2} ) \setminus \mathcal{A} (
\mathfrak{n}_{1} ) \cap \mathcal{A} ( \mathfrak{n}_{2} ) )$: In this
case there exists a child $\mathfrak{m}_{2}$ of the node
$\mathfrak{n}_{2}$ satisfying \eqref{eqn:monotone3}.
\end{enumerate}
This completes the induction in this case.
\end{enumerate}
The proposition follows by applying the result proved above to the leaf
nodes of the trees $\mathrm{Tree} ( V,\mathcal{A}_{1} )$, $\mathrm{Tree} ( V,
\mathcal{A}_{2} )$ and $\mathrm{Tree} ( V, \mathcal{A} )$.
\end{proof}

\subsubsection{Roadmap associated to $\mathrm{Tree} ( V,\mathcal{A} )$}

We now prove that the union of the (limits of the) sets contained in the
leaves of $\mathrm{Tree} ( V,\mathcal{A} )$ form a roadmap. Most of this section
is devoted to the proof of the following theorem, which is the key result
needed to prove the correctness of our algorithms.

\begin{theorem}
\label{thm:correctness}The semi-algebraic set
\[ \mathrm{DCRM} ( V, \mathcal{A} ) \assign \bigcup_{\mathfrak{l} \in
 \mathrm{Leav} ( V,\mathcal{A} )}\lim_{\zeta_{1}} (
 \mathrm{Bas} ( \mathfrak{l} ) ) \]
contains $\mathcal{A}$, and is a roadmap of $V$.
\end{theorem}

Theorem \ref{thm:correctness} will follow from the following more general
proposition.

\begin{proposition}
\label{prop:correctness}Let $\mathfrak{n}$ be a node in $\mathrm{Tree} (
V,\mathcal{A} )$ of level $t$. Then, the semi-algebraic set
\[ \mathrm{DCRM} ( \mathrm{Bas} ( \mathfrak{n} ) , \mathcal{A ( \mathfrak{n} )}
 ) \assign \bigcup_{\mathfrak{l} \in \mathrm{Leav} ( \mathfrak{n} )}
 \lim_{\zeta_{t+1}} ( \mathrm{Bas} ( \mathfrak{l} ) ) \]

contains $\mathcal{A} ( \mathfrak{n} )$, and is a roadmap of $\mathrm{Bas} (
\mathfrak{n} )$.
\end{proposition}

Several intermediate results will be used in the proof of Proposition
\ref{prop:correctness} and Theorem \ref{thm:correctness}.

The following relation defined on elements of $\{ 0,1 \}^{t}$ will be used to
define a notion of ``neighbor'' amongst the leaf nodes of $\mathrm{Tree} (
V,\mathcal{A} )$ which in turn will be used to prove the existence of
connecting paths in the roadmap of $V$ defined by $\mathrm{Tree} ( V,\mathcal{A}
)$ having some extra structure.

\begin{definition}
\label{def:neighbor}We define a symmetric and reflexive relation $N_{t}$ on
elements of $\{ 0,1 \}^{ t}$ by induction on $t$ as follows.
\begin{enumerate}
\item If $t=1$, $0N_{1} 1$.

\item For all $s,s' \in \{ 0,1 \}^{t}$, $sN_{t} s'$ implies that
$0sN_{t+1} 0s' $, and $1sN_{t+1} 1s$.

\item Finally, $01^{t-1} N_{t} 1^{t}$ for all $t \geq 1$.
\end{enumerate}
\end{definition}

\begin{remark}
\label{rem:treeofstrings}The relation $N_{t}$ defined in Definition
\ref{def:neighbor} induces the structure of a tree on the set $\{ 0,1
\}^{t}$. This tree in the case $t=4$ is displayed in Figure \ref{fig:tree}. 
The edges in the tree correspond to pairs of elements $s,t \in \{ 0,1
\}^{4}$, with $s N_{4}t$.

\begin{figure}
\vspace{-2.5in}
\includegraphics[scale=0.4]{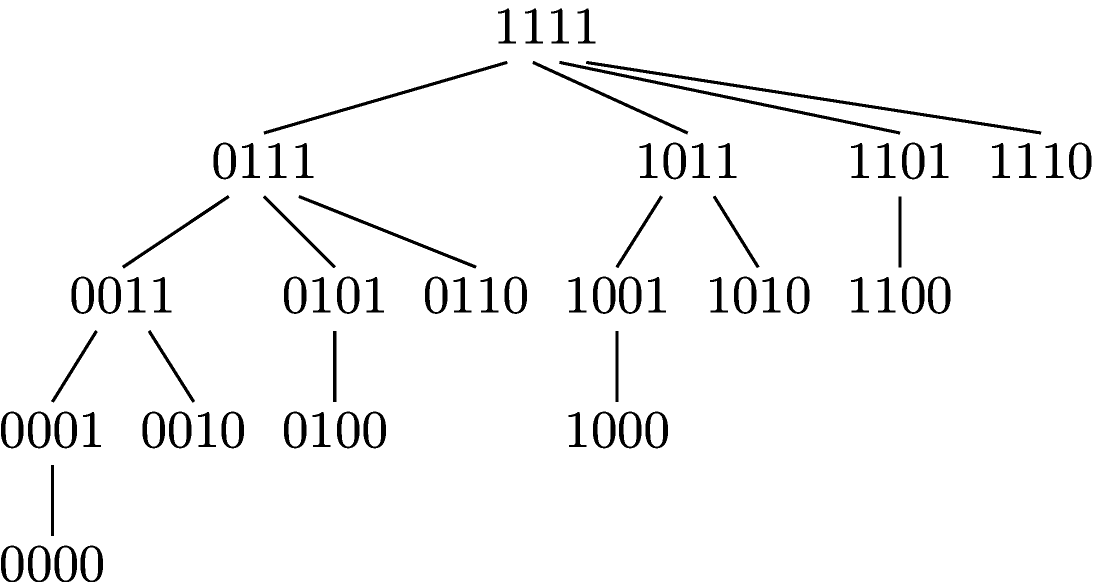}
\caption{Tree of leaves }
\label{fig:tree}
\end{figure}
\end{remark}

The following proposition which uses the relation defined in Definition
\ref{def:neighbor} above will be used to prove the existence of connecting
paths in the roadmap. These connecting paths will have a certain special
structure -- and this structure will be defined using the relation defined in
Definition \ref{def:neighbor}.

\begin{proposition}
\label{prop:structure2}Let $\mathfrak{n}$ be a node of $\mathrm{Tree} (
V,\mathcal{A} )$ with $\mathrm{level} ( \mathfrak{n} ) =t$,
$\mathfrak{l},\mathfrak{l}' \in \mathrm{Leav} ( \mathfrak{n} )$, and $x \in
\mathcal{A} ( \mathfrak{l} ) ,x' \in \mathcal{A} ( \mathfrak{l}' )$, such
that $\lim_{\zeta_{t+1}} ( x ) \in \mathrm{Cc} ( \lim_{\zeta_{t+1}} ( x' ) ,
\mathrm{Bas} ( \mathfrak{n} ) )$. Then, there exist $\mathfrak{\mathfrak{l}=
\mathfrak{\mathfrak{l}}}_{0} , \ldots ,\mathfrak{l}_{N} =\mathfrak{l}' \in
\mathrm{Leav} ( \mathfrak{n} )$, and for each $i,0 \leq i \leq N$,
$x_{2i} ,x_{2i+1} \in \mathcal{\mathcal{A}} ( \mathfrak{l}_{i} )$
 such that
\begin{enumerate}
\item $x_{0} =x$, $x_{2N+1} =x'$;

\item for all $i=1, \ldots ,N,$ $\lim_{\zeta_{t+1}} ( x_{2i-1} ) =
\lim_{\zeta_{t+1}} ( x_{2i} )$;

\item for all $i=0, \ldots ,N$, $x_{2i+1} \in \mathrm{Cc} ( x_{2i} ,
\mathrm{Bas} ( \mathfrak{\mathfrak{\mathfrak{l}}_{i}} ) )$;

\item for all $i=0, \ldots ,N-1,$ $s ( \mathfrak{l}_{i} ) N_{\log ( k' )}
s ( \mathfrak{l}_{i+1} )$.
\end{enumerate}
\end{proposition}

The proof of Proposition \ref{prop:structure2} will use the following lemma.

\begin{lemma}
\label{lem:intermediate}Let $\mathfrak{m}_{1} ,\mathfrak{m}_{2}$ be two
distinct children of a node $\mathfrak{n}$ of $\mathrm{Tree} ( V,\mathcal{A}
)$ with $\mathrm{level} ( \mathfrak{n} ) =t$, and for $i=1,2$, let $B_{i}$ be
a semi-algebraically connected component of $\mathrm{Bas} ( \mathfrak{m}_{i}
)$. Suppose that $\lim_{\gamma_{t+1}} ( B_{1} ) \cap \lim_{\gamma_{t+1}} (
B_{2} ) \neq \emptyset$. Then, there exists $\mathfrak{l}_{1} \in
\mathrm{Leav} ( \mathfrak{m}_{1} ) ,\mathfrak{l}_{2} \in \mathrm{Leav} (
\mathfrak{m}_{2} )$, and $x_{1} \in \mathcal{A} ( \mathfrak{l}_{1} ) ,x_{2}
\in \mathcal{A} ( \mathfrak{l}_{2} )$, such that
\begin{eqnarray*}
\lim_{\zeta_{t+1}} ( x_{1} ) & = & \lim_{\zeta_{t+1}} ( x_{2} ) ,\\
\lim_{\zeta_{t+2}} ( x_{i} ) & \in & B_{i} , \text{ for } i=1,2, 
\end{eqnarray*}
and
\[ s ( \mathfrak{l}_{1} )N_{\log ( k' )} s ( \mathfrak{l}_{2} ) . \]
\end{lemma}

\begin{proof}
There are four cases to consider.
\begin{enumerate}
\item $\mathfrak{m}_{1}$ is a left child and $\mathfrak{m}_{2}$ a right
child of $\mathfrak{n}$. Let $\mathfrak{m}_{1} =\mathfrak{n} ( \alpha )$
for some $\alpha \in \mathcal{I} ( \mathfrak{n} )$. Since $\mathcal{B} (
\mathfrak{n} ) = \widetilde{\mathrm{Bas}} ( \mathfrak{n} )^{0} \cap
\widetilde{\mathrm{Bas}} ( \mathfrak{n} )^{1}$, and $\lim_{\gamma_{t+1}} (
B_{2} ) \cap \lim_{\gamma_{t+1}} ( B_{1} ) \neq \emptyset$, there
exists a point $x \in \lim_{\gamma_{t+1}} ( B_{2} ) \cap
\lim_{\gamma_{t+1}} ( B_{1} ) \subset \mathcal{B ( \mathfrak{n} )}$.
Moreover $x \in \lim_{\gamma_{t+1}} ( \mathcal{A} ( \alpha ) \cap
B_{1} )$ by definition of $\mathcal{A} ( \alpha )$ (see Notation 50),
since $\mathcal{B} ( \mathfrak{n} ) = \widetilde{\mathrm{Bas}} (
\mathfrak{n} )^{0} \cap \widetilde{\mathrm{Bas}} ( \mathfrak{n} )^{1}$ is
finite. Moreover $x\in \mathcal{B ( \mathfrak{n} ) \cap}
\lim_{\gamma_{t+1}} ( B_{2} ) \subset \mathcal{B ( \mathfrak{n} ) \cap}
\lim_{\gamma_{t+1}} ( \mathrm{Bas} ( \mathfrak{m}_{2} ) ) \subset
\lim_{\gamma_{t+1}} ( \mathcal{A} ( \mathfrak{m}_{2} ) )$ using 
\eqref{eqn:Aofm}.

Using \eqref{eqn:rightmostB}, there exist for $i=1,2 
,\mathfrak{l}_{i} \in \mathrm{Leav} ( \mathfrak{m}_{i} )$, and $x_{i} \in
\mathcal{A} ( \mathfrak{l}_{i} )$, with $\lim_{\gamma_{t+1}} ( x_{i} )
=x$. Then, $\lim_{\zeta_{t+1}} ( x_{i} ) = \lim_{\zeta_{t+1}} ( x )$, and
$\lim_{\zeta_{t+2}} ( x_{i} ) \in B_{i}$ and
\begin{equation}
s ( \mathfrak{l}_{i} ) =s ( \mathfrak{m}_{i} ) 1 \cdots 1 \label{eqn:s}
.
\end{equation}
Since in this case, $s ( \mathfrak{m}_{1} ) =s ( \mathfrak{n} ) 0$, and
$s ( \mathfrak{m}_{2} ) =s ( \mathfrak{n} ) 1$, it follows from Definition
\ref{def:neighbor} and \eqref{eqn:s} that $s ( \mathfrak{l}_{1} )
N_{\log ( k' )} s ( \mathfrak{l}_{2} )$.

\item $\mathfrak{m}_{1}$ is a right child and $\mathfrak{m}_{2}$ a left
child of $\mathfrak{n}$. This case is similar to the one above with the
roles of $\mathfrak{m}_{1}$ and $\mathfrak{m}_{2}$ reversed.

\item Both $\mathfrak{m}_{1} ,\mathfrak{m}_{2}$ are right children of
$\mathfrak{n}$. In this case, $ \mathrm{Bas} ( \mathfrak{m}_{1}
 ) \cap \mathrm{Bas} ( \mathfrak{m}_{2} ) = \emptyset$, and hence
$\lim_{\gamma_{t+1}} ( \mathrm{Bas} ( \mathfrak{m}_{1} ) ) \cap
\lim_{\gamma_{t+1}} ( \mathrm{Bas} ( \mathfrak{m}_{2} ) ) = \emptyset$,
since the descriptions of $ \mathrm{Bas} ( \mathfrak{m}_{1}
 )$ and $\mathrm{Bas} ( \mathfrak{m}_{2} )$  do not depend on
$\gamma_{t+1}$. Thus, there is nothing to prove in this case.

\item Both $\mathfrak{m}_{1} ,\mathfrak{m}_{2}$ are left children of
$\mathfrak{n}$. In this case there exists $\alpha_{1} , \alpha_{2} \in
\mathcal{I} ( \mathfrak{n} )$, such that for $i=1,2$, $\mathfrak{m}_{i}
=\mathfrak{n} ( \alpha_{i} )$. In this case there exists for $i=1,2$,
$x'_{i} \in \mathcal{A} ( \alpha_{i} )$, such that $\lim_{\gamma_{t+1}} (
x'_{1} ) = \lim_{\gamma_{t+1}} ( x'_{2} ) \in \lim_{\gamma_{t+1}} (
B_{1} ) \cap \lim_{\gamma_{t+1}} ( B_{2} )$ (using 
Proposition \ref{prop:alpha-beta}).
Moreover, using \eqref{eqn:rightmostA}, there exists for $i=1,2$, \
$\mathfrak{l}_{i} \in \mathrm{Leav}^{1} ( \mathfrak{m}_{i} )$, and $x_{i}
\in \mathcal{A} ( \mathfrak{l}_{i} )$, such that $x'_{i} =
\lim_{\zeta_{t+2}} ( x_{i} )$. Notice that, for $i=1,2$, $s (
\mathfrak{l}_{i} ) =s ( \mathfrak{m}_{i} ) 1 \cdots 1$. It is now easy to
check that $s ( \mathfrak{l}_{1} ) =s ( \mathfrak{l}_{2} )$, and that the
tuple $( x_{1} ,\mathfrak{l}_{1} ,x_{2} ,\mathfrak{l}_{2} )$ then
satisfies the required properties.
\end{enumerate}
\end{proof}

\begin{proof}[Proof of Proposition \ref{prop:structure2}]
The proof of the proposition  is by induction
on $t= \mathrm{level} ( \mathfrak{n} )$. The base case is when $\mathfrak{n}$
is a leaf node, in which case the statement clearly holds. Otherwise,
suppose that the proposition is true for all nodes having level greater than $t$.

Using Corollary \ref{cor:propertyofAtilde}, we can assume without loss of
generality that
\[ \lim_{\gamma_{t+1}} ( x ) \in \mathrm{Cc} ( \lim_{\gamma_{t+1}} ( x' ) ,
 \widetilde{\mathrm{Bas}} ( \mathfrak{n} ) ) . \]

Since by 
Proposition \ref{prop:good-connectivity},
$\widetilde{\mathrm{Bas}} (
\mathfrak{n} )^{0} \cup \widetilde{\mathrm{Bas}} ( \mathfrak{n} )^{1}$ has
good connectivity property with respect to $\widetilde{\mathrm{Bas}} (
\mathfrak{n} )$, and since 
\[
\bigcup_{\mathfrak{l}\text{ child of  } \mathfrak{n}} \lim_{\gamma_{t+1}} (\mathcal{A} (
\mathfrak{l} ) )\subset \widetilde{\mathrm{Bas}} ( \mathfrak{n} )^{0}
\cup \widetilde{\mathrm{Bas}} ( \mathfrak{n} )^{1},
\]
  it follows that
$\lim_{\gamma_{t+1}} ( x' ) \in \mathrm{Cc} ( \lim_{\gamma_{t+1}} ( x ) ,
\widetilde{\mathrm{Bas}} ( \mathfrak{n} )^{0} \cup \widetilde{\mathrm{Bas}} (
\mathfrak{n} )^{1} )$. So there exists a sequence
$\mathfrak{m}=\mathfrak{m}_{0} , \ldots ,\mathfrak{m}_{n} =\mathfrak{m}'$ of
children of $\mathfrak{n}$, for each $i, 0 \leq i \leq n
$, a semi-algebraically connected component $B_{i }$ of $\mathrm{Bas}
( \mathfrak{m}_{i} )$, with 
\[B_{0} = \mathrm{Cc} ( \lim_{\zeta_{t+2}} ( x ),
\mathrm{Bas} ( \mathfrak{m}_{0} ) ) ,B_{n} = \mathrm{Cc} ( \lim_{\zeta_{t+2}} (
x' ) , \mathrm{Bas} ( \mathfrak{m}_{n} ) ),
\] 
and for each $i, 0 \leq i \leq
n-1$, $\lim_{\gamma_{t+1}} ( B_{i} )\cap \lim_{\gamma_{t+1}}
( B_{i+1} ) \neq \emptyset $.

Applying Lemma \ref{lem:intermediate} we have that for each $i,0\leq i
\leq n $, there exists $\mathfrak{l}_{2i} ,\mathfrak{l}_{2i+1} \in
\mathrm{Leav} ( \mathfrak{m}_{i} )$, and for each $j,1 \leq j \leq 2n$,
$\bar{x}_{j} \in \mathcal{A} ( \mathfrak{l}_{j} )$, such that, for each $i,0
\leq i \leq n-1,$ \
\begin{eqnarray*}
\lim_{\zeta_{t+1}} ( \bar{x}_{2i+1} ) & = & \lim_{\zeta_{t+1}} (
\bar{x}_{2i+2} ) ,
\end{eqnarray*}
\[ \lim_{\zeta_{t+2}} ( \bar{x}_{2i} ) , \lim_{\zeta_{t+2}} ( \bar{x}_{2i+1}
 ) \in B_{i} , \]
and
\begin{equation}
s ( \mathfrak{l}_{i} )N_{\log ( k' )} s ( \mathfrak{l}_{i+1} ) .
\label{eqn:adjacent}
\end{equation}
Let also
\begin{eqnarray*}
\bar{x}_{0} & = & x,\\
\bar{x}_{2n+1} & = & x' ,\\
\mathfrak{l}_{0} & = & \mathfrak{l},\\
\mathfrak{l}_{2n+1} & = & \mathfrak{l}' .
\end{eqnarray*}
We now apply the induction hypothesis to each of the pairs $\bar{x}_{2i} ,
\bar{x}_{2i+1}$,
$0 \leq i \leq n$, and use \eqref{eqn:adjacent} to complete the
induction.
\end{proof}

\begin{proposition}
\label{prop:structure1}Let $\mathfrak{n}$ be a node of the tree $\mathrm{Tree}
( V,\mathcal{A} )$, with $\mathrm{level} ( \mathfrak{n} ) =t $, and
$\mathrm{Leav} ( \mathfrak{n} )$ be the set of leaves of the sub-tree of
$\mathrm{Tree} ( V,\mathcal{A} )$ rooted at $\mathfrak{n}$. For any two leaves
$\mathfrak{l}$ and $\mathfrak{l}'$ in $\mathrm{Leav} ( \mathfrak{n} )$ and any
two points $x \in \lim_{\zeta_{t+1}} ( \mathrm{Bas} ( \mathfrak{l} ) ), x' \in
\lim_{\zeta_{t+1}} ( \mathrm{Bas} ( \mathfrak{l}' ) )$, such that $x' \in
\mathrm{Cc} ( x, \mathrm{Bas} ( \mathfrak{n} ) )$, there exists a semi-algebraic
path $\Gamma$ connecting $x$ to $x'$, such that:
\begin{enumerate}
\item $\Gamma$ is a concatenation of semi-algebraic paths $\Gamma_{0} ,
\ldots , \Gamma_{m}$, where each $\Gamma_{i} \subset \lim_{\zeta_{t+1}} (
\mathrm{Bas} ( \mathfrak{l}_{i} ) )$, for some $\mathfrak{l}_{i} \in
\mathrm{Leav} ( \mathfrak{n} )$;

\item for each $i,0 \leq i \leq m-1$, $s ( \mathfrak{l}_{i} )
N_{\log ( k' )}s ( \mathfrak{l}_{i+1} )$.
\end{enumerate}
\end{proposition}

\begin{proof}
Immediate consequence of Proposition \ref{prop:structure2}.
\end{proof}

The following two propositions will be used in the proof of Proposition \ref{prop:correctness}.
\begin{proposition}
\label{prop:leaves2} Let $\mathfrak{n}$ be a node of the tree $\mathrm{Tree} (
V,\mathcal{A} )$, with $\mathrm{level} ( \mathfrak{n} ) =t$, and let
$\mathrm{Leav}^{ 0} ( \mathfrak{n} )$ be the set of leaves $\mathfrak{m}$ of
the subtree of $\mathrm{Tree} ( V,\mathcal{A} )$ rooted at $\mathfrak{n}$,
such that $s ( \mathfrak{m} )$ contains no $1$ to the right of $s (
\mathfrak{n} )$. Then, the semi-algebraic set
\begin{eqnarray*}
L & = & \bigcup_{\mathfrak{m} \in \mathrm{Leav}^{0} ( \mathfrak{n} )}
\lim_{\zeta_{t+1}} ( \mathrm{Bas} ( \mathfrak{m} ) )
\end{eqnarray*}
is such that for all $x \in \R_{t}$, $L_{( w ( \mathfrak{n} ) ,x )}$ meets
every semi-algebraically connected component of $\mathrm{Bas} ( \mathfrak{n}
)_{( w ( \mathfrak{n} ) ,x )}$.
\end{proposition}

\begin{proof}
 The proof is by induction on $t= \mathrm{level} ( \mathfrak{n} )$. If
$\mathfrak{n}$ is a leaf node with $| s ( \mathfrak{n} ) | =0$, then
$\mathrm{Leav}^{ 0} ( \mathfrak{n} ) = \{ \mathfrak{n} \}$ and there is
nothing to prove. Now assume that the proposition is true for all
$\mathfrak{n}'$, with $\mathrm{level} ( \mathfrak{n}' ) >t $.

Note that the left children of $\mathfrak{n}$ are precisely those children
$\mathfrak{m}$ of $\mathfrak{n}$ with $s ( \mathfrak{m} ) = ( s (
\mathfrak{n} ) ,0 )$, and these are in 1-1 correspondence with $\alpha \in
\mathcal{I} ( \mathfrak{n} )$. Denote by $\mathfrak{n} ( \alpha )$ the left
child of $\mathfrak{n}$ corresponding to $\alpha \in \mathcal{I} (
\mathfrak{n} )$ .

We denote 
(with a slight abuse of notation)
$\Ext \left( \widetilde{\mathrm{Bas}}
( \mathfrak{n} )^{0} , \R_{t+1} \right)$ by $\widetilde{\mathrm{Bas}} (
\mathfrak{n} )^{0}$ and $\Ext \left( \widetilde{\mathrm{Bas}} ( \mathfrak{n}
) , \R_{t+1} \right)$ by $\widetilde{\mathrm{Bas}} ( \mathfrak{n} )$, and make the
following claims.

{\noindent}1. For each $w \in \R_{t+1}^{\ell}$, where $\ell = k'/2^{t+1}$, $\widetilde{\mathrm{Bas}} ( \mathfrak{n} )^{0}_{( w ( \mathfrak{n}
) ,w )}$ meets every semi-algebraically connected component of
$\widetilde{\mathrm{Bas}} ( \mathfrak{n} )_{( w ( \mathfrak{n} ) ,w )}$
(Proposition \ref{prop:properties-of-S0}). It follows immediately (since
$\ell \geq 1$) that for each $x' \in \R_{t+1}$, $\widetilde{\mathrm{Bas}} (
\mathfrak{n} )^{0}_{( w ( \mathfrak{n} ) ,x' )}$ meets every
semi-algebraically connected component of $\widetilde{\mathrm{Bas}} (
\mathfrak{n} )_{( w ( \mathfrak{n} ) ,x' )}$.

{\noindent}2. Also, $\lim_{\zeta_{t+1}} ( \widetilde{\mathrm{Bas}} (
\mathfrak{n} ) ) = \mathrm{Bas} ( \mathfrak{n} ) \subset \R_{t}^{k}$
(Proposition \ref{prop:limtilde}). It follows that for any $x \in
\R_{t}$, and $C$ a semi-algebraically connected component of $\mathrm{Bas} (
\mathfrak{n} )_{( w ( \mathfrak{n} ) ,x )}$, there exists $x' \in \R_{t+1}$
with $\lim_{\zeta_{t+1}} ( x' ) =x$, and a semi-algebraically connected
component $D$ of $\widetilde{\mathrm{Bas}} ( \mathfrak{n} )_{( w (
\mathfrak{n} ) ,x' )}$ such that $\lim_{\zeta_{t+1}} ( D ) \subset C$.

{\noindent}3. Using Claim 1. there exists a semi-algebraically connected
component $D^{0}$ of $\widetilde{\mathrm{Bas}} ( \mathfrak{n} )_{( w (
\mathfrak{n} ) ,x' )}^{0}$ which is contained in $D$.

Now since,
\[ \widetilde{\mathrm{Bas}} ( \mathfrak{n} )^{0} = \bigcup_{\alpha \in
 \mathcal{I} ( \mathfrak{n} )} \lim_{\gamma_{t+1}} \mathrm{Bas} (
 \mathfrak{n} ( \alpha ) ) \]
there exists a left child $\mathfrak{n} ( \alpha )$ of $\mathfrak{n}$ and a
semi-algebraically connected component $D_{\mathfrak{n} ( \alpha )}$ of
$\mathrm{Bas} ( \mathfrak{n} ( \alpha ) )$ such that $\lim_{\gamma_{t+1}} (
D_{\mathfrak{n} ( \alpha )} ) \subset D^{0}$.

Noting that being the left child of $\mathfrak{n}$, $\mathrm{level} (
\mathfrak{n} ( \alpha ) ) > \mathrm{level} ( \mathfrak{n} )$, and noting that
the fact $s ( \mathfrak{n} ( \alpha ) ) =s ( \mathfrak{n} ) 0$ implies that
$\mathrm{Fix} ( \mathfrak{n} ( \alpha ) ) = \mathrm{Fix} ( \mathfrak{n} )$, we
can apply the induction hypothesis to obtain that $L'_{( w ( \mathfrak{n} )
,x' )}$ meets $D_{\mathfrak{n} ( \alpha )}$, where
\[ L' = \bigcup_{\mathfrak{m} \in \mathrm{Leav}^{ 0} ( \mathfrak{n} ( \alpha
 ) )} \lim_{\zeta_{t+2}} ( \mathrm{Bas} ( \mathfrak{m} ) ) . \]
Now $\lim_{\zeta_{t+1}} ( L' ) \subset L$, which implies that
$\lim_{\zeta_{t+1}} ( L'_{( w ( \mathfrak{n} ) ,x' )} ) \subset L_{( w (
\mathfrak{n} ) ,x )}$. Moreover, $L'_{( w ( \mathfrak{n} ) ,x )} \cap
D_{\mathfrak{n} ( \alpha )} \neq \emptyset$, $\lim_{\gamma_{t+1}} (
D_{\mathfrak{n} ( \alpha )} ) \subset D^{0} \subset D$, and
$\lim_{\zeta_{t+1}} ( D ) \subset C$. Together they imply that $L_{( w (
\mathfrak{n} ) ,x )} \cap C \neq \emptyset$.
\end{proof}

\begin{corollary}
\label{cor:meetscc}$\mathrm{DCRM} ( \mathrm{Bas} ( \mathfrak{n} ) , \mathcal{A} (
\mathfrak{n}) )$ meets every semi-algebraically connected component of
$\mathrm{Bas} ( \mathfrak{n} )$.
\end{corollary}

We are now ready for the proof of Proposition \ref{prop:correctness}, and as
an immediate consequence Theorem \ref{thm:correctness}.

\begin{proof}[Proof of Proposition \ref{prop:correctness}]
The fact that
$\mathcal{A} ( \mathfrak{n} )$ is contained in the set 
\[
\mathrm{DCRM} (
\mathrm{Bas} ( \mathfrak{n} ) ,\mathcal{A} ( \mathfrak{n} ) ) =
\bigcup_{\mathfrak{l} \in \mathrm{Leav} ( \mathfrak{n} )} 
\lim_{\zeta_{t+1}} ( \mathrm{Bas} ( \mathfrak{l} ) )
\]
 follows from \eqref{eqn:rightmostA}. The roadmap property $\mathrm{RM}_{1}$ follows from
Proposition \ref{prop:structure1} and Corollary \ref{cor:meetscc}. The
property $\mathrm{RM}_{2}$ follows from Proposition \ref{prop:leaves2}.
\end{proof}

\begin{proof}[Proof of Theorem \ref{thm:correctness}]
Follows immediately from Proposition
\ref{prop:correctness}, setting $\mathfrak{n}=\mathfrak{r}$, and observing 
that $\mathcal{A}=\mathcal{A} ( \mathfrak{r} )$ by construction, and that
$\mathrm{Fix} ( \mathfrak{r} ) =0$.
\end{proof}

\subsection{Preliminary definitions and algorithms}\label{prelim}

In this subsection we introduce certain notation, definitions and algorithms
that will be used in Algorithm \ref{algo:divide} (Divide) in the next
subsection. Recall that in the description of the tree $\mathrm{Tree} (
V,\mathcal{A} )$, at each node $\mathfrak{n}$ of $\mathrm{Tree} ( V,\mathcal{A}
)$, some coordinates have been fixed and the basic semi-algebraic set
$\mathrm{Bas} ( \mathfrak{n} )$ is contained in the fiber over the point $w (
\mathfrak{n} )$ consisting of the fixed coordinates. We now explain how we
represent algebraically the points that fix the fibers in our construction,  and
also the necessary algorithms to compute these points. We refer the reader to
{\cite{BPRbook2}} for any missing detail.

A root of a univariate polynomials is going to be described by a Thom
encoding.

\begin{notation}
\label{not:derivatives} Let $P$ be a univariate polynomial of degree~$p$ in
$\D [X]$. We denote by $\mathrm{Der} ( P )$ the list $P,P' , \ldots ,P^{(p)}$.
Let $P \in \D [X]$ and $\sigma \in \{0,1, \um 1\}^{\Der (P)}$ a sign
condition on the set $\Der (P)$ of derivatives of $P$. The {\tmem{Thom
encoding}} of a root $x$ of $P$ in $\R$ is equal to $\sigma$ if the sign
condition taken by the set $\Der (P)$ at $x$ coincides with $\sigma$. Note
that two different roots of $P$ have different Thom encodings (see
{\cite{BPRbook2}} Proposition 2.28).
\end{notation}

Because we need to fix successively blocks of coordinates of decreasing size,
triangular Thom encodings appear naturally.

\begin{definition}
\label{def:triangular-thom}
A {\tmem{triangular system of polynomials}} with
variables $T= ( T_{1} , \ldots ,T_{t} )$ is a tuple $\mathcal{T=} ( F_{1} ,
\ldots ,F_{t} )$ where
\begin{eqnarray*}
& F_{i} \in & \D [T_{1} , \ldots ,T_{i} ],1 \leq i \leq t,
\end{eqnarray*}
such that $\ZZ ( \mathcal{T} , \R^{t} )$ is finite. 
A {\tmem{triangular Thom
encoding}} specifying 
\[
\theta = ( \theta_{1} , \ldots , \theta_{t} ) \in
\R^{t}
\] 
is a pair $( \mathcal{T} , \tau )$ where $\mathcal{T}$ is a
triangular system of polynomials,  and $\tau = \tau_{1} , \ldots , \tau_{t}$
is a list of Thom encodings, such that $\tau_{i}$ is the Thom encoding of
the real root $\theta_{i}$ of $F_{i} ( \theta_{1} , \cdots , \theta_{i-1}
,T_{i} )$, for $i=1, \ldots ,t$.
\end{definition}

Moreover, we need to describe points in the corresponding fibers, which is
done using real univariate representations.

\begin{definition}
\label{realuniv}
A \emph{$k$-real univariate representation $u$ over a
triangular Thom encoding $\mathcal{T}, \tau$} specifying $\theta
\in \R^{t}$ is of the form
\[ u= (f(T,U), \sigma ,F(T,U)) ,\]
where $f (T,U) ,F (T,U) = ( f_{0} ( T,U ) , \ldots ,f_{k} ( T,U ) )$ is a
$k+2$-tuple of polynomials in $\D [T,U]$, such that $f ( \theta ,U )$ and
$f_{0} ( \theta ,U )$ are co-prime, and $\sigma$ is the Thom encoding of a
real root $x$ of $f ( \theta ,U )$. The {\tmem{point associated}} to $u$ is
the point
\[ \left( \frac{f_{1} ( \theta ,x)}{f_{0} ( \theta ,x )} , \ldots ,
 \frac{f_{k} ( \theta ,x)}{f_{0} ( \theta ,x)} \right) \in \R^{k} . \]
For $1 \leq p \leq k$, we call the real univariate representation $u_{\leq p}
= (f(T,U), \sigma ,F_{\leq p} (T,U))$ where $F_{\leq p} (T,U) = ( f_{0} (
T,U ) , \ldots ,f_{p} ( T,U ) )$, over the initial real triangular Thom
encoding $\mathcal{T}, \tau$ to be the {\tmem{projection of $u$ to the first
$p$ coordinates}}. Geometrically this corresponds to forgetting the last
$k-p$ coordinates of the associated point.
\end{definition}

We now give a few auxiliary algorithms. The first one computes the limit of a
Thom Encoding and is used in the determination of the $( B,G
)$-pseudo-critical values needed in our construction.

In the following algorithm $\overline{\eps} = \left( \eps_{1} , \ldots ,
\eps_{t} \right)$ is a tuple of infinitesimals.

{\algorithm\label{alg:limitofaThomencoding}[Limit of a Thom Encoding]
\begin{itemize}
\item {\tmstrong{Input}}:  a Thom encoding $\left(
f_{\overline{\eps}} , \sigma_{\overline{\eps}} \right)$ ,
$f_{\overline{\eps}} \in \D \left[ \overline{\eps} ,U \right]$, representing
$x_{\overline{\eps}} \in \R_{t} \la \overline{\eps} \ra$
bounded over $\R$.

\item {\tmstrong{Output}}: a Thom encoding $( f_{,} \sigma )$ , $f \in \D
[ U ]$, representing
\[ x= \lim_{\eps_{1}} ( x_{\bar{\epsilon}} ) \in \R . \]
\item {\tmstrong{Complexity and degree bounds}}: If $D_{1}$ (resp. $D_{2}$)
is a bound on the degree of $f_{\overline{\eps}}$ with respect to $U$ (resp.
$\overline{\eps}$) the number of arithmetic operations in $\D$ is bounded by
$D_{1}^{O (1)} D_{2}^{O (t)}$. Moreover, the degrees in $U$ of the
polynomials appearing in the output are still bounded by $D_{1}$.

\item {\tmstrong{Procedure}}:

\item {\tmstrong{Step 1}}. Replace $f_{\overline{\eps}}$ by
$\overline{\eps}^{-o_{\overline{\eps}} \left( f_{\overline{\eps}} \right)} 
f_{\overline{\eps}}$ (see Notation \ref{not:order}). Denote by $f (T)$ the
polynomial obtained by substituting successively $\eps_{t}$ by $0$, and then
$\eps_{t-1}$ by 0, and so on, and finally $\eps_{1}$ by $0$, in
$f_{\overline{\eps}}$.

\item {\tmstrong{Step 2}}. Compute the set $\Sigma$ of Thom encodings of
roots of $f (T)$ using Algorithm 10.11 (Sign Determination) from
{\cite{BPRbook2}}.

\item {\tmstrong{Step 3}}. Identify the Thom encoding $\sigma$ using
Algorithm 10.13 (Univariate Sign Determination) from {\cite{BPRbook2}}, by
checking whether a ball of infinitesimal radius $\delta$ ($1 \gg \delta \gg
\overline{\eps} >0$) around the point $x$ represented by the real univariate
representation $f, \sigma$ contains $x_{\overline{\eps}}$.
\end{itemize}}

\begin{proof}[Proof of correctness] Follows immediately from the correctness of
Algorithm 10.11 (Sign Determination)  and
Algorithm 10.13 (Univariate Sign Determination) in \cite{BPRbook2}.
\end{proof}

\begin{proof}[Proof of complexity and degree bounds]
Follows from the complexity
of Algorithm 10.11 (Sign Determination)  and
Algorithm 10.13 (Univariate Sign Determination) in  \cite{BPRbook2}. The
fact that the degree in $U$ of the polynomials in the output are bounded by
$D_{1}$ is clear.
\end{proof}

\begin{remark}
\label{rem:complexityofring} Our algorithms use several algorithms from
{\cite{BPRbook2}} such as Algorithm 12.16 (Bounded Algebraic Sampling),
Algorithm 14.9 (Global Optimization), Algorithm 15.2 (Curve Segments), and
Algorithm 11.19 (Restricted Elimination) with one important modification.
Each of these algorithms described in {\cite{BPRbook2}} has an associated
structure which is an ordered domain in which all computations (i.e.,
arithmetic operations and sign evaluations) take place. In the calls to
these algorithms in this paper, this ordered domain will be of the form
$\D_{t} [ \theta ]$, where $\theta \in \R_{t}^{m}$ is specified by a
triangular Thom encoding $( \mathcal{T}, \tau )$ and involves $4t$
infinitesimals (see Notation \ref{not:rtdt}). Each element of $\D_{t} [
\theta ]$ is represented by some polynomial in $\D_{t} [ T ] = \D [ \eta
,T_{1} , \ldots ,T_{m} ]$ and arithmetic operations are performed as
ordinary polynomial arithmetic in the ring $\D_{t} [ T ]$. For the
evaluation of the sign of an element in $\D_{t} [ \theta ]$ represented by a
polynomial $f \in \D_{t} [ T ]$ we also use an algorithm from
{\cite{BPRbook2}}, namely Algorithm 12.10 (Triangular Sign Determination)
with input $f,\mathcal{T}, \tau$.

Suppose that the degree of the output (and of the intermediate
computations) of a particular algorithm in {\cite{BPRbook2}} is bounded by
some function $f ( d,k,s )$ of the degrees $d$, the number of variables $k$,
and the number of polynomials $s$. If $d' ,k' ,s'$ is a bound on the
degrees, number of variables and number of the input polynomials (considered
as polynomials with coefficients in $\D_{t} [ \theta ]$) in a call to that
algorithm in this paper, then the degree bound of the output (and
intermediate computations) is $f ( d' ,k' ,s' )$ in the ring $\D_{t} [
\theta ]$.

But we want to evaluate the complexity in the ring $\D$. Denoting by $N$ a
bound on the degrees in $T, \eta$ of the input polynomials, we have the
following :
\begin{itemize}
\item the degrees in $T, \eta$ of the output (and of the intermediate
computations) are bounded by $O ( Nf ( d' ,k' ,s' ) )$,

\item if the complexity of a particular algorithm in {\cite{BPRbook2}} is
bounded by some function $F ( d,k,s )$, then the number of arithmetic
operations and sign evaluations in $\D_{t} [ \theta ]$ of the call to that
algorithm in this paper is bounded by $F ( d' ,k' ,s' )$ , while the cost
of the call to that algorithm in this paper, i.e.,  the number of arithmetic
operations and sign evaluations in $\D$, is bounded by $N^{O ( m+t )} F (
d' ,k' ,s' )$.
\end{itemize}
These statements do not follow immediately from the complexity results on
the algorithms given in {\cite{BPRbook2}}. It is necessary to inspect the
algorithms in {\cite{BPRbook2}} carefully, noticing that they are all based
on linear algebra subroutines and determinant computations.
\end{remark}

We now describe an algorithm for computing the $( B,G )$-pseudo-critical
values of a family of polynomials (cf. Definition \ref{def:Bpseudocritical}), using
Notation \ref{not:rtdt}.

{\algorithm\label{alg:special}[$( B,G )$-pseudo-critical values over a Triangular Thom
Encoding]
\begin{itemize}
\item {\tmstrong{Input}}:
\begin{enumerate}
\item a triangular Thom encoding $\left(\mathcal{T}
,\tau\right)$ with $\mathcal{T} \subset \D_{t} [ T ]$,
fixing a point $\theta \in \R^{m}$, $m \leq t$;

\item a family of polynomials $\mathcal{P}= \{ P_{1} , \ldots ,P_{s} \}
\subset \D_{t} [T,X_{1} , \ldots ,X_{k} ] $, such that $\ZZ (
\mathcal{P} ( \theta , \cdot ) , \R^{k} )$ is bounded;

\item a matrix $B =( b_{i, j} )_{1 \leq i \leq s,0\leq j \leq k} \in
\mathbb{N}_{>0}^{s \times ( k+1 )}$ having good rank property;

\item a polynomial $G \in \D_{t} [ X_{1} , \ldots ,X_{k} ]$.
\end{enumerate}
\item { \bf Output}: a set of Thom encodings $( f, \sigma )$ over
$\left(\mathcal{T},\tau\right)$ specifying a finite subset of
$\R$ containing the $( B,G )$-pseudo-critical values of $\ZZ (\mathcal{P} (
\theta , \cdot ) , \R_{t}^{k} )$.

\item { \bf Complexity and degree bounds}: $s^{k}D^{O ( t+m )}d^{O
( k )}$ arithmetic operations in $\D$, where $D$ is a bound on the degree of
$\mathcal{T}$ with respect to $T, \eta$, and $d$ is a bound on the degrees
of the polynomials in $\mathcal{P}$ and of $G$. The degrees in $T, \eta$ \
of the polynomials appearing in the Thom encodings over $( \mathcal{T,} \tau
)$ output are bounded by $O ( d )^{k} D$, while the degree in the new
variable $U$ is bounded by $O ( d )^{k}$.

\item {\bf Procedure}:

\item {\tmstrong{Step 1}}. For each $I \subset [ 1,s ] $ with
$\mathrm{card} ( I ) \leq k$, $\sigma \in \{ -1,1 \}^{I}$, compute using
Algorithm 12.16 (Bounded Algebraic Sampling) from {\cite{BPRbook2}} with
ring $\D_{t} [ \theta , \gamma ]$ and input $\mathrm{CritEq} (
\tilde{\mathcal{P}}_{I,B} ,G )$ (see Definition \ref{def:Bpseudocritical}
for the definition of $\tilde{\mathcal{P}}_{I,B}$) a set of real
univariate representations over $( \mathcal{T, \tau} )$ with associated
points meeting every semi-algebraically connected component of $\ZZ \left(
\mathrm{CritEq} ( \tilde{\mathcal{P}}_{I,B} ,G ) \cup \{ G-Z \} , \R_{t}
\langle \gamma \rangle^{k+ \mathrm{card} ( I ) +2} \right)$.

\item {\tmstrong{Step 2}}. For each real univariate representation
\[ ( f,g_{0} ,g_{1} , \ldots ,g_{k} ,g_{\lambda_{0}} , \ldots
 ,g_{\lambda_{\mathrm{card} ( I )}} ,g_{Z} ) , \sigma \]
over $( \mathcal{T, \tau} )$ output in the previous step, where 
\[
f,g_{0}
,g_{1} , \ldots ,g_{k} ,g_{\lambda_{0}} , \ldots ,g_{\lambda_{\mathrm{card} (
I )}} ,g_{Z} \in \D_{t} [ T, \gamma ,U ],
\]
eliminating $U$ from the equations
\[ f ( T,U ),Z g_{0} ( T,U ) -g_{Z} ( T,U ) \]
obtain a Thom encoding $( A ( T, \gamma ,Z ) , \alpha )$ over $(
\mathcal{T,} \tau )$ describing a point $a \in \R_{t} \la \gamma \ra$.

\item {\tmstrong{Step 3}}. Compute a Thom encoding describing
$\lim_{\gamma} ( a )$ using Algorithm \ref{alg:limitofaThomencoding} (Limit
of a Thom Encoding).

\item {\tmstrong{Step 4}}. Output the set of all real univariate representations computed in Step
3.
\end{itemize}
}

\begin{proof}[Proof of correctness]
 The correctness of Algorithm \ref{alg:special}
is a consequence of the correctness of Algorithm 12.16 (Bounded Algebraic
Sampling) and Algorithm 12.14 (Limit of bounded points) from
{\cite{BPRbook2}} given the definition of $( B,G )$-pseudo-critical values
(see Definition \ref{def:Bpseudocritical}).
\end{proof}

\begin{proof}[Proof of complexity and degree bounds]
It follows from the
complexity of Algorithm 12.16 (Bounded Algebraic Sampling) and Algorithm
12.14 (Limits of bounded points) from {\cite{BPRbook2}} and from Remark
\ref{rem:complexityofring}, that the complexity is bounded by 
\[
s^{k}D^{O (m+t )} d^{O ( k )}.
\] 
Moreover, it follows from the complexity analysis of
Algorithm 12.16 (Bounded Algebraic Sampling) from {\cite{BPRbook2}}, and
that of Algorithm \ref{alg:limitofaThomencoding} (Limit of a Thom
Encoding), that the degrees in $T, \eta$ of the polynomials appearing in the
Thom encodings over $( \mathcal{T,} \tau )$ output is bounded by $O ( d
)^{k} D$, and the degree in $U$ is bounded by $O ( d )^{k}$.
\end{proof}

As mentioned earlier, we will need to compute certain well chosen finite sets
of points which correspond to points that minimize locally the distance
between pairs of semi-algebraically connected components of some basic
semi-algebraic set described in the input. For technical reasons, we need
such an algorithm in two different versions. In the first algorithm (Algorithm
\ref{alg:closest-point}) the input is a basic semi-algebraic set and a point,
while in the second algorithm (Algorithm \ref{alg:closest-pair}) the input
is a pair of basic semi-algebraic sets.

We use again Notation \ref{not:rtdt}.

{\algorithm
\label{alg:closest-point}
[Closest Point over a Triangular Thom
Encoding]
\begin{itemize}
\item{\bf Input}: A triangular Thom encoding $\left(\mathcal{T}
, \tau\right)$,
$\mathcal{ T} \subset \D_{t} [ T ]$, fixing a
point $\theta \in \R^{m}$, finite subsets $\mathcal{P} ,\mathcal{Q}
\subset \D_{t} [T,X_{1} , \ldots ,X_{k} ]$ with $\mathrm{Bas} ( \mathcal{P} (
\theta , \cdot ),\mathcal{Q} ( \theta , \cdot ) )$ bounded, and a real
univariate representation $u= ( g, \sigma ,G )$ over $\left(\mathcal{T}, \tau\right)$ with associated point $x$.

\item {\bf Output}: A finite set of real univariate representations
over $( \mathcal{T}, \tau )$ with associated points $\mathrm{MinDi} (
\mathrm{Bas} ( \mathcal{P} ( \theta , \cdot ), \mathcal{Q} (\theta , \cdot )
) , \{ x \} )$.

\item {\bf Complexity and degree bounds}: Let $\mathrm{card} (
\mathcal{Q} )\leq t$, $\deg_{X} ( \mathcal{P} , 
 \mathcal{Q} ) \leq d$, $\deg_{\eta ,T} ( \mathcal{T} ) \leq
D$, $\deg_{\eta ,T ,U} ( u ) \leq D$,  and $\deg_{\eta ,T} (
\mathcal{P} , \mathcal{Q} ) \leq dD$. Then the
number of arithmetic operations in $\D$  is bounded by $d^{O ( k )} D^{O (
m+t )}$. The degrees in $T, \eta$ of the polynomials appearing in the Thom
encodings over $( \mathcal{T}, \tau )$ output are bounded by $O ( d )^{k}
D$, while the degree in the new variable $U$ is bounded by $O ( d )^{k}$. 

\item {\tmstrong{Procedure}}:

\item {\tmstrong{Step 1}}. Let $S \subset \R^{k} \times \R^{k}$ be the
semi-algebraic sets defined by
\begin{eqnarray*}
S & := & \mathrm{Bas} ( \mathcal{P} ( \theta , \cdot ) ,\mathcal{Q} (
\theta , \cdot ) ) \times \{ x \} ,
\end{eqnarray*}
where, with $G= ( g_{0} , \ldots ,g_{k} )$, the point $x$ associated to $u$
is defined by 
\[ \sum_{i=1}^{k} ( g_{0} ( T,U ) Y_{i} - g_{i} ( T,U ) )^{2} . \]

\item {\tmstrong{Step 2}}. Let $F =\sum_{1 \leq i \leq k} ( X_{i} -Y_{i}
)^{2}$. Apply Algorithm 14.9 (Global Optimization) from {\cite{BPRbook2}}
with ring $\D [ \theta , \theta_{g} ]$,  where $( \theta , \theta_{g} )$
is associated to the triangular Thom encoding $( ( \mathcal{T},g ) , ( \tau
, \sigma ) )$, and the pair $( S,F )$ as input, and
project the set of real univariate representations over $( (
\mathcal{T},g ) , ( \tau , \sigma ) )$ that are output, to the first $k$ coordinates.

\item {\tmstrong{Step 3.}} For each univariate representation $w= ( h (
T,U,V ) , \sigma_{h} ,H ( T,U,V ) )$ output in Step 2, use Algorithm 12.16
(Bounded Algebraic Sampling) from {\cite{BPRbook2}} with ring $\D [ \theta
]$ and the polynomials $\{ g,h \}$ to obtain a set of real univariate
representations $v= ( e ( T,T' ) , \sigma_{e} ,E= ( e_{0,} e_{U} ,e_{V} )
)$. Substitute the rational functions $\tfrac{e_{U}}{e_{0}} ,
\frac{e_{V}}{e_{0}}$ for $U,V$ in the real univariate representation $w$, and
output the resulting real univariate representation over $( \mathcal{T},
\tau )$. 
\end{itemize}}

\begin{proof}[Proof of correctness]The correctness of Algorithm
\ref{alg:closest-point} is a consequence of the correctness of Algorithm
14.9 (Global Optimization), and of Algorithm 12.16 (Bounded Algebraic
Sampling) from {\cite{BPRbook2}}. The degree bounds follow from the
complexity analysis of the above algorithms.
\end{proof}

\begin{proof}[Proof of complexity and degree bounds]
It follows from the
complexity analysis of Algorithm 14.9 (Global Optimization) from {\cite{BPRbook2}},
and Remark \ref{rem:complexityofring},  that the complexity of Step 2 is
bounded by $d^{O ( k )} D^{O ( m+t )}$. Moreover, the degrees in $\eta
,T,U,V$ of the polynomials appearing in the real univariate representation
$w$ are bounded by $D d^{O ( k )}$. The cardinality of the set of real
univariate representations output is bounded by $2^{t} d^{O ( k )}$. It
follows from the complexity of Algorithm 12.16 (Bounded Algebraic
Sampling) from {\cite{BPRbook2}} and Remark \ref{rem:complexityofring} that
the degrees in $\eta ,T,T'$ of the $e,E$ are bounded by $D^{O ( 1 )} d^{O (
k )}$, and that the complexity of Step 3 is bounded by $d^{O ( k )} D^{O
( t )}$.
\end{proof}

{\algorithm
\label{alg:closest-pair}
[Closest Pairs over a Triangular Thom
Encoding]
\begin{itemize}
\item {\bf Input}: a triangular Thom encoding $\left(\mathcal{T}
, \tau\right)$,$\mathcal{T} \subset \D_{t} [ T ]$, fixing a point
$\theta \in \R_{t}^{m}$ and finite subsets $\mathcal{P}_{1} ,\mathcal{Q}_{1}
,\mathcal{P}_{2} ,\mathcal{Q}_{2} \subset \D_{t} [T ,X_{1} , \ldots
,X_{k} ]$ such that $\mathrm{Bas} ( \mathcal{P}_{1} ( \theta , \cdot )
\mathcal{,Q}_{1} ( \theta , \cdot ) )$ and $\mathrm{Bas} ( \mathcal{P}_{2} (
\theta , \cdot ) \mathcal{,Q}_{2} ( \theta , \cdot ) )$ are bounded.

\item {\bf Output}: A finite set of real univariate representations
over $( \mathcal{T}, \tau )$ with associated points
\[ \mathrm{MinDi} ( \mathrm{Bas} ( \mathcal{P}_{1} ( \theta , \cdot )
 \mathcal{,Q}_{1} ( \theta , \cdot ) ) , \mathrm{Bas} ( \mathcal{P}_{2} (
 \theta , \cdot ) \mathcal{,Q}_{2} ( \theta , \cdot ) ) ) . \]
 
\item {\bf Complexity and degree bounds}: Suppose that
\begin{enumerate}
\item
$
\mathrm{card} (\mathcal{Q}_{1} ), \mathrm{card} ( \mathcal{Q}_{2} ) \leq t,
$
\item
$\deg_{X} ( \mathcal{P}_{1} ,\mathcal{Q}_{1} ,\mathcal{P}_{2} ,
 \mathcal{Q}_{2} ) \leq d$, 
 \item
 $\deg_{\eta ,T} ( \mathcal{T} ) \leq
D$, 
\item
and $\deg_{\eta ,T} ( \mathcal{P}_{1} ,\mathcal{Q}_{1}
,\mathcal{P}_{2} ,\mathcal{Q}_{2} ) \leq dD$.
\end{enumerate}
 Then the number
of arithmetic operations in {\D} is bounded by $d^{O ( k )} D^{O ( m+t )}$.
The degrees in $T, \eta$ of the polynomials appearing in the Thom
encodings over $( \mathcal{T}, \tau )$ output are bounded by $O ( d )^{k}
D$, while the degree in the new variable $U$ is bounded by $O ( d )^{k}$.

\item {\tmstrong{Procedure}}:

\item {\tmstrong{Step 1}}. Let $S \subset \R^{k} \times \R^{k}$ be the
semi-algebraic sets defined by
\begin{eqnarray*}
S & = & \mathrm{Bas} ( \mathcal{P}_{1} ( \theta , \cdot ) ,\mathcal{Q}_{1} (
\theta , \cdot ) ) \times \mathrm{Bas} ( \mathcal{P}_{2} ( \theta ,
\cdot ) ,\mathcal{Q}_{2} ( \theta , \cdot ) ) .
\end{eqnarray*}
\item {\tmstrong{Step 2}}. Let $F =\sum_{1 \leq i \leq k} ( X_{i} -Y_{i}
)^{2}$. Apply Algorithm 14.9 (Global Optimization) from {\cite{BPRbook2}} to
the pair $( S,F )$ with ring $\D_{t} [ \theta ]$ and project the output set
of real univariate representations over $( \mathcal{T}, \tau )$ to the first
$k$-coordinates as well as to the last $k$-coordinates.
\end{itemize}}

\begin{proof}[Proof of correctness]
 The correctness of Algorithm
\ref{alg:closest-pair} is a consequence of the correctness of Algorithm 14.9
(Global Optimization) from {\cite{BPRbook2}}. 
\end{proof}

\begin{proof}[Proof of complexity and degree bounds] It follows from the
complexity of Algorithm 14.9 (Global Optimization) from {\cite{BPRbook2}}
and Remark \ref{rem:complexityofring} that the complexity of Step 2 is
bounded by $^{} d^{O ( k )} D^{O ( m+t )}$. The degree bounds follow from
the complexity analysis of the Algorithm 14.9 (Global Optimization) from
{\cite{BPRbook2}}.
\end{proof}

\subsection{The Divide algorithm\label{subsec:divide}}

We now describe Algorithm \ref{algo:divide} (Divide) which will be used later
to create the left and right children of a node of the tree, $\mathrm{Tree} (
V,\mathcal{A} )$, described in Section \ref{subsec:tree} above.

{\algorithm\label{algo:divide}[Divide]
\begin{itemize}
\item {\tmstrong{Input}}: A tuple $( s, ( \mathcal{T} , \tau ) , \mathcal{P}
, \mathcal{Q} ,A )$ satisfying the following.
\begin{enumerate}
\item $s \in \{ 0,1 \}^{t}$.

\item $( \mathcal{T}, \tau )$ is a triangular Thom encoding fixing $\theta
\in \R_{t}^{| \mathrm{fix} ( s ) |} $, where $\mathcal{T}$ is a
triangular system with variables $T_{\mathrm{fix} ( s )} = ( T_{i_{1}} ,
\ldots ,T_{i_{| \mathrm{fix} ( s ) |}} ),i_{j} \in \mathrm{fix} ( s
) = \{ i \mid s_{i} =1 \}$.

\item $\mathcal{P}\subset \D_{t} [ T_{\mathrm{fix} ( s )} 
,X_{\mathrm{Fix} ( s ) +1} , \ldots ,X_{k} ]$ is a finite set of
polynomials, where $\mathrm{Fix} ( s ) = \sum_{i=1}^{t} s_{i} k'/2^{i}$, and $\mathcal{Q}\subset \D_{t} [ T_{\mathrm{fix} ( s )}
,X_{\mathrm{Fix} ( s ) +1} , \ldots ,X_{k} ]$ is a set of $t-\mathrm{card}
( \mathrm{fix} ( s ) )$ polynomials, defining a semi-algebraic set
\[
\mathrm{Bas} ( \mathcal{P} ( \theta , \cdot ) ,\mathcal{Q} ( \theta , \cdot
) ) \subset \R_{t}^{k- \mathrm{Fix} ( s )}
\]
(cf. Notation \ref{not:rtdt}).

\item $A$ is a finite set of real univariate representations over
$\mathcal{T}$, with associated points $\mathcal{A} \subset S= \mathrm{Bas} (
\mathcal{P} ( \theta , \cdot ) ,\mathcal{Q} ( \theta , \cdot ) )$, meeting
every semi-algebraically connected component of $S$.

\item 
$\ZZ ( \mathcal{P} ( \theta , \cdot ) ,
\R_{t}^{k- \mathrm{Fix} ( s )} )
$
is strongly of dimension $\leq p= k'/2^{t}$,
if $t \neq 0$. More precisely, for every $z \in \R_{t}^{p}$, 
$\ZZ ( \mathcal{P} ( \theta , \cdot ) , \R_{t}^{k-\mathrm{Fix} ( s )})_{z}$ 
is a finite set (possibly empty).
\end{enumerate}

\item {\tmstrong{Output}}:

A tuple $( \tilde{\mathcal{P}} , \tilde{\mathcal{Q}} , \tilde{A}
,N,B, ( \mathcal{P}^{0} ( \alpha ) ,\mathcal{Q}^{0} ( \alpha ) 
,A ( \alpha ) )_{\alpha \in \mathcal{I} ( \tilde{\mathcal{P}} ,
\tilde{\mathcal{Q}} ,p/2 )} )$ where:
\begin{enumerate}
\item $\tilde{\mathcal{P}} \subset \D_{t} \left[ \zeta_{t+1} ,
\eps_{t+1} \right] [ T_{\mathrm{fix} ( s )} ,X_{\mathrm{Fix} ( s ) +1} ,
\ldots X_{k} ]$, is a finite set of polynomials with $\mathrm{card} (
\tilde{\mathcal{P}} ) =k- \mathrm{Fix} ( s ) -p$.

\item $\tilde{\mathcal{Q}} \subset \D_{t} \left[ \zeta_{t+1} ,
\eps_{t+1} , \delta_{t+1} \right] [ T_{\mathrm{fix} ( s )} ,X_{X_{\mathrm{Fix}
( s ) +1}} , \ldots X_{k} ]$ is a finite set of polynomials with
$\mathrm{card} ( \tilde{\mathcal{Q}} ) = \mathrm{card} (\mathcal{Q} ) =t-
\mathrm{card} ( \mathrm{fix} ( s ) )$.

\item $\tilde{A}$ is a set of real univariate representations over $(
\mathcal{T}, \tau )$ whose set of associated points is
$\tilde{\mathcal{A}} \subset \tilde{S} = \mathrm{Bas} (
\tilde{\mathcal{P}} ( \theta , \cdot ) , \tilde{\mathcal{Q}} ( \theta ,
\cdot ) )$.

\item $N$ is a set of real univariate representations, $u= ( h, \sigma ,H
)$, over $( \mathcal{T}, \tau )$ with associated points $\mathcal{N}
\subset \R_{t+1}^{p/2}$(and new variable $T_{t+1}$).

\item $B= \bigcup_{u \in N} B ( u )$, where for each $u= ( h, \sigma ,H )
\in N$ output in (4), $B ( u )$ is a set of real univariate
representations over $( ( \mathcal{T},h ) , ( \tau , \sigma ) )$
describing $\theta' = ( \theta ,x_{\sigma} ) \in \R_{t}^{\mathrm{Fix} ( s )
+1}$ whose set of associated points is $\mathcal{B} ( u ) \subset
\mathrm{Bas} ( \tilde{\mathcal{P}}_{u} ( \theta' , \cdot ) ,
\tilde{\mathcal{Q}}_{u} ( \theta' , \cdot ) )$. We denote the set
of associated points of $B$ by $\mathcal{B}$.

\item For every $\alpha \in \mathcal{I} ( \tilde{\mathcal{P}} ,
\tilde{\mathcal{Q}} ,p/2 )$ (see Notation
\ref{not:indices}), 
\[
\mathcal{P}^{0} ( \alpha ) ,\mathcal{Q}^{0} (
\alpha ) \subset \D_{t+1} [ T_{\mathrm{fix} ( s )} ,X_{\mathrm{Fix} ( s ) +1}
, \ldots X_{k} ]
\]
 are finite subsets with $\mathrm{card} ( \mathcal{Q}^{0}
( \alpha ) )=\mathrm{card} (\mathcal{Q} ) +1$, and $A ( \alpha )$ is a
set of real univariate representations over $\mathcal{T}$, whose set of
associated points is $\mathcal{A} ( \alpha ) \subset S^{0} ( \alpha )
= \mathrm{Bas} ( \mathcal{P}^{0} ( \alpha ) ( \theta , \cdot )
,\mathcal{Q}^{0} ( \alpha ) ( \theta , \cdot ) ) .$
\end{enumerate}
The tuple $( \tilde{\mathcal{P}} , \tilde{\mathcal{Q}} , \tilde{A} ,N,B
 , ( \mathcal{P}^{0} ( \alpha ) ,\mathcal{Q}^{0} ( \alpha )
 ,A ( \alpha ) )_{\alpha \in \mathcal{I} ( \tilde{\mathcal{P}} ,
\tilde{\mathcal{Q}} ,p/2 )} )$ satisfies the following properties.
Let
\begin{eqnarray*}
\tilde{S}^{0} &=& \lim_{\gamma_{t+1}} ( \cup_{\alpha \in
\mathcal{I} ( \tilde{\mathcal{P}} , \tilde{\mathcal{Q}} ,p/2 )}
S^{0} ( \alpha)), \\ 
\tilde{S}^{1} &=&
\tilde{S}_{\mathcal{N}}.
\end{eqnarray*}

\begin{enumerate}
\item $\lim_{\zeta_{t}} ( \tilde{S} ) =S$.

\item $\tilde{S}^{0} \cup \tilde{S}^{ 1}$ has good
connectivity properties with respect to $\tilde{S}$.

\item 
$\tilde{S}^{0}$ and 
$\tilde{S}_{\mathcal{N}}$ are both strongly of dimension $\leq p/2$.

\item $\tilde{\mathcal{A}} = \mathrm{MinDi} ( \tilde{S} ,\mathcal{A} ) \cup
\mathrm{MinDi} ( \tilde{S} , \tilde{S} )$, 
$\mathcal{N} \supset \pi_{[
\mathrm{Fix} ( s ) +1, \mathrm{Fix} ( s ) +p/2 ]} (\tilde{\mathcal{A}})$.

\item $\mathcal{B} = \tilde{S}^{0} \cap \tilde{S}^{1}$.

\item $\mathcal{A} ( \alpha ) = \mathrm{MinDi} ( S^{0} ( \alpha ) ,
\mathcal{B} ) \cup \left( \bigcup_{\beta \in \mathcal{I}( \tilde{P} ,
\tilde{\mathcal{Q}} , \ell )} \mathrm{MinDi} ( S^{0} ( \alpha ) ,S^{0} (
\beta ) ) \right)$.

\item For every $\alpha \in \mathcal{I}$, $\mathcal{A
( \alpha )} \subset S^{0} ( \alpha )$ meets every semi-algebraically
connected component of $S^{0} ( \alpha )$, 
and for every $\alpha , \beta$
in $\mathcal{I}$, and $C$ (resp. $D$) semi-algebraically connected
component of $S^{0} ( \alpha )$ (resp. $S^{0} ( \beta )$) such that
$\lim_{\gamma_{t+1}} ( C ) \cap \lim_{\gamma_{t+1}} ( D )$ is
non-empty, $\lim_{\gamma_{t+1}} ( C \cap \mathcal{A} ( \alpha ) ) \cap
\lim_{\gamma_{t+1}} ( D \cap \mathcal{A} ( \beta ) )$ is non-empty, and
meets every semi-algebraically connected component of $\lim_{\gamma_{t+1}}
( C ) \cap \lim_{\gamma_{t+1}} ( D )$.
\end{enumerate}
\item {\tmstrong{Complexity and degree bounds}}: In order to simplify the
complexity analysis, we are going to make the following assumptions which
are going to be satisfied for each call to this algorithm in Algorithm
\ref{alg:bounded} (Divide and Conquer Roadmap Algorithm for Bounded
Algebraic Sets). Let the triangular system $\mathcal{T}$ in the input be
$\mathcal{T}= ( F_{1} , \ldots ,F_{| \mathrm{fix} ( s ) |} )$, where for each
$h, \; 1 \leq h \leq \mathrm{card} ( \mathrm{fix} ( s ) )$,
$F_{h} \in \D_{t} [ T_{i_{1}} , \ldots ,T_{i_{h}} ]$. Also denote $\eta =
\left( \zeta_{1} , \eps_{1} , \delta_{1} , \gamma_{1} \cdots , \zeta_{t} ,
\eps_{t} , \delta_{t} , \gamma_{t} \right)$ and $\eta_{t+1} = \left(
\zeta_{t+1} , \eps_{t+1} , \delta_{t+1} , \gamma_{t+1} \right)$ (as in
Notation \ref{not:rtdt}). Let $c > 0$ be a constant. We assume that:
\begin{enumerate}
\item $\deg_{X} ( \mathcal{P}, \mathcal{Q} ) \leq
( 2 k )^{t } d$;
\item $\deg_{T_{\mathrm{fix} ( s )}} ( \mathcal{P},
\mathcal{Q} )$, $\deg_{\eta} ( \mathcal{P},
\mathcal{Q} )$, $\deg_{T_{\mathrm{fix} ( s )}} ( F_{h} )$, $\deg_{\eta} (
F_{h} )$ are all bounded by $D^{t} ( ( 2k )^{t} d )^{c k t}$;

\item The degrees in $\eta ,T_{\mathrm{fix} ( s )}$ of the polynomials
(belonging to $\D_{t} [ T_{\mathrm{fix} ( s )} ,T_{t+1} ]$) appearing in the
univariate representations $A$ are bounded by 
\[
D^{t} ( ( 2k )^{t} d )^{c k
t},
\]
 while the degrees in $T_{t+1}$ are bounded by $D$.
\end{enumerate}
With the above assumption on the input parameters, the output tuple
($\tilde{\mathcal{P}} , \tilde{\mathcal{Q}}$,$\tilde{A} ,N,B$,$(
\mathcal{P}^{0} ( \alpha ) ,\mathcal{Q}^{0} ( \alpha ),A (
\alpha ) )_{\alpha \in \mathcal{I} ( \tilde{\mathcal{P}} ,
\tilde{\mathcal{Q}} ,p/2 )}$) satisfies the following, for $c$ large enough.

\begin{enumerate}
\item\label{item:complexity-divide1} 
\begin{eqnarray*}
\deg_{X} ( \tilde{\mathcal{P}} ,
\tilde{\mathcal{Q}} ) &\leq&2( 2 k )^{t } d,\\
\deg_{T_{\mathrm{fix} ( s
)}} ( \tilde{\mathcal{P}} , \tilde{\mathcal{Q}} ) ,
 &\leq& 2D^{t} ( ( 2k )^{t} d )^{c k t},\\
\deg_{\eta} ( \tilde{\mathcal{P}} ,
\tilde{\mathcal{Q}} ) &\leq& 2D^{t} ( ( 2k )^{t} d )^{c k t},\\
\deg_{\eta_{t+1}} ( \tilde{\mathcal{P}} ,
\tilde{\mathcal{Q}} )&=&1.
\end{eqnarray*}

\item\label{item:complexity-divide2}
\begin{eqnarray*}
\deg_{X} ( \mathcal{P}^{0} ( \alpha ) ,
\mathcal{Q}^{0} ( \alpha ) )& \leq & ( 2 k )^{t } d k, \\
\deg_{T_{\mathrm{fix} ( s )}} ( \mathcal{P}^{0} ( \alpha ) ,
 \mathcal{Q}^{0} ( \alpha ) ) 
&\leq& D^{t+1} ( ( 2k )^{t+1} d )^{c k 
( t+1 )},\\
\deg_{\eta} ( \mathcal{P}^{0}
( \alpha ) , \mathcal{Q}^{0} ( \alpha ) ) 
&\leq& D^{t+1} ( ( 2k )^{t+1} d )^{c k 
( t+1 )}.
\end{eqnarray*}

\item \label{item:complexity-divide3}
 The univariate representations in $\tilde{A} ,N,
B ,A ( \alpha )$ have degrees in the new variable $T_{t+1}$,
as well as in $\eta_{t+1}$, bounded by $D ( ( 2k )^{t+1} d )^{ c k}$, and
have degrees at most $D^{t+1} ( ( 2k )^{t+1} d )^{c k( t+1 )}$ in the
variables $\eta ,T_{\mathrm{fix} ( s )}$.
The cardinalities of the sets $\tilde{A} ,N, B,A (
\alpha )$ are all bounded by $( \mathrm{card} ( \mathcal{A} ) +1 ) ( ( 2k
)^{t+1} d )^{ c k}$.

\end{enumerate}
The 
complexity of the algorithm 
is bounded by
\[ ( \mathrm{card} ( \mathcal{A} ) +1 ) D^{O ( t^{2} )} ( k^{t } d )^{O (
 t^{2} )} . \]
\item {\tmstrong{Procedure}}:

\item {\tmstrong{Step 1}}. Define $\tilde{\mathcal{P}}$ and
$\tilde{\mathcal{Q}}$ as in Notation \ref{not:notationtilde}.

\item {\tmstrong{Step 2}}. Compute $\tilde{M}$ as follows. For each subset
$\tilde{\mathcal{Q}}' \subset \tilde{\mathcal{Q}}$ and $ 
\tilde{\mathcal{P}} \cup \tilde{\mathcal{Q}}' = \{ F_{1} , \ldots ,F_{m} \} 
$, compute, using Algorithm 12.16 (Bounded Algebraic Sampling) from
{\cite{BPRbook2}} in the ring $\D_{t} [ \theta ]$, a finite set of real
univariate representations, $\tilde{M} ( \tilde{\mathcal{Q}}' )$ over $(
\mathcal{T}, \tau )$ whose associated points are the real solutions to the
system
\[ \mathrm{CritEq}_{p/2} ( \tilde{\mathcal{P}} ( \theta , \cdot ) \cup
 \tilde{\mathcal{Q}}' ( \theta , \cdot ) ,G ) \]
and projecting the real univariate representations to the first $k$
coordinates.

Let
\begin{eqnarray*}
\tilde{M} & = & \bigcup_{\tilde{\mathcal{Q}}' \subset \tilde{\mathcal{Q}}}
\tilde{M} ( \tilde{\mathcal{Q}}' ) .
\end{eqnarray*}
Note that the associated set of points, $\tilde{\mathcal{M}}$ of
$\tilde{M}$, is the finite set of critical points of $G$ on $\mathrm{Bas} (
\tilde{\mathcal{P}} ( \theta , \cdot ) , \tilde{\mathcal{Q}} ( \theta ,
\cdot ) )$.

\item {\tmstrong{Step 3}}. Compute a set, $D^{0}$, of Thom encodings over $(
\mathcal{T}, \tau )$ as follows.

Let
\[ F=\prod_{\tilde{\mathcal{Q}}' \subset \tilde{\mathcal{Q}}} F (
 \tilde{\mathcal{Q}}' ) , \]
where
\[ F ( \tilde{\mathcal{Q}}' ) = \sum_{P \in \mathrm{CrEq}_{\ell} (
 \tilde{\mathcal{P}} \cup \tilde{\mathcal{Q}}' ,G )} P^{2} . \]

Compute using Algorithm \ref{alg:special} ($( B,G )$-pseudo-critical values
over a Triangular Thom Encoding), a set, $D^{0}$, of Thom encodings over $(
\mathcal{T}, \tau )$, whose set of associated values, $\mathcal{D}^{0}$,
contain the $( B,G )$-pseudo-critical values of the set $ \{ F ( \theta ,
\cdot ) \} \cup \tilde{\mathcal{Q}} ( \theta , \cdot )$, with
$B=\mathcal{H}_{\mathrm{card} ( \mathcal{Q} ) +1,k- \mathrm{Fix} ( s ) +
\mathrm{card} ( \tilde{\mathcal{P}} ) + \mathrm{card} ( \tilde{\mathcal{Q}} )
+2}$ (see Notation \ref{not:cauchy}).

\item {\tmstrong{Step 4}}. Compute $M^{0}$ as follows. For each $( h,
\tau_{h} ) \in D^{0}$, use 
Algorithm 13.3 (Sampling on an Algebraic Set)
from {\cite{BPRbook2}} in the ring $\D_{t} [ \theta' ]$ (where $\theta'$ is
specified by $\mathcal{T} \cup \{ h ( T_{\mathrm{fix} ( s )} ,U ) \} , ( \tau
, \tau_{h} )$) with input the set of polynomials 
$
\{F\}
\cup
\tilde{\mathcal{Q}} \cup \{ G-U \}$
(where $F$ is as in Step 3), 
to obtain real univariate
representations 
$(e, \sigma_{e} ,E )$, where $e \in \D_{t+1} [T_{\mathrm{fix} ( s )} ,U,V ]$.
Use Algorithm 12.16 (Bounded Algebraic
Sampling) from {\cite{BPRbook2}} in the ring $\D_{t} [ \theta ]$ again with
input
$\{ h,e\}$ 
to obtain a real univariate representation 
$u= ( e',\tau_{e'} ,E' )$
over $( \mathcal{T}, \tau )$ with 
$e' \in \D_{t+1} [T_{\mathrm{fix} ( s )} ,T_{t+1} ]$. 
Substitute the rational functions, in 
$E'$
corresponding to $U,V$ into the polynomials in 
$E$
to obtain 
$E_{u}$.
Output the resulting set of real univariate representations 
$( e', \tau_{e'},E_{u} )$ 
over $( \mathcal{T}, \tau )$.

\item {\tmstrong{Step 5}}. Compute $N$ as follows. First compute $\tilde{A}$ by
applying Algorithm \ref{alg:closest-point} (Closest Point over a Triangular
Thom Encoding) with input $( ( \mathcal{T}, \tau ) , ( \tilde{\mathcal{P}} ,
\tilde{\mathcal{Q}} ) ,u )$ for each $u \in A$ and Algorithm
\ref{alg:closest-pair} (Closest Pairs over a Triangular Thom Encoding) with
input $( ( \mathcal{T}, \tau ) , ( \tilde{\mathcal{P}} , \tilde{\mathcal{Q}}
) , ( \tilde{\mathcal{P}} , \tilde{\mathcal{Q}} ) )$. Keeping the first
$p/2$ coordinates of these real univariate representations, obtain a set of
real univariate representations, $N$, over $( \mathcal{T}, \tau )$, with
associated set of points $\mathcal{N} = \pi_{[ \mathrm{Fix} ( s ) +1,
\mathrm{Fix} ( s ) +p/2 ]} (\tilde{\mathcal{M}} \cup
\mathcal{M}^{0} \cup \tilde{\mathcal{A}} )$ (identifying those which are
equal). For each $w \in \mathcal{N}$, with corresponding real univariate
representation $( f, \sigma ,F )$, let $( \mathcal{T}_{w} , \tau_{w} )$
denote the real triangular Thom encoding $( ( \mathcal{T},f ) , ( \tau ,
\sigma ) )$.

\item {\tmstrong{Step 6}}. Compute $B$ as follows. For each univariate
representation $u= ( e, \tau_{e} ,E ) \in N$, substitute the rational
functions in $u$, for the block of variables $X_{\mathrm{Fix} ( s ) +1} ,
\ldots ,X_{\mathrm{Fix} ( s ) +p/2}$, in the polynomials
$F$,$\tilde{\mathcal{Q}}$ to obtain $F_{u} , \tilde{\mathcal{Q}}_{u}$. Now
apply Algorithm 12.16 (Bounded Algebraic Sampling) from {\cite{BPRbook2}} in
the ring $\D_{t} [ \theta'' ]$ (where $\theta''$ is specified by $( (
\mathcal{T},e ) , ( \tau , \tau_{e} ) )$) with input the polynomials $F_{u}
, \tilde{\mathcal{Q}}_{u}$, and project to the co-ordinates $X_{\mathrm{Fix} (
s ) +p/2+1} , \ldots ,X_{k}$ to obtain $B ( u )$.

\item {\tmstrong{Step 7}}. For every $\alpha = ( \tilde{\mathcal{Q}}'
,r,J,J' ) \in \mathcal{I} ( \tilde{\mathcal{P}} , \tilde{\mathcal{Q}} ,p/2
)$, compute
\begin{eqnarray*}
\mathcal{P}^{0} ( \alpha ) & \assign & \tilde{\mathcal{P}} \cup
\tilde{\mathcal{Q}}' \cup \bigcup_{i \in [ \mathrm{Fix} ( s ) +p/2+1,k ]
\setminus \mathrm{fix} ( s )} \{ \mathrm{jac} ( \alpha ,i ) \} ,\\
\mathcal{Q}^{0} ( \alpha ) & \assign & \tilde{\mathcal{Q}} \cup \{
 \mathrm{jac} ( \alpha ,i)^{2} - \gamma \} .
\end{eqnarray*}
(see Notation \ref{not:covering}).

\item {\tmstrong{Step 8}}. Compute $A ( \alpha )$ by applying for each
$\beta \in \mathcal{I} ( \tilde{\mathcal{P}} , \tilde{\mathcal{Q}} ,p/2 )$,
Algorithm \ref{alg:closest-point} (Closest Point over a Triangular Thom
Encoding) with input 
\[
( ( \mathcal{T}, \tau ) ,\mathcal{P}^{0} ( \alpha )
,\mathcal{Q}^{0} ( \alpha ) ,u )
\] for each $u \in B$ computed in Step 6, and
Algorithm \ref{alg:closest-pair} (Closest Pairs over a Triangular Thom
Encoding) with input 
\[
( ( \mathcal{T}, \tau ) ,\mathcal{P}^{0} ( \alpha )
,\mathcal{Q}^{0} ( \alpha ) ,\mathcal{P}^{0} ( \beta ) ,\mathcal{Q}^{0} (
\beta ) ).
\]
\end{itemize}

\begin{proof}[Proof of correctness]
 The correctness of the algorithm follows from
the correctness of the various algorithms called inside the algorithm, and
Propositions \ref{prop:limtilde}, \ref{prop:closesta}, \ref{prop:closestb},
\ref{prop:special}, and \ref{prop:Cramer}. 
\end{proof}

\begin{proof}[Proof of complexity and degree bounds]
We first prove that the bounds stated in (\ref{item:complexity-divide1}), (\ref{item:complexity-divide2}), and (\ref{item:complexity-divide3}) are true.

\ref{item:complexity-divide1}. 
It is clear from Step 1 and Notation \ref{not:notationtilde}, that the
degrees of the polynomials in $\tilde{\mathcal{P}}$ (respectively,
$\tilde{\mathcal{Q}}$) are at most twice the degrees of the polynomials in
$\mathcal{P}$ (respectively, $\mathcal{Q}$). It follows from the assumptions
on the input that $ \deg_{X} ( \tilde{\mathcal{P}} ) , \deg_{X} (
\tilde{\mathcal{Q}} ) \leq2( 2 k )^{t } d$, and $\deg_{T_{\mathrm{fix}
( s )}} ( \tilde{\mathcal{P}} ) , \deg_{T_{\mathrm{fix} ( s )}} (
\tilde{\mathcal{Q}} ) , \deg_{\eta} ( \tilde{\mathcal{P}} ) , \deg_{\eta} (
\tilde{\mathcal{Q}} ) \leq2D^{t} ( ( 2k )^{t} d )^{c k t}$. It also
follows from Notation \ref{not:notationtilde}, that $\deg_{\eta_{t+1}} (
\tilde{\mathcal{P}} ) , \deg_{\eta_{t+1}} ( \tilde{\mathcal{Q}} )=1$. This
proves Part (1) of the complexity estimate of the output.

\ref{item:complexity-divide2}. Part (\ref{item:complexity-divide2}) 
is an easy consequence of the degree bounds on
$\tilde{\mathcal{P}}$ and $\tilde{\mathcal{Q}}$ proved above in (\ref{item:complexity-divide1}), and the
definitions of $\mathcal{P}^{0} ( \alpha )$ and $\mathcal{Q}^{0} ( \alpha
)$.

\ref{item:complexity-divide3}.
 We now bound the degrees of the univariate representations in
$\tilde{M}$,$D^{0} ,M^{0} , \tilde{A} ,N$. They have degrees in the new
variable, as well as in $\eta_{t+1}$, bounded by $( ( 2k )^{t+1} d )^{ c
k}$, and have degrees at most $D^{t} ( ( 2k )^{t+1} d )^{c k( t+1 )}$ in
the variables $T_{\mathrm{fix} ( s )} , \eta$.

i. The univariate representations in $\tilde{M}$ are obtained by applying
Algorithm 12.16 (Bounded Algebraic Sampling) from {\cite{BPRbook2}} to the
set of equations in $\mathrm{CritEq}_{p/2} ( \tilde{\mathcal{P}} \cup
\tilde{\mathcal{Q}}' ,G )$, for each subset $\tilde{\mathcal{Q}}' \subset
\tilde{\mathcal{Q}}$, and then projecting the real univariate
representations to the first $k$ coordinates. The number of variables
(including the Lagrangian variables $\lambda_{i}$'s) is at most $2k$. The
degrees in $X$ of the polynomials in $\mathrm{CritEq}_{\ell} (
\tilde{\mathcal{P}} \cup \tilde{\mathcal{Q}}' ,G )$ are bounded by the
degrees in $X$ of the polynomials in $\tilde{\mathcal{P}}$ and
$\tilde{\mathcal{Q}}$ which are at most $ 2( 2 k )^{t } d$ (using the bounds
in (1)), and the degrees in the Lagrangian variables are all equal to $1$.
The degrees in $T_{\mathrm{fix} ( s )}$ and $\eta$ in $\mathrm{CritEq}_{p/2}
( \tilde{\mathcal{P}} \cup \tilde{\mathcal{Q}}' ,G )$ are bounded by their
degrees in $\tilde{\mathcal{P}}$ and $\tilde{\mathcal{Q}}$ which are at most
$2D^{t} ( ( 2k )^{t} d )^{c k t}$ (using the bounds in (1)). Finally, the
degrees in $\eta_{t+1}$ of the polynomials in $\mathrm{CritEq}_{p/2} (
\tilde{\mathcal{P}} \cup \tilde{\mathcal{Q}}' ,G )$ are at most $1$. Now
using the complexity analysis of Algorithm 12.16 (Bounded Algebraic
Sampling) from {\cite{BPRbook2}}, we get the following bounds.

\begin{itemize}
\item
The degrees in the new variable $T_{t+1}$ and the new
infinitesimals $\eta_{t+1}$ are bounded by
\[ (2( 2 k )^{t } d )^{2c_{1} k} , \]
where $c_{1} >0$ is a constant; choosing $c$ to be sufficiently large
compared to $c_{1}$,
\begin{eqnarray*}
(2( 2 k )^{t } d )^{2c_{1} k} & \leq & ( ( 2 k )^{t+1 } d )^{c k} 
.
\end{eqnarray*}
\item
The degrees in $T_{\mathrm{fix} ( s )}$ and $\eta$, are bounded by
\[ 2D^{t+1} ( ( 2k )^{t} d )^{c k t} (2( 2 k )^{t } d )^{2c_{1} k} \leq
 D^{t+1} ( ( 2k )^{t+1} d )^{c k( t+1 )} , \]
given the choice of $c$.
\end{itemize}

ii. The real Thom encodings, $u \in D^{0}$ over $\mathcal{T}$, are computed
using Algorithm \ref{alg:special} ($( B,G )$-pseudo-critical values over a
Triangular Thom Encoding), with the polynomial

\begin{equation}
F=\prod_{\tilde{\mathcal{Q}}' \subset \tilde{\mathcal{Q}}} F (
\tilde{\mathcal{Q}}' ) , \label{eqn:product}
\end{equation}
as input, where
\[ F ( \tilde{\mathcal{Q}}' ) = \sum_{P \in \mathrm{CrEq}_{\ell} (
 \tilde{\mathcal{P}} \cup \tilde{\mathcal{Q}}' ,G )} P^{2} . \]
The number of polynomials, $F ( \tilde{\mathcal{Q}}' )$, appearing in the
product in \eqref{eqn:product} is bounded by $2^{\mathrm{card} (
\tilde{\mathcal{Q}} )} \leq 2^{t} \leq k$. Using the facts noted about the
degrees in the various variables of the polynomials in
$\tilde{\mathcal{P}}$,$\tilde{\mathcal{Q}}$ we obtain that the degree in $X$
of $F$ is bounded by $ 2^{t+1} ( 2 k )^{t } d$. The degrees in the
Lagrangian variables are bounded by $2^{t}$. The degrees in the variables
$T_{\mathrm{fix} ( s )}$, and $\eta$ are bounded by $2^{t+1} D^{t} ( ( 2k
)^{t} d )^{c k t}$ . Finally, the degree in $\eta_{t+1}$ in $F$ is at most
$2^{t}$. The number of variables is at most $2k$.

Using the complexity analysis of Algorithm \ref{alg:special} ($( B,G
)$-pseudo-critical values over a Triangular Thom Encoding) we get the
following bounds.

\begin{itemize}
\item
The degree in the new variable $U$ is bounded by
\[ ( 2^{t+1} ( 2 k )^{t } d )^{2c_{2} k} \]
where $c_{2} >0$ is a constant, while the degree in $\eta_{t+1}$ is bounded by
\[ 2^{t}( 2^{t+1} ( 2 k )^{t } d )^{2c_{2} k}; \]
choosing $c$ to be sufficiently large compared to $c_{2}$, and noting that
$2^{t} \leq k$,
\begin{eqnarray*}
2^{t}( 2^{t+1} ( 2 k )^{t } d )^{2c_{2} k} & \leq & ( ( 2 k )^{t+1
} d )^{c k}.
\end{eqnarray*}
\item The degrees in $T_{\mathrm{fix} ( s )}$ and $\eta$ are bounded by
\[ 
2^{t+1} D^{t} ( ( 2k )^{t} d )^{c k t} (2^{t+1} ( 2 k )^{t } d
 )^{2c_{2} k} \leq D^{t+1} ( ( 2k )^{t+1} d )^{c k( t+1 )} , \]
given the choice of $c$.
\end{itemize}

iii. In Step 4 (computation of $M^{0}$), the degrees of the polynomials in
the real univariate representation 
$(e,\sigma_{e},E)$
computed are bounded as follows.

\begin{itemize}
\item
The degrees in the new variable $V$and $\eta_{t+1}$ are
bounded by
\[
(t 2^{t+1} (2k)^t d)^{c_3(2k+1)}.
\]
using the complexity of
Algorithm 13.3 (Sampling on an Algebraic Set) from
\cite{BPRbook2}.
\item
 The degrees in $T_{\mathrm{fix} ( s )}$, and $\eta$ are
bounded by
\[ 
2D^{t} ( ( 2k )^{t} d )^{c k t} (t 2^{t+1} (2k)^t d)^{c_3(2k+1)}.
\]
\end{itemize}
Using again the complexity analysis of Algorithm 12.16 (Bounded Algebraic
Sampling) from {\cite{BPRbook2}} we obtain that the degrees of the
polynomials in $u$ in the various variables are bounded as follows.

\begin{itemize}
\item
The degrees in the new variable $T_{t+1}$ and in $\eta_{t+1}$
are bounded by
\[
 ( \max ( 2 \cdot2^{t} ( 2^{t+1} ( 2 k )^{t } d )^{2c_{2} k} ,
 (t 2^{t+1} (2k)^t d)^{c_3(2k+1)})^{2c_{1}} \leq ( ( 2 k )^{t+1 } d )^{ck} 
 \]
using the complexity of Algorithm 12.16 (Bounded Algebraic Sampling) from
{\cite{BPRbook2}} and the degree bounds in $U$ and $V$ of the polynomials
$D^{0}$ and 
$e$,
and choosing $c$ sufficiently large.

\item
The degrees in $T_{\mathrm{fix} ( s )}$ and $\eta$ are bounded
by the maximum of the degrees in $T_{\mathrm{fix} ( s )}$, and $\eta$ in the
polynomials $D^{0}$ and $f$ multiplied by $( ( 2 k )^{t+1 } d )^{c k}$. It
follows that these degrees are bounded by
\[ D^{t} ( ( 2 k )^{t+1 } d )^{c k ( t+1 )} . \]
\end{itemize}

iv. Using the complexity of Algorithm \ref{alg:closest-pair} (Closest Pairs
over a Triangular Thom Encoding), and Algorithm \ref{alg:closest-point}
(Closest Point over a Triangular Thom Encoding) and the degree estimates of
$\tilde{\mathcal{P}}$ and $\tilde{\mathcal{Q}}$ and of the univariate
representations in $A$, we obtain that the degrees in the univariate
representations in $\tilde{A}$ are bounded as follows.
\begin{itemize}
\item The degrees in the new variable $T_{t+1}$ and $\eta_{t+1}$ are
bounded by
\[ D (2( 2 k )^{t } d )^{2c_{1} k} , \]
where $c_{1} >0$ is a constant; and given the choice of $c$, we have that
\begin{eqnarray*}
(2( 2 k )^{t } d )^{2c_{1} k} & \leq & ( ( 2 k )^{t+1 } d )^{c
k}.
\end{eqnarray*}
\item The degrees in $T_{\mathrm{fix} ( s )}$, and $\eta$, are bounded by
\[ 2D^{t+1} ( ( 2k )^{t} d )^{c k t} (2( 2 k )^{t } d )^{2c_{1} k}
 \leq D^{t+1} ( ( 2k )^{t+1} d )^{c k( t+1 )} , \]
given the choice of $c$.

\item Finally, the cardinality of $\tilde{A}$ is bounded by $( \mathrm{card}
( \mathcal{A} ) +1 )( ( 2 k )^{t+1 } d )^{c k}$ using the complexity
analysis of Algorithm \ref{alg:closest-pair} (Closest Pairs over a
Triangular Thom Encoding) and Algorithm \ref{alg:closest-point} (Closest
Point over a Triangular Thom Encoding).
\end{itemize}
Together, (i),(ii) (iii) and (iv) above imply that the univariate
representations in $\tilde{M}$,$D^{0} ,M^{0} , \tilde{A} ,N$ have degrees in
the new variable, as well as in $T_{t+1} , \eta_{t+1}$, bounded by $D ( ( 2k
)^{t+1} d )^{ c k}$, and have degrees at most $D^{t+1} ( ( 2k )^{t+1} d )^{c
k( t+1 )}$ in $T_{\mathrm{fix} ( s )} , \eta$. Moreover, the
cardinalities of $\tilde{M}$, $D^{0} ,M^{0}$ are bounded by $( ( 2k )^{t+1}
d )^{ c k}$ (taking into account that the number of polynomials in the call
to Algorithm \ref{alg:special} ($( B,G )$-pseudo-critical values over a
Triangular Thom Encoding) in Step 3 is bounded by $t+1$). The cardinalities
of $\tilde{A} ,N$ are bounded by $( \mathrm{card} ( \mathcal{A} ) +1 )( (
2k )^{t+1} d )^{ c k}$.

Using the bound on the degrees of $F$ and the univariate representations in
$N$ obtained above, and the degree estimates of the output of Algorithm
12.16 (Bounded Algebraic Sampling) in {\cite{BPRbook2}}, we get that the
degrees of the polynomials appearing in $B$ in the new variable are bounded
by 
\begin{eqnarray*}
(2( 2 k )^{t } d )^{2c_{1} k} & \leq & ( ( 2 k )^{t+1 } d )^{c k},
\end{eqnarray*}
while the degrees in $T_{\mathrm{fix} ( s )} , \eta$ are bounded by
\begin{eqnarray*}
D^{t+1} ( ( 2k )^{t+1} d )^{c k( t+1 )} ( 2( 2 k )^{t } d ) (2(
2 k )^{t } d )^{2c_{1} k} \leq D^{t+1} ( ( 2k )^{t+1} d )^{c k( t+1 )} .
&& 
\end{eqnarray*}
Finally, the degrees in $T_{t+1} , \eta_{t+1}$ are bounded by
\begin{eqnarray*}
D (2( 2 k )^{t } d )^{2c_{1} k} (2( 2 k )^{t } d )^{2c_{1} k} &
\leq & D ( ( 2 k )^{t+1 } d )^{c k}.
\end{eqnarray*}
The cardinality of $B$ is bounded by $( \mathrm{card} ( \mathcal{A} ) +1 ) (
( 2 k )^{t+1 } d )^{c k}$.

The degree estimates, as well as the estimates on the cardinality of $A (
\alpha )$ are now a consequence of the bounds on the degrees of
$\mathcal{P}^{0} ( \alpha )$, $\mathcal{Q}^{0} ( \alpha )$ and $B$, and the
cardinality of $B$ proved above, and the complexity of Algorithm
\ref{alg:closest-pair} (Closest Pairs over a Triangular Thom encoding), and
Algorithm \ref{alg:closest-point} (Closest Point over a Triangular Thom
encoding)). This proves Part (\ref{item:complexity-divide3}) of the complexity of the output.

It follows from the complexity estimates of the algorithms used in various
steps of Algorithm \ref{algo:divide}, namely  Algorithm 12.16 (Bounded
Algebraic Sampling) in {\cite{BPRbook2}}, 
Algorithm 13.3 (Sampling on an Algebraic Set) in \cite{BPRbook2},
Algorithm 14.9 (Global
Optimization) in {\cite{BPRbook2}},  
Algorithm \ref{alg:special} ($( B,G
)$-pseudo-critical values over a Triangular Thom Encoding), Algorithm
\ref{alg:closest-pair} (Closest Pairs over a Triangular Thom Encoding),  and
Algorithm \ref{alg:closest-point} (Closest Point over a Triangular Thom
Encoding)), and 
Remark \ref{rem:complexityofring},
as well as the degree estimates proved above, that the complexity of
the whole algorithm is bounded by $( \mathrm{card} ( \mathcal{A} ) +1 ) D^{O (
t^{2} )} ( k^{t } d )^{O ( t^{2 } k )}$.
\end{proof}

\begin{remark}
\label{rem:aux}Notice that we never reduce any intermediate polynomial
obtained in the computation, modulo $\mathcal{T}$, and that $( \mathcal{T,
\tau} )$ is used only if the sign of an element of $\D_{t} [ \theta ]$,
represented by a polynomial, is required. As a result the degrees in the
$T$'s and also in the infinitesimals occurring in $\mathcal{T}$, grow. We
analyzed this growth carefully in the complexity analysis of Algorithm
\ref{algo:divide} (Divide). This is a point of difference between the
algorithm presented in the current paper, and that in {\cite{BRMS10}}. In
the Baby-step Giant-step algorithm presented in {\cite{BRMS10}} a process of
pseudo-reduction was necessary since the degree growth in $\mathcal{T}$
would have spoiled the overall complexity of the algorithm. This phenomenon
does not occur here because the number of different blocks of variables (and
hence the size of the triangular systems) in the algorithm of this paper is
much smaller ($O ( \log ( k ) )$) compared to $O( \sqrt{k})$
in the Baby-step Giant-step algorithm) and hence we can tolerate the growth
in degree in the current paper without resorting to reducing in each step.
This is fortunate, since pseudo-reduction is not anymore an option for us,
as the growth in the degrees in the various infinitesimals in this
divide-and-conquer approach would reach $d^{ O ( k^{ 2} )}$, and will be
unacceptably large.
\end{remark}

\subsection{Computation of the tree $\mathrm{Tree} ( V,\mathcal{A} )$}

In this subsection we describe an algorithm computing the tree $\mathrm{Tree} (
V,\mathcal{A} )$, using in a recursive way Algorithm \ref{algo:divide}
(Divide) and analyze the complexity of this algorithm.

The description of the algorithm will use the following notation.

\begin{notation}
\label{not:node} A {\tmem{node}} $\mathfrak{n}$ of level $t= \mathrm{level} (
\mathfrak{n} )$ is a tuple 
\[
( s ( \mathfrak{n} )
,\mathcal{T} ( \mathfrak{n} ) , W ( \mathfrak{n} ) ,\mathcal{P} (
\mathfrak{n} ) ,\mathcal{Q} ( \mathfrak{n} ) ,A ( \mathfrak{n} ) 
),
\] 
where
\begin{enumerate}
\item $s ( \mathfrak{n} ) \in \{ 0,1 \}^{t}$;

\item $\mathcal{T} ( \mathfrak{n} ) , \tau ( \mathfrak{n} ) ,W (
\mathfrak{n} )$ is a block triangular system fixing a point $\theta (
\mathfrak{n} ) \in \R_{t}^{\mathrm{card} ( \mathrm{fix} ( \mathfrak{n} ) )}$,
and $w ( \mathfrak{n} ) \in \R_{t}^{\mathrm{Fix} ( \mathfrak{n} )}$, where
\begin{eqnarray*}
\mathrm{fix} ( \mathfrak{n} ) & =& \{ i \mid s ( \mathfrak{n} )_{i} =1
\} ,\\
\mathrm{Fix} ( \mathfrak{n} ) & = & \sum_{i=1}^{t} s ( \mathfrak{n} )_{i} 
k'/2^{i} ;
\end{eqnarray*}
\item $\mathcal{P} ( \mathfrak{n} ) \subset \R_{t} [ T_{\mathrm{fix} (
\mathfrak{n} )},  X_{\mathrm{Fix} ( \mathfrak{n} ) +1} , \ldots ,X_{k} ]
$;

\item $\mathcal{Q} ( \mathfrak{n} ) \subset \R_{t} [ T_{\mathrm{fix} (
\mathfrak{n} )} ,X_{\mathrm{Fix} (\mathfrak{n} )+1}
, \ldots ,X_{k} ]$, $\mathrm{card} ( \mathcal{Q} (
\mathfrak{n} ) )=t- \mathrm{card} ( \mathrm{fix} ( \mathfrak{n} ) )$;

\item $A ( \mathfrak{n} )$ is a set of real univariate representations
over $\mathcal{T} ( \mathfrak{n} )$, with associated points $\mathcal{A} (
\mathfrak{n} ) \subset \mathrm{Bas} ( \mathcal{P} ( \mathfrak{n} ) ( \theta
( \mathfrak{n} ) , \cdot ) ,\mathcal{Q} ( \mathfrak{n} ) ( \theta (
\mathfrak{n} ) , \cdot ) )$.
\end{enumerate}
\end{notation}

{\algorithm{\label{alg:bounded}
[Computation of the tree $\mathrm{Tree} (
V,\mathcal{A} )$]
\begin{itemize}
\item {\tmstrong{Input}}: A polynomial $P \in \D [ X_{1} , \ldots ,X_{k} ]$,
such that $V= \ZZ ( P, \R^{k} )$ is bounded and strongly of dimension $\leq k'$, 
and a set $A$ of
real univariate representations with associated set of points $\mathcal{A}
\subset \ZZ ( P, \R^{k} )$.

\item {\tmstrong{Output}}: the tree $\mathrm{Tree} ( V,\mathcal{A} )$.

\item {\tmstrong{Complexity and degree bounds}}: Let $d = \deg ( P )$ and
suppose that the degrees of the polynomials appearing in the real univariate
representations in $A$ are bounded by $D$. The complexity is bounded by 
\[(
\card( \mathcal{A} ) +1 ) D^{O ( \log^{2} ( k' ) )} ( k^{\log ( k' )} d
)^{O ( k \log^{2} ( k' ) }.
\]

For each leaf node $\mathfrak{n}$ of the tree $\mathrm{Tree} ( V,\mathcal{A}
)$ output by the algorithm, the degrees in $T_{\mathrm{Fix} ( \mathfrak{n}
)}$, as well as in the variables $\eta$'s, of the polynomials in
$\mathcal{T} ( \mathfrak{n} )$, as well as those in $\mathcal{P} (
\mathfrak{n} ) $ and $\mathcal{Q} ( \mathfrak{n} )$, are all bounded
by 
\[
D^{\log ( k' ) +1} ( k^{\log ( k' )} d )^{O ( k\log ( k' ) )}.
\] 
The
degrees in $X_{\mathrm{Fix} ( \mathfrak{n} ) +1} , \ldots ,X_{k}$ of
$\mathcal{P} ( \mathfrak{n} ) $ and $\mathcal{Q} ( \mathfrak{n} )$
are bounded by 
\[
(O ( k ))^{\log ( k' )} d.
\]

\item {\tmstrong{Procedure}}:

\item {\tmstrong{Step 1}}. Initialize $\mathfrak{r}$ to be the node with
\begin{enumerate}
\item $s ( \mathfrak{r} ):= ( )$;

\item $\mathcal{T} ( \mathfrak{r} ) , \tau ( \mathfrak{r} ) ,W (
\mathfrak{r} )$ is empty;

\item $\mathcal{P} ( \mathfrak{r} ) := \{ P \}$;

\item $\mathcal{Q} ( \mathfrak{r} ) := \emptyset$;

\item $A ( \mathfrak{r} ) :=A$.
\end{enumerate}
Initialize the set $\mathrm{Nodes} := \{ \mathfrak{r} \}$.

\item {\tmstrong{Step 2}}. Repeat until $\mathrm{level} ( \mathfrak{n} ) = 
\log ( k' )$ for all $\mathfrak{n} \in \mathrm{Nodes}$;
\begin{itemize}
\item Select $\mathfrak{n} \in \mathrm{Nodes}$, such that $\mathrm{level} (
\mathfrak{n} ) < \log ( k' )$.

\item Remove $\mathfrak{n}$ from $\mathrm{Nodes}$.

\item Call Algorithm \ref{algo:divide} (Divide) with input $( s (
\mathfrak{n} ) ,\mathcal{T} ( \mathfrak{n} ) ,\tau ( \mathfrak{n} )
,\mathcal{P} ( \mathfrak{n} ) ,\mathcal{Q} ( \mathfrak{n} ) ,A (
\mathfrak{n} ) )$.
\begin{itemize}
\item For each $\alpha$ output  by Algorithm
\ref{algo:divide} (Divide) add a node $\mathfrak{m}$ to $\mathrm{Nodes}$
with
\begin{enumerate}
\item $s ( \mathfrak{m} ) := ( s ( \mathfrak{n} ) ,0 )$;

\item $( \mathcal{T} ( \mathfrak{m} ) , \tau ( \mathfrak{m} ) ) := (
\mathcal{T} ( \mathfrak{n} ), \tau ( \mathfrak{n} ) )$; $W (
\mathfrak{m} ) := W ( \mathfrak{n} )$;

\item $\mathcal{P} ( \mathfrak{m} ) :=\mathcal{P}^{0} ( \alpha )
\subset \R_{t+1} [ T_{\mathrm{fix}( \mathfrak{m} 
)} ,X_{\mathrm{Fix} ( \mathfrak{m} ) +1} , \ldots ,X_{k} ]$;

\item $\mathcal{Q} ( \mathfrak{m} ) :=\mathcal{Q}^{0} ( \alpha )
\subset \R_{t+1} [ T_{\mathrm{fix} ( \mathfrak{m} )} ,X_{\mathrm{Fix} (
 \mathfrak{m} )+1} , \ldots ,X_{k} ]$;

\item $A ( \mathfrak{m} ) :=A ( \alpha )$.
\end{enumerate}
\item For each real univariate representation $u_{w}=( f_{w} ,F_{w}
) , \tau$, in $N$, representing a point $w \in \mathcal{N}$, with
$f_{w,} ,F_{w} \subset \D_{t} [ T_{\mathrm{fix}( \mathfrak{n}
 )} ,T_{t+1} ]$, output by the algorithm, add a node
$\mathfrak{m}$ to $\mathrm{Nodes}$ with
\begin{enumerate}
\item $s ( \mathfrak{m} ) := ( s ( \mathfrak{n} ) ,1 )$;

\item $( \mathcal{T} ( \mathfrak{m} ) , \tau ( \mathfrak{m} ) ) := (
\mathcal{T} ( \mathfrak{n} ) ,f_{w} ), ( \tau ( \mathfrak{n}
) , \tau )$; $W ( \mathfrak{m} ) := ( W ( \mathfrak{n} ) ,u_{w} )$;

\item $\mathcal{P} ( \mathfrak{m} ) :=\mathcal{P} ( \mathfrak{n} )_{w}
\subset \R_{t} [ T_{\mathrm{fix} ( \mathfrak{m} )} ,X_{\mathrm{Fix} (
\mathfrak{m} ) +1} , \ldots ,X_{k} ]$;

\item $\mathcal{Q} ( \mathfrak{m} ) :=\mathcal{Q} ( \mathfrak{n}
)_{w} \subset \R_{t} [ T_{\mathrm{fix} ( \mathfrak{m} )}
,X_{\mathrm{Fix} (\mathfrak{m} )+1} , \ldots
,X_{k} ]$;

\item $A ( \mathfrak{m} ) := \tilde{A} ( \mathfrak{n} )_{w} \cup B (
\mathfrak{n} )_{w}$.
\end{enumerate}
\item Denote by $\mathcal{A} ( \mathfrak{m} )$ (respectively,
$\mathcal{B} ( \mathfrak{m} ) , \tilde{\mathcal{A}} ( \mathfrak{m} )$)
the set of points associated to $A ( \mathfrak{m} )$ (respectively, $B (
\mathfrak{m} ) , \tilde{A} ( \mathfrak{m} )$).
\end{itemize}
\end{itemize}
\end{itemize}}

\begin{proof}[Proof of correctness]
The correctness of the algorithm follows from
Theorem \ref{thm:correctness} and the correctness of Algorithm
\ref{algo:divide}.
\end{proof}

\begin{proof}[Proof of complexity and degree bounds]
 We first observe that for
each call to Algorithm \ref{algo:divide} (Divide) in Algorithm
\ref{alg:bounded} the input satisfies the estimates used in the complexity
analysis of Algorithm \ref{algo:divide} (Divide). This is obvious when
Algorithm \ref{algo:divide} (Divide) is called with
$\mathfrak{n}=\mathfrak{r}$, and in the other cases it follows inductively
from the bound on the degrees proved in the complexity analysis of Algorithm
\ref{algo:divide} (Divide).

Using the complexity analysis of Algorithm \ref{algo:divide} (Divide) the
number of right children of any node $\mathfrak{n}$ of level $t$ in the tree
created by the algorithm is bounded by the cardinality of $N (
\mathfrak{n} )$, which in turn is bounded by $( \mathrm{card} ( \mathcal{A} (
\mathfrak{n} ) ) +1 )( ( 2k )^{t+1} d )^{ c k}$. The number of left
children of $\mathfrak{n}$ is bounded by $k \times \mathrm{card} ( \mathcal{I}
( \mathcal{P} ( \mathfrak{n} ) ,\mathcal{Q} ( \mathfrak{n} ) , k'/2^{\mathrm{level} ( \mathfrak{n} )} ) ) =O ( 1 )^{k}$. Thus, the total number
of children of a node $\mathfrak{n}$ of level $t$ is bounded by $(
\mathrm{card} ( \mathcal{A} ( \mathfrak{n} ) ) +1 ) ( ( 2k )^{t+1} d )^{ c
k}$.

Now let $\mathfrak{m},\mathfrak{m}'$ be two distinct right children of
$\mathfrak{n}$. Then, clearly $\mathcal{A} ( \mathfrak{m} ) \cap \mathcal{A}
( \mathfrak{m}' )= \emptyset$, and using the complexity analysis of
Algorithm \ref{algo:divide} (Divide) we have
\begin{eqnarray*}
\sum_{\mathfrak{m}\text{ right child of }\mathfrak{n}}
\mathrm{card} ( \mathcal{A} ( \mathfrak{m} ) ) & \leq & \mathrm{card} (
\mathcal{B} ( \mathfrak{n} ) ) + \mathrm{card} (
\tilde{\mathcal{A}} ( \mathfrak{n} ) )\\
& \leq & 2 ( \mathrm{card} ( \mathcal{A} ( \mathfrak{n} ) ) +1 )( ( 2k
)^{t+1} d )^{ c k} .
\end{eqnarray*}
Moreover, for each left child $\mathfrak{m}$ of $\mathfrak{n}$, $\mathrm{card}
( \mathcal{A} ( \mathfrak{m} ) ) \leq \mathrm{card} ( \mathcal{A} (
\mathfrak{n} ) ) ( ( 2k )^{t+1} d )^{ c k}$ again using the complexity
analysis of Algorithm \ref{algo:divide} (Divide). But since there are only
$O ( 1 )^{k}$ left children of $\mathfrak{n}$ we obtain that \
\begin{eqnarray}
\sum_{\mathfrak{m} \text{ child of }\mathfrak{n}} ( \mathrm{card}
( \mathcal{A} ( \mathfrak{m} ) ) +1 ) & \leq & 3 ( \mathrm{card} (
\mathcal{A} ( \mathfrak{n} ) ) +1 )( ( 2k )^{t+1} d )^{ c k},
\nonumber\\
& \leq & ( \mathrm{card} ( \mathcal{A} ( \mathfrak{n} ) ) +1 )( ( 2k
)^{t+1} d )^{ c'k} .\label{eqn:numberofchildren}
\end{eqnarray}
for some $c' >0$.

We now prove by induction on $t$, with base case $t=0$, that
\begin{eqnarray}
\sum_{\mathrm{level} ( \mathfrak{n} ) =t} ( \mathrm{card} ( \mathcal{A} (
\mathfrak{n} ) ) +1 ) & \leq & ( \mathrm{card} ( \mathcal{A} ) +1 ) ( ( 2k
)^{t+1} d )^{ c''k ( t+1 )} .\label{eqn:total}
\end{eqnarray}
for some $c'' >0$.

For the base case,
\begin{eqnarray*}
\mathrm{card} ( \mathcal{A} ( \mathfrak{r} ) ) +1 & \leq & ( \mathrm{card} (
\mathcal{A} ) +1 ) ( 2k d )^{ c''k} .
\end{eqnarray*}
For the inductive step (from $t-1$ to $t$) we have
\begin{eqnarray*}
\sum_{\mathrm{level} ( \mathfrak{m} ) =t} ( \mathrm{card} ( \mathcal{A} (
\mathfrak{m} ) ) +1 ) & = & \sum_{\mathrm{level} ( \mathfrak{n} ) =t-1}
\sum_{\mathfrak{m} \text{ child of }\mathfrak{n}} ( \mathrm{card}
( \mathcal{A} ( \mathfrak{m} ) ) +1 )\\
& \leq & \sum_{\mathrm{level} ( \mathfrak{n} ) =t-1} ( \mathrm{card} (
\mathcal{A} ( \mathfrak{n} ) ) +1 )( ( 2k )^{t} d )^{ c'k}
\text{(using (\ref{eqn:numberofchildren}))}\\
& \leq & ( \mathrm{card} ( \mathcal{A} ) +1 ) ( ( 2k )^{t} d )^{ c''k t} 
( ( 2k )^{t} d )^{ c'k} \\
&&\text{ (using the induction hypothesis)}\\
& \leq & ( \mathrm{card} ( \mathcal{A} ) +1 ) ( ( 2k )^{t+1} d )^{ c''k (
t+1 )} ,
\end{eqnarray*}
taking $c''$ large compared to $c'$.

Using the complexity analysis of Algorithm \ref{algo:divide} (Divide) and
\eqref{eqn:total}, we now obtain that the total cost of the calls to
Algorithm \ref{algo:divide} (Divide) for all nodes of the tree at level $t$
is bounded by
\[ ( \mathrm{card} ( \mathcal{A} ) +1 ) ( ( 2k )^{t+1} d )^{ c''k ( t+1 )} 
 D^{O ( t^{2} )} ( k^{t}d )^{O ( t^{2} k )} = ( \mathrm{card} (
 \mathcal{A} ) +1 ) D^{O ( t^{2} )} ( k^{t}d )^{O ( t^{2} k )} . \]
Since the tree has depth $\log ( k' )$ the total cost of the calls to \
Algorithm \ref{algo:divide} (Divide) is bounded by
\[ ( \mathrm{card} ( \mathcal{A} ) +1 ) D^{O ( \log^{2} ( k' ) )} ( k^{\log ( k' )}
 d )^{O ( k\log^{2} ( k' ) )} . \]
From the complexity analysis of Algorithm \ref{algo:divide} (Divide) we also
get that for each $\mathfrak{n} \in \mathrm{Leav}$, the degrees in
$T_{\mathrm{Fix} ( \mathfrak{n} )}$, ad  the variables $\eta$'s, of
the polynomials in $\mathcal{T} ( \mathfrak{n} )$, as well as those in
$\mathcal{P} ( \mathfrak{n} ) $ and $\mathcal{Q} ( \mathfrak{n} )$,
are all bounded by 
\[
D^{\log ( k' ) +1} ( k^{\log ( k' )} d )^{O ( k\log ( k'
) )}.
\]
 The degrees in $X_{\mathrm{Fix} ( \mathfrak{n} ) +1} , \ldots ,X_{k}$
of $\mathcal{P} ( \mathfrak{n} ) $ and $\mathcal{Q} ( \mathfrak{n}
)$ are bounded by $O ( k )^{\log ( k' )} d$. 
\end{proof}

\begin{remark}
\label{rem:middlepathisbad} Note that in the above analysis of the degrees
in the variables $T$'s and the infinitesimals $\eta$'s of the
polynomials in $\mathcal{P} ( \mathfrak{n} ) ,\mathcal{Q} ( \mathfrak{n} )$
depend on $s ( \mathfrak{n} )$. It is instructive to work out the actual
bound on the degrees in the following three cases (in the special case where
$\mathcal{A}= \emptyset$):

{\noindent}1. $s ( \mathfrak{n} ) =0^{t}$: In this case, the polynomials in
$\mathcal{P} ( \mathfrak{n} ) ,\mathcal{Q} ( \mathfrak{n} )$ do not have any
$T' s$ in them, and the degrees in $\eta$ is bounded by $O ( k )^{t}
d$. It is not difficult to see that the complexity of Algorithm
\ref{algo:divide} (Divide) at such a node is bounded by $d^{O ( k )} k^{O (
k t )}= d^{O ( k )} k^{O ( k\log ( k' ) )}$.

{\noindent}2. $s ( \mathfrak{n} ) =1^{t}$: In this case, the degrees of the
polynomials in $\mathcal{P} ( \mathfrak{n} ) ,\mathcal{Q} ( \mathfrak{n} )$
in $T$'s and $\eta$'s are bounded by $O ( d )^{k'/2^{t-1}}$. As a
consequence, it is not difficult to see that the complexity of Algorithm
\ref{algo:divide} (Divide) at such a node is bounded by $d^{O ( k )} k^{O (
k\log ( k' ) )}$.

If these were the only types of nodes in the tree computed by Algorithm
\ref{alg:bounded} then we would obtain an algorithm with complexity $d^{O (
k\log ( k' ) )} k^{O ( k\log^{2} ( k' ) )}$. In fact the complexity is
worse and this is caused by paths in the tree which are away from the
extreme left and right ones. For example consider a node $\mathfrak{n}$ with
level $t$, and with $s ( \mathfrak{n} ) =0101 \ldots$. The polynomials in
$\mathcal{P} ( \mathfrak{n} ) ,\mathcal{Q} ( \mathfrak{n} )$ will depend on
$T_{2} ,T_{4}, \ldots$ while since $\mathrm{Free} ( \mathfrak{m} )
\geq \frac{k}{2}$ for each node $\mathfrak{m}$ along the path from the root
to $\mathfrak{n}$, the degrees in each of the $T_{2i}$ can only be bounded
by $( k^{i} d )^{O ( k i )}$, and is thus $( k^{t} d )^{O ( k t )}$ in the
worst case. As a result the complexity of the call to Algorithm
\ref{algo:divide} (Divide) at the node $\mathfrak{n}$ can only be bounded by
$( k ^{\log ( k' )}d )^{O ( k\log^{2} ( k' ) )}$, and this dominates the
complexity of all calls to Algorithm \ref{algo:divide} (Divide) in Algorithm
\ref{alg:bounded}.

\end{remark}

\subsection{Divide and Conquer Roadmap}

We are going to construct the
roadmap of $V$ from the tree $\mathrm{Tree} ( V,\mathcal{A} )$,  by taking limits of the basic semi-algebraic sets  associated to the leaves,  which are strongly of  
dimension $\leq 1 $. Theorem \ref{thm:correctness}
then guarantees the correctness of the algorithm. Thus, we need to know how to
compute limits of points and curve segments, and this what we explain below.

In the following $\overline{\eps} = \left( \eps_{1} , \ldots ,
\eps_{t} \right)$ is a tuple of infinitesimals.
{\algorithm{\label{alg:limitofaboundedpoint}
[Limit of a Bounded Point]
\begin{itemize}
\item {\tmstrong{Input}}: 
\begin{enumerate}
\item A Thom encoding $\left( f_{\overline{\eps}} ,
\sigma_{\overline{\eps}} \right)$ , $f_{\overline{\eps}} \in \D \left[
\overline{\eps} ,U \right]$, representing $x_{\overline{\eps}} \in \R
\la \overline{\eps} \ra$ bounded over $\R$.

\item A real univariate representation $\left( g_{\overline{\eps}} ,
\tau_{\overline{\eps}} ,G_{\overline{\eps}} \right)$ over $\left(
f_{\overline{\eps}} , \sigma_{\overline{\eps}} \right)$, where
$g_{\overline{\eps}} ,G_{\overline{\eps}} \subset \D [ \overline{\eps}
,U,V]$, representing a point $z_{\overline{\eps}} \in \R \la
\overline{\eps} \ra^{p}$ bounded over $\R$.
\end{enumerate}
\item {\tmstrong{Output}}: a real univariate representation $(g, \tau ,G)$
representing
\[ z= \lim_{\eps_{1}} ( z_{\bar{\epsilon}} ) \in \R^{p} . \]
\item {\tmstrong{Complexity and degree bounds}}: If $D$ is a bound on the
degrees of the polynomials in $f_{\overline{\eps}}
,g_{\overline{\eps}}$ and $G_{\eps}$ with respect to $U,V$ and
$\overline{\eps}$, then the degrees of the polynomials appearing in the
output are bounded by $D^{O ( 1 )}$, and the number of arithmetic operations
in $\D$ is bounded by $p^{O ( 1 )} D^{O (t)}$.

\item {\tmstrong{Procedure}}:

\item {\tmstrong{Step 1}}. Using Algorithm 12.16 (Bounded Algebraic
Sampling) from {\cite{BPRbook2}} in the ring $\D \left[ \overline{\eps}
\right]$ with input $\left\{ f_{\overline{\eps}} ,g_{\overline{\eps}}
\right\}$ obtain a set of univariate representation $u= \left(
h_{\overline{\eps}} ,H= ( h_{0} ,h_{U} ,h_{V} ) \right)$. For each such $u$,
substitute the rational function $\frac{h_{U}}{h_{0}}$ for $U$ in
$f_{\overline{\eps}}$ and its derivatives with respect to $U$. Similarly,
substitute the rational functions $\frac{h_{U}}{h_{0}} ,
\frac{h_{V}}{h_{0}}$ for $U,V$ in $g_{\overline{\eps}}$ and its derivatives
with respect to $V$ to obtain $f_{\overline{\eps} ,u} ,g_{\overline{\eps}
,u}$ and $\mathrm{Der} \left( f_{\overline{\eps}} \right)_{u}$,$\mathrm{Der}
\left( g_{\overline{\eps}} \right)_{u}$.

\item {\tmstrong{Step 2}}. Using Algorithm 10.13 (Univariate Sign
Determination) from \cite{BPRbook2} with input $h_{\eps}$,
$\mathrm{Der} \left( h_{\eps} \right)$ and $\mathrm{Der} \left(
f_{\overline{\eps}} \right)_{u}$,$\mathrm{Der} \left( g_{\overline{\eps}}
\right)_{u}$, determine a real univariate representation $u= \left(
h_{\overline{\eps}} , \tau'_{\overline{\eps}} ,E= ( h_{0} ,h_{U} ,h_{V} )
\right)$ whose associated point is $\left( u_{\overline{\eps}}
,v_{\overline{\eps}} \right)$ where $u_{\overline{\eps}}$ is associated to
the Thom encoding $\left( f_{\overline{\eps}} , \sigma_{\overline{\eps}}
\right)$ and $v_{\overline{\eps}}$ is associated to the Thom encoding
$\left( g_{\overline{\eps}} , \tau_{\overline{\eps}} \right)$ over $\left(
f_{\overline{\eps}} , \sigma_{\overline{\eps}} \right)$. Substitute the
rational functions $\frac{h_{U}}{h_{0}} , \frac{h_{V}}{h_{0}}$ for $U,V$ in
$G_{\overline{\eps}}$, to obtain $G_{\overline{\eps} ,u}$ and replace
$\left( g_{\overline{\eps}} , \tau_{\overline{\eps}} ,G_{\overline{\eps}}
\right)$ by the new real univariate representation $\left( e, \tau_{e}
,G_{\overline{\eps} ,u} \right)$, where $e \in \D \left[ \overline{\eps , }
T \right]$.

\item {\tmstrong{Step 3}}. Replace (see Notation \ref{not:order})
$g_{\overline{\eps}}$ by $\overline{\eps}^{-o_{\overline{\eps}} \left(
g_{\overline{\eps}} \right)}g_{\overline{\eps}}$. Denote by $g (T)$ the
polynomial obtained by substituting successively $\eps_{t}$ by $0$, and then
$\eps_{t-1}$ by 0, and so on, and finally $\eps_{1}$ by $0$, in
$g_{\overline{\eps}}$. Similarly denote by $G (T)$ the polynomials obtained
by substituting successively $\eps_{t}$ by $0$, and then $\eps_{t-1}$ by 0,
and so on, and finally $\eps_{1}$ by $0$, in $G_{\overline{\eps}}$.

\item {\tmstrong{Step 4}}. Compute the set $\Sigma$ of Thom encodings of
roots of $g (T)$ using Algorithm 10.13 (Univariate Sign Determination) from
{\cite{BPRbook2}}. Denoting by $\mu_{\sigma}$ the multiplicity of the root
of $g (T)$ with Thom encoding $\sigma$, define $G_{\sigma}$ as the $(
\mu_{\sigma} -1 )$-st derivative of $G$ with respect to $T$.

\item {\tmstrong{Step 5}}. Identify the Thom encoding $\sigma$ and
$G_{\sigma}$ representing $z$ using Algorithm 11.13 (Univariate Sign
Determination) from {\cite{BPRbook2}}, by checking whether a ball of
infinitesimal radius $\delta$ ($1 \gg \delta \gg \overline{\eps} >0$) around
the point $z$ represented by the real univariate representation $g, \sigma
,G_{\sigma}$ contains $z_{\overline{\eps}}$.
\end{itemize}}}

\begin{proof}[Proof of correctness]
The correctness of the algorithm follows from
the correctness of Algorithm 12.16 (Bounded Algebraic Sampling)  and
Algorithm 10.13 (Univariate Sign Determination) from {\cite{BPRbook2}}.
\end{proof}

\begin{proof}[Proof of complexity and degree bounds]
 Using the complexity analysis
of Algorithm 12.16 (Bounded Algebraic Sampling) from {\cite{BPRbook2}} the
number of arithmetic operations in the ring $\D \left[ \overline{\eps}
\right]$ is bounded by $p^{O ( 1 )} D^{O ( 1 )}$. The degrees in $T,
\overline{\eps}$ of the polynomials in $H$, as well as $f_{\overline{\eps}
,u} ,g_{\overline{\eps} ,u}$, are also bounded by $D^{O ( 1 )}$. From the
complexity analysis of Algorithm 10.13 (Univariate Sign Determination) from
{\cite{BPRbook2}}, the number of arithmetic operations in $\D \left[
\overline{\eps} \right]$ in Step 2 is bounded by $D^{O ( 1 )}$. It also
follows that the number of arithmetic operations in $\D$ of these two steps
is bounded by $D^{O ( t )}$. From the bound on the degrees in
$\overline{\eps}$ it follows that the complexity of Step 3 is bounded by
$p^{O ( 1 )} D^{O ( t )}$. Since the degrees do not increase in Step 3, it
follows again from the complexity analysis of Algorithm 10.13 (Univariate
Sign Determination) from {\cite{BPRbook2}} that the number of arithmetic
operations in $\D$ in Steps 4 and 5 is bounded by $p^{O ( 1 )} D^{O ( t )}$.
Thus, the total complexity is bounded by $p^{O ( 1 )} D^{O ( t )}$, and the
degrees of the polynomials appearing in the output are bounded by $D^{O ( 1
)}$.
\end{proof}

\begin{definition}
Let $({g}_{1} , \tau_{1} )$, $({g}_{2} , \tau_{2} )$ be
Thom encodings above a Thom encoding $( h, \sigma )$. We denote by $z \in
\R$ the point specified by $( h, \sigma )$, and by $(z,a)$,
$(z,b)$ the points specified by $({g}_{1} , \tau_{1} )$ and $(
{g}_{2} , \tau_{2} )$.

A {\tmem{curve segment representation}}$\left(u, \rho \right)$ on $(
{g}_{1} , \tau_{1} ) , ({g}_{2} , \tau_{2} )$ over $( h,
\sigma )$ {\subindex{Curve segment}{representation}} is:
\begin{itemize}
\item a parametrized univariate representation with parameters $(X_{\le i}
)$,  
i.e.,
\[ u= (f(Z,X,U),f_{0} (Z,X,U),f_{1} (Z,X,U), \ldots ,f_{k} (Z,X,U))
 \subset \D [ Z,X,U ] , \]
\item a sign condition $\rho$ on $\Der (f)$ such that for every $x \in
(a,b)$ there exists a real root $u (x)$ of $f (z,x,U)$ with Thom encoding
$\rho$ and $f_{0} (z,x,u ( x ) ) \ne 0$.
\end{itemize}
The {\tmem{curve segment associated to $u, \rho$}} is the semi-algebraic
function $\upsilon$ which maps a point $x$ of $(a,b)$ to the point of
$\R^{k}$ defined by
\[ \upsilon (x) = \left( x, \frac{f_{1} (z,x,u ( x ) )}{f_{0} (z,x,u ( x )
 )} , \ldots , \frac{f_{k} (z,x,u ( x ) )}{f_{0} (z,x,u ( x ) )} \right) .
\]
It is a continuous injective semi-algebraic function.
\end{definition}

In the following algorithms we will need to compute descriptions of the limits
of certain curve segments. These limits are computed using a slight
modification of Algorithm (Limit of a Curve) from {\cite{BRMS10}}.}}

{\algorithm{\label{alg:bounded_refined}
[Divide and Conquer Roadmap Algorithm
for Bounded Algebraic Sets]
\begin{itemize}
\item {\tmstrong{Input}}: A polynomial $P \in \D [ X_{1} , \ldots ,X_{k} ]$
such that $\ZZ ( P, \R^{k} )$ is bounded, and a set $A= \{
u_{1} , \ldots ,u_{m} \}$ of real univariate representations with associated
set of points $\mathcal{A} \subset \ZZ ( P, \R^{k} )$.

\item {\tmstrong{Output}}: A  roadmap, $\mathrm{DCRM} ( \ZZ ( P , \R^{k} )
,\mathcal{A})$,  of $\ZZ ( P, \R^{k} )$
containing $\mathcal{A}$.

\item {\tmstrong{Complexity and degree bounds}}: Let $d = \deg ( P )$ and
suppose that the degrees of the polynomials appearing in the real univariate
representation in $A$ with associated point $p_{i}$ be bounded by $D_{i}$. \
The complexity is bounded by
\[ \left( 1+ \sum_{i=1}^{m} D_{i}^{O ( \log^{2} ( k' ) )} \right) ( k^{\log
 ( k' )} d )^{O ( k \log ( k' )^{2} } . \]
Moreover, the degrees of the polynomials appearing in the descriptions of
the curve segments and points in the output are bounded by
\[ ( \max_{1 \leq i \leq k}D_{i} )^{O ( \log ( k' ) )} ( k^{\log ( k' )} d
 )^{O ( k\log ( k' ) )} . \]

\item {\tmstrong{Procedure}}:

\item {\tmstrong{Step 1}}. For each $i,1 \leq i \leq m$, call Algorithm
\ref{alg:bounded} with input $P$ and $\{ u_{i} \}$, and denote by
$\mathrm{Leav}$ the collection of the leaves of $\mathrm{Tree} ( V, \{ u_{i} \}
),i=1, \ldots ,m$.

\item {\tmstrong{Step 2}}. Define
\begin{eqnarray*}
\Gamma & := & \bigcup_{\mathfrak{n} \in \mathrm{Leav}} \mathrm{Bas} (
 \mathfrak{n} ) ) .
\end{eqnarray*}
Compute $\lim_{\zeta_{1}} ( \Gamma )$ as follows.
\begin{enumerate}
\item For each $\mathfrak{n} \in \mathrm{Leav}$, let
\[ F ( \mathfrak{n} ):= \sum_{f \in \mathcal{T} ( \mathfrak{n} )} f^{2}
 \in \D_{\log ( k' )} [ T_{\mathrm{fix} ( \mathfrak{n} )} ]
 . \]
\item Compute using Algorithm 12.16 (Bounded Algebraic Sampling) from
{\cite{BPRbook2}} in the ring $\D_{\log ( k' )} [ \theta (
\mathfrak{n} ) ]$ and $F ( \mathfrak{n} )$ as input, and compute a real
univariate representation $u_{\mathfrak{n}} = ( h_{\mathfrak{n}} ,
\sigma_{\mathfrak{n} ,} H_{\mathfrak{n}} )$ whose associated point $w (
\mathfrak{n} ) \in \R_{\log ( k' )}^{\mathrm{Fix} ( \mathfrak{n} )}$.

\item For every $\mathcal{Q}' \subset \mathcal{Q} ( \mathfrak{n} )$, apply
Algorithm 15.2 (Curve Segments) from {\cite{BPRbook2}} with input the Thom
encoding $( h_{\mathfrak{n}} , \sigma_{\mathfrak{n} ,} )$ specifying
$c_{\mathfrak{n}} \in \R_{\log( k' )}$ and the polynomial
\[ G ( \mathfrak{n} )= \sum_{g \in \mathcal{P} ( \mathfrak{n}
 )_{u_{\mathfrak{n}}} \cup \mathcal{Q}'_{u_{\mathfrak{n}}}} g^{2} \in
 \D_{\log ( k' )} [ U,X_{\mathrm{Fix} ( \mathfrak{n} ) +1} , \ldots
 ,X_{k} ] . \]
with parameter $X_{\mathrm{Fix} ( \mathfrak{n} ) +1}$ to obtain a set,
$\Gamma_{\mathfrak{n}}$, of curve segments with associated sets contained
in $\{ w ( \mathfrak{n} ) \} \times \R^{k- \mathrm{Fix} ( \mathfrak{n}
)}_{\log ( k' )}$. Subdivide the interval of definition of each curve
segment into pieces above which the sign of $\mathcal{Q} ( \mathfrak{n} )$
remains fixed on the curve segment, and retain only those contained
$\mathrm{Bas} ( \mathcal{P} ( \mathfrak{n} ) ,\mathcal{Q} ( \mathfrak{n} ) )$ using
Algorithm 11.19 (Restricted Elimination) from {\cite{BPRbook2}}.

\item Now apply Algorithm (Limit of a Curve) from {\cite{BRMS10}} to the
curve segments output in Step 3, with the following modifications : we use
Algorithm \ref{alg:limitofaboundedpoint} (Limit of a Bounded Point)
instead of the corresponding algorithm in {\cite{BRMS10}} and replace the
various instances of substituting $\eps$ by $0$ by substituting
successively $\eta_{t}$ by $0$, and then $\eta_{t-1}$ by 0, and so on, and
finally $\eta_{1}$ by $0$.
\end{enumerate}
\end{itemize}}}

\begin{proof}[Proof of correctness]
Follows from the correctness of Algorithm
\ref{alg:bounded}, Algorithm \ref{alg:limitofaboundedpoint}, Algorithm
(Limit of a Curve) from {\cite{BRMS10}}, and Proposition \ref{prop:additive}.
\end{proof}

\begin{proof}[Proof of complexity and degree bounds]
 The complexity of Step 1 is
\[ \left( 1+ \sum_{i=1}^{m} D_{i}^{O ( \log^{2} ( k' ) )} \right) ( k^{\log (
 k' )} d )^{O ( k \log^{2} ( k' ) )} \]
by the complexity of Algorithm \ref{alg:bounded}. Now using the complexity
analysis of Algorithm 12.16 (Bounded Algebraic Sampling) from
{\cite{BPRbook2}}, as well as those of {\tmname{}}Algorithm 15.2 (Curve
Segments) from {\cite{BPRbook2}}, and Algorithm 11.19 (Restricted
Elimination) from {\cite{BPRbook2}}, the number of arithmetic operations in
$\D [ \eta ]$ in Steps 2(i), 2(ii), and 2(iii) is bounded by
\[ \left( 1+ \sum_{i=1}^{m} D_{i}^{O ( \log ( k' ) )} \right) ( k^{\log (
 k' )} d )^{O ( k \log ( k' ) )} . \]
So using the fact that the degrees in $\eta$ are bounded by $D_{i}^{O (
\log ( k' ) )} ( k^{\log ( k' )} d )^{O ( k\log ( k' ) )}$ in the complexity
analysis of Algorithm \ref{alg:bounded}, the number of arithmetic operations
in $\D$ is bounded by
\[ \left( 1+ \sum_{i=1}^{m} D_{i}^{O ( \log^{2} ( k' ) )} \right) ( k^{\log (
 k' )} d )^{O ( k \log^{2} ( k' ) )} . \]
Note also that the degrees (in the parameters as well as in $\eta$) of the
polynomials used to describe the curve segments output in Step 2(iii) are
bounded by $D_{i}^{O ( \log ( k' ) )} ( k^{\log ( k' )} d )^{O ( k\log ( k' )
)}$. Finally using the complexity analysis of Algorithm (Limit of a
Curve) from {\cite{BRMS10}}, as well as that of Algorithm
\ref{alg:limitofaboundedpoint} (Limit of a Bounded Point), we get that the
complexity of Step 2(iv) is bounded by
\[ \left( 1+ \sum_{i=1}^{m} D_{i}^{O ( \log^{2} ( k' ) )} \right) ( k^{\log (
 k' )} d )^{O ( k \log^{2} ( k' ) )} \]
as well. Thus, the total complexity is bounded by
\[ \left( 1+ \sum_{i=1}^{m} D_{i}^{O ( \log^{2} ( k' ) )} \right) ( k^{\log (
 k' )} d )^{O ( k \log^{2} ( k' ) )} . \]
Moreover, it follows from above and the complexity of Algorithm (Limit of
a Curve) from {\cite{BRMS10}}, and Algorithm
\ref{alg:limitofaboundedpoint} (Limit of a Bounded Point), that the degrees
of the polynomials appearing in the descriptions of the curve segments and
points in the output are bounded by $( \max_{1 \leq i \leq k}D_{i} )^{O (
\log ( k' ) )} ( k^{\log ( k' )} d )^{O ( k\log ( k' ) )}$. 
\end{proof}

We are now in a position to describe a divide-and-conquer algorithm for
computing a roadmap of a general (i.e.,  possibly unbounded) algebraic set. The
procedure of passing from the bounded case to the unbounded one is the same as
that used in {\cite{BPR99}} as well as in {\cite{BRMS10}}. We include it here
for the sake of completeness.

We first need a notation.

\begin{notation}
\label{not:bigcauchy} Let $F \in \D [ X ]$ be given by $F= a_{p} X^{p} +
\cdots +a_{q} X^{q}$. We denote
\begin{eqnarray*}
c ( P ) & = & \left( \sum^{p}_{i=q} \left\vert \tfrac{a_{i}}{a_{q}} \right\vert
\right)^{-1} .
\end{eqnarray*}
\end{notation}

{\algorithm{\label{alg:main} [Divide and Conquer Roadmap Algorithm for General
Algebraic Sets]
\begin{itemize}
\item {\tmstrong{Input}}: A polynomial $P \in \D [ X_{1} , \ldots ,X_{k} ]$
such that $\ZZ ( P, \R^{k} )$ is bounded and strongly of dimension
$\leq k'$, and a set $A= \{
u_{1} , \ldots ,u_{m} \}$ of real univariate representations with associated
set of points $\mathcal{A} \subset \ZZ ( P, \R^{k} )$.

\item {\tmstrong{Output}}: A roadmap, $\mathrm{DCRM} ( \ZZ ( P , \R^{k} )
,\mathcal{A})$,  of $\ZZ ( P, \R^{k} )$
containing $\mathcal{A}$.

\item {\tmstrong{Complexity and degree bounds}}: Let $d = \deg ( P )$, and
suppose that the degrees of the polynomials appearing in the real univariate
representation in $A$ with associated point $p_{i}$ are bounded by $D_{i}$.
The complexity is bounded by
\[ \left( 1+ \sum_{i=1}^{m} D_{i}^{O ( \log^{2} ( k' ) )} \right) ( k^{\log
 ( k' )} d )^{O ( k \log ( k' )^{2} } . 
 \]
The degrees of the polynomials appearing in the descriptions of the curve
segments and points in the output are bounded by 
\[
( \max_{1 \leq i \leq k} 
D_{i} )^{O ( \log ( k' ) )} ( k^{\log ( k' )} d )^{O ( k\log ( k' ) )}.
\]

\item {\tmstrong{Procedure}}:

\item {\tmstrong{Step 1}}. Introduce new variables $X_{k+1}$ and $\eps$ and
replace $P$ by the polynomial
\[ P_{\eps} =P^{2} + \left( \eps^{2} \left( \sum_{i=1}^{k+1} X_{i}^{2}
 \right) -1 \right)^{2} . \]
\item {\tmstrong{Step 2}}. Replace $A$ by $A_{\eps}$, the set of real
univariate representations representing the elements $\mathcal{A}_{\eps}$ of
$\ZZ ( P_{\eps} , \R \la \eps \ra^{k} )$ above
the points represented by $\mathcal{A}$ using Algorithm 12.16 (Bounded
Algebraic Sampling) {\cite{BPRbook2}}.

\item {\tmstrong{Step 3}}. Call Algorithm \ref{alg:bounded_refined} (Divide
and Conquer Roadmap Algorithm for Bounded Algebraic Sets) with input
$P_{\eps} ,A$, performing arithmetic operations in the domain $\D \left[
\eps \right]$. The algorithm outputs a roadmap 
\[
\mathrm{DCRM} \left( \ZZ ( P_{\eps} , \R \la \eps \ra^{k+1} )
,\mathcal{A}_{\eps} \right)
\] 
composed of points and curves whose description involves $\eps$.

\item {\tmstrong{Step 4}}. Denote by $\mathcal{L}$ the set of polynomials in
$\D \left[ \eps \right]$ whose signs have been determined in the preceding
computation and take
\[ a= \min_{F \in \mathcal{L}}c ( F ) \]
using Notation \ref{not:bigcauchy}.

Replace $\eps$ by $a$ in the polynomial $P_{\eps}$ to get a polynomial
$P_{a}$. Replace $\eps$ by $a$ in the output roadmap to obtain a roadmap
which when projected to $\R^{k}$ gives a roadmap of $\ZZ ( P ,
\R^{k} ) \cap \bar{B}_{k} \left( 0, \tfrac{1}{a} \right)$, (where
$\bar{B}_{k} (x,r)$ is the $k$-dimensional closed ball of center $x$ and
radius $r$) containing the finite set of points $\mathcal{A}$.

\item {\tmstrong{Step 5}}. In order to extend the roadmap outside
$\bar{B}_{k} \left( 0, \tfrac{1}{a} \right)$ collect all the points $( y_{1}
, \ldots ,y_{k} ,y_{k+1} ) \in \R \la \eps \ra^{k+1}$ in
the roadmap 
\[
\mathrm{DCRM} ( \ZZ ( P_{\eps} , \R \la \eps
\ra^{k+1} ) ,\mathcal{A}_{\eps})
\]
 which satisfies
$\eps^{2} ( y_{1}^{2} + \cdots +y_{k+1}^{2} ) =1$. Each such point is
described by a real univariate representation involving $\eps$. Add to the
roadmap the curve segment obtained by first forgetting the last coordinate
and then treating $\eps$ as a parameter which varies over $( 0,a ]$ to
get a roadmap $\mathrm{DCRM} ( \ZZ ( P , \R^{k} )
,\mathcal{A})$.
\end{itemize}}}

\begin{proof}[Proof of correctness] The correctness follows from the correctness
of Algorithm \ref{alg:bounded_refined} (Divide and Conquer Roadmap Algorithm
for Bounded Algebraic Sets), taking into account that a bounded algebraic set is always
strongly of dimension $\leq k-1$.
\end{proof}

\begin{proof}[Proof of complexity and degree bounds]
 It is clear that the
complexity is dominated by that of the third step. The degree bounds on the
polynomials in the output follow from those of Algorithm
\ref{alg:bounded_refined} (Divide and Conquer Roadmap Algorithm for Bounded
Algebraic Sets).
\end{proof}

\subsection{Proofs of Theorem \ref{thm:mainbis}, Theorem \ref{thm:main} and Theorem \ref{thm:geodesic}}

\begin{proof}[Proof of Theorem \ref{thm:mainbis}]
Clear from the 
the proofs of correctness, complexity and degree bounds
of Algorithm \ref{alg:bounded_refined}
(Divide and Conquer Roadmap Algorithm
for Bounded Algebraic Sets).
\end{proof}

\begin{proof}[Proof of Theorem \ref{thm:main}]
Clear from the 
proofs of correctness, complexity and degree bounds
of Algorithm \ref{alg:main} (Divide and Conquer Roadmap Algorithm for
General Algebraic Sets).
\end{proof}

\begin{proof}[Proof of Theorem \ref{thm:geodesic}]
 Using Theorem \ref{thm:main},
we have that $x,y$ can be connected by a semi-algebraic path, consisting of
at most $( k ^{\log ( k )} d )^{O ( k\log ( k ) )}$ curve segments, and
each curve segment has degree bounded by $( k ^{\log ( k )} d )^{O ( k 
\log ( k ) )}$. It is clear that a curve segment of degree bounded by $D$
meets a generic hyperplane in at most $O ( D^{2} )$ points. It now follows
immediately from the Cauchy-Crofton formula {\cite{Santalo}} that the length
of each curve segment appearing in the path is bounded by $( k ^{\log (k
)} d )^{O ( k\log ( k ) )}$, and finally that the total length of the
path is also bounded by $( k ^{\log ( k )} d )^{O ( k\log ( k ) )}$.
\end{proof}

\section{Annex : Auxiliary proofs}\label{sec:aux-proofs}\label{auxproof}

The main purpose of this Annex is to prove Lemma \ref{lemmafora}, Lemma
\ref{lemmaforb} and Proposition \ref{prop:properties-of-special} which were
stated earlier in the paper. We first proof an auxiliary proposition that will
be needed in the proofs of Lemma \ref{lemmaforb} and Proposition
\ref{prop:properties-of-special}.

\begin{proposition}
\label{prop:morsegeneral}Let $\mathcal{P},\mathcal{Q} \subset \mathbb{R} [
X_{1} , \ldots ,X_{k} ]$ be a finite set of polynomials and let $G \in
\mathbb{R} [ X_{1} , \ldots ,X_{k} ]$. Suppose that $[ a,b ] \subset
\mathbb{R}$ is such that $G ( \mathrm{Crit} ( \mathcal{P},\mathcal{Q},G ) )
\cap ( a,b )$ is empty. Let $S= \mathrm{Bas} ( \mathcal{P},\mathcal{Q} )$ be
bounded. Then, for any $c \in ( a,b )$ the semi-algebraic set $S_{a<G<b}$
is homeomorphic to $S_{G=c} \times ( a,b )$ by a fiber preserving
homeomorphism. In particular, for each semi-algebraically connected
component $C$ of $S_{a<G<b}$, $C_{G=c}$ is non-empty and semi-algebraically
connected.
\end{proposition}

\begin{proof} The condition that $G ( \mathrm{Crit} ( \mathcal{P},\mathcal{Q},G ) ) \cap [
a,b ]$ is empty implies that $S_{a<G<b}$ is a Whitney-stratified set with
strata $\ZZ ( \mathcal{P} \cup \mathcal{Q}' ,\mathbb{R}^{k} )_{a<G<b} \cap
S$, $\mathcal{Q}' \subset \mathcal{Q}$, where the dimension of $\ZZ (
\mathcal{P} \cup \mathcal{Q}' ,\mathbb{R}^{k} )_{a<G<b}$ is equal to $k- (
\mathrm{card} ( \mathcal{P} \cup \mathcal{Q}')$ if non-empty. The
proposition now follows from a basic result in stratified Morse theory (see
for example, Theorem SMT Part A in {\cite{GM}}).
\end{proof}

\subsection{Proof of properties of $G$-critical values}

\begin{proof}[Proof of Lemma \ref{lemmafora}]
Let $x$ and $y$ be two points of
$C_{G \le a}$ and $\gamma : [0,1] \to C$ be a semi-algebraic path connecting
$x$ to $y$ inside $C$. We want to prove that there is a semi-algebraic path
connecting $x$ to $y$ inside $C_{G \le a}$.

If $\mathrm{Im}( \gamma ) \subset C_{G \le a}$ there is nothing to prove.
If $\mathrm{Im}( \gamma ) \not\subset C_{G \le a}$,
\[ \exists c \in \R , \forall a<d<c, \mathrm{Im}( \gamma ) \cap
 S_{G=d} \neq \emptyset . \]
Let $\eps$ be a positive infinitesimal. Then
\[ \Ext ( \gamma ([0,1]), \R \la \eps\ra
 ) \cap \Ext \left( S, \R \la \eps \ra \right)_{G=a+ \eps} \neq
 \emptyset 
 \]
using {\cite[Proposition 3.17]{BPRbook2}}.

Since
\[ \{u \in [0,1] \subset \R \la \eps \ra \mid \Ext ( \gamma ,
 \R \left \la \eps \ra \right )(u) \in \Ext
 \left( S, \R \la \eps \ra \right)_{G<a+ \eps} \} \]
and
\[ \{u \in [0,1] \subset \R \la \eps \ra \mid \Ext (
 \gamma , \R \left \la \eps \ra \right )(u) \in \Ext \left( S ,
 \R \la \eps \ra \right)_{[a+ \eps \leq G \leq b]} \} \]
are semi-algebraic subsets of $[0,1] \subset \R\la \eps \ra$
there exists by {\cite[Corollary 2.79]{BPRbook2}} a finite partition
$\mathfrak{P}$ of $[0,1] \subset \R \la \eps \ra$ such that
for each open interval $(u,v)$ of $\mathfrak{P}$, $\Ext ( \gamma , \R \left
\la \eps \ra \right ) (u,v)$ is either contained in
\[
\Ext \left( S, \R \la \eps \ra \right)_{G<a+ \eps},
\] 
or in
\[
\Ext \left( S, \R \la \eps \ra \right)_{[a+ \eps
\leq G \leq b]},
\] 
with $\gamma (u)$ and $\gamma (v)$ in $C_{G=a+\varepsilon}$.

If $\Ext ( \gamma , \R \left \la \eps \ra \right ) (u,v)$ is contained in
\tmtextrm{$\Ext \left( S, \R \la \eps \ra \right)$}\tmrsub{$[a+ \eps
\leq G \leq b]$}, we can replace $\gamma$ by a semi-algebraic path
$\gamma'_{[a,b]}$ connecting $\gamma (u)$ to $\gamma (v)$ inside $C_{G<a+
\eps}$. Note that there is no critical point of $G$ in $\Ext \left( S,
\R \la \eps \ra \right)_{[a+ \varepsilon \leq G \leq b]}$ by
{\cite[Proposition 3.17]{BPRbook2}}.

By Proposition \ref{prop:morsegeneral}, if $D$ is a semi-algebraically
connected component of 
\[
\Ext (S, \R \left \la \eps \ra \right )_{a+ \eps
\leq G \leq b},
\] 
$D_{G=a+ \eps}$ is a semi-algebraically connected
component of $\Ext (S, \R \left \la \eps \ra \right )_{G=a+ \eps}$.

Construct a semi-algebraic path $\gamma'$ from $x$ to $x'$ inside $C_{G \le
a+ \eps}$, obtained by concatenating pieces of $\gamma$ inside $\Ext
\left( S, \R \la \eps \ra \right)_{G<a+ \eps}$ and the paths
$\gamma'_{(u,v)}$ connecting $\gamma (u)$ to $\gamma (v)$ for $(u,v)$ such
that $\Ext ( \gamma , \R \left \la \eps \ra \right ) (u,v) \subset
\Ext \left( S, \R \la \eps \ra \right)_{a+ \eps \leq G \leq
b}$. Note that such a semi-algebraically connected path $\gamma'$ is
closed and bounded. Applying {\cite[Proposition 12.43]{BPRbook2}},
\tmtextrm{lim}$_{\eps} ( \gamma' ([0,1]))$ is semi-algebraically connected,
contains $x$ and $x'$ and is contained in \tmtextrm{lim}$_{\eps} (C_{G \le
a+ \eps} ) =C_{G \le a}$. This is enough to prove the lemma.
\end{proof}

\begin{proof}[Proof of Lemma \ref{lemmaforb}]
We are going to prove the lemma
assuming $\R = \mathbb{R}$. The general case follows from a standard
transfer argument that we omit.

 Part (\ref{item:lemmaforb1}) follows immediately from Proposition
\ref{prop:morsegeneral}. We now prove Part (\ref{item:lemmaforb2}). Since $\mathcal{M}$ is
finite, there is a point $x \in C_{G=b}$ which is not a critical point of
$G$ on $S$. Let $\mathcal{P}_{x}=\{ P\in\mathcal{P} \cup \mathcal{Q}
\mid P ( x )= 0 \}$. Then, since $x$ is not a $G$-critical point of $\ZZ
( \mathcal{P}_{x} , \R^{k} )$, it follows that $T_{x}\ZZ (
\mathcal{P}_{x} , \R^{k} )$ is not tangent to the level surface of $G$
defined by $G=b$, and hence for $\eps >0$ infinitesimal, $B_{k} \left( x,
\eps \right)_{G<b} \cap T_{x}\ZZ ( \mathcal{P}_{x} , \R \la \eps
\ra^{k} )$ is not empty (where $B_{k} (x,r)$ is the $k$-dimensional
open ball of center $x$ and radius $r$), and hence $B_{k} \left( x, \eps
\right)_{G<b} \cap$ $\ZZ ( \mathcal{P}_{x} , \R \la \eps \ra^{k}
)$ is not empty either. Let $y\in B_{k} \left( x, \eps
\right)_{G<b} \cap \ZZ ( \mathcal{P}_{x} , \R \la \eps \ra^{k}
)$. Then, since $\lim_{\eps} y =x $ and $y \in \ZZ (
\mathcal{P}_{x} , \R \la \eps \ra^{k} )$, we have that for each
polynomial $P\in \mathcal{P} \cup \mathcal{Q}$, $P ( x )$ and $P ( y )$
have the same signs, and hence $y\in S$. Moreover, since $S$ is closed and
$\lim_{\eps} y = x\in C$, we have that $y\in \Ext \left( C, \R \la
\eps \ra \right)$. Now using the transfer principle it follows $C_{G<b}$
is non-empty.

Parts (\ref{item:lemmaforb2a}) and (\ref{item:lemmaforb2b}) are immediate consequences of Proposition
\ref{prop:morsegeneral}.

We prove (\ref{item:lemmaforb2c}). Clearly, $\cup_{i=1}^{r} \overline{B_{i}} \subset C$.
Suppose that $x \in A=C \setminus \cup_{i=1}^{r} \overline{B_{i}}$. For
$r>0$ and small enough, $B_{k} (x,r) \cap C_{G<b} = \emptyset$ . Note that
$G ( b ) =b$, since otherwise $x$ belongs to $C_{G<b}$, and thus to one of
the $B_{i}$'s.

Applying Proposition \ref{prop:morsegeneral}, we deduce from the fact that
$B_{k} (x,r) \cap C_{G<b} =B_{k} (x,r)_{G<b} \cap C= \emptyset$ that $x$ is
a $G$-critical point of $\ZZ ( \mathcal{P}_{x} , \R^{k} )$. In
other words $x \in \mathcal{M}$. But since by assumption $\mathcal{M}$
is finite, this implies that $A$ is a finite set and is thus closed. Since
$C$ is semi-algebraically connected and $\cup_{i=1}^{r} \overline{B_{i}}$
is closed and non-empty, $A$ must be empty.
\end{proof}

\subsection{Proof of properties of $( B,G )$-pseudo-critical values}

\begin{proof}[Proof of Proposition \ref{prop:properties-of-special}]
We are going
to prove the proposition assuming $\R = \mathbb{R}$. The general case
follows from a standard transfer argument that we omit. Let $\mathcal{F}= \{
F_{1} , \ldots ,F_{s} \}$, where $\mathcal{P}= \{ F_{1} , \ldots ,F_{m} \}
,\mathcal{Q}= \{ F_{m+1} , \ldots ,F_{s} \}$.
\begin{enumerate}
\item It follows from the good rank property of the matrix $B$ and
Proposition \ref{prop:parametrized-dimension} that for any $I \subset [
1,s ]$, $\sigma \in \{ -1,1 \}^{I}$, the algebraic sets $\ZZ (
\tilde{\mathcal{F}}_{I, \sigma ,B} , \R \langle \gamma \rangle^{k}
)$ have at most isolated singularities. It now follows by the
semi-algebraic Sard's theorem {\cite{BCR}}, that the set of critical
values of $G$ restricted to the various $\ZZ (
\tilde{\mathcal{F}}_{I, \sigma ,B} , \R \langle \gamma \rangle^{k}
)$ is a finite set. This proves that the set $\mathcal{D} (
\mathcal{P} \cup \mathcal{Q},B,G )$ is finite.

\item Let
\begin{eqnarray*}
\tilde{\mathcal{F}} & =& \bigcup_{i=1}^{m} \{ \pm F_{i} + \gamma H_{i}
\} \cup \bigcup^{s}_{i=m+1} \{ F_{i} + \gamma H_{i} \}
\end{eqnarray*}
and $\tilde{S} = \mathrm{Bas} ( \emptyset , \tilde{\mathcal{F}} ) \subset \R
\langle \gamma \rangle^{k}$.Clearly, $\Ext \left( S, \R \langle
\gamma \rangle \right) \subset \tilde{S}$, and since $S$ is bounded, there
exists a unique semi-algebraically connected component, $\tilde{D}$ of
$\tilde{S}_{a \leq G \leq b}$, such that $\tilde{D}$ is bounded over $\R$,
 $\Ext \left( D, \R \langle \gamma \rangle \right) \subset \tilde{D}$,
and $\lim_{\gamma} \tilde{D} =D$. It is also clear from the definition of
$\tilde{S}$ and the fact that $c \in \R$, that $\Ext \left( D_{G=c} , \R
\langle \gamma \rangle \right) \subset \tilde{D}_{G=c}$, and $D_{G=c} =
\lim_{\gamma} ( \tilde{D}_{G=c} )$. Since, $\tilde{D}$ is bounded
over $\R$, in order to prove that $D_{G=c}$ is non-empty and
semi-algebraically connected, it suffices to prove (using Proposition
12.43 in {\cite{BPRbook2}}) that $\tilde{D}_{G=c}$ is non-empty and
semi-algebraically connected. Since, $\lim_{\gamma} ( \mathrm{Crit} (
\tilde{\mathcal{F}}_{I, \sigma ,B} ,G ) ) \cap$ $[a,b] \setminus \{c\}$ is
empty for all $I \subset [ 1,s ]$ and $\sigma \in \{ -1,1 \}^{I}$, it
follows that for all $I \subset [ 1,s ]$ and $\sigma \in \{ -1,1 \}^{I}$,
$( \mathrm{Crit} ( \tilde{\mathcal{F}}_{I, \sigma ,B} ,G ) ) \cap [ a,b ]$
belong to the interval $[c- \delta ,c+ \delta ]$ (respectively, $[ a,a+
\delta ]$ if $c=a$, and $[ b- \delta ,b ]$ if $c=b$), where $\delta >0$ is
a new infinitesimal.

We claim that $\Ext( \tilde{D} , \R \langle \gamma , \delta \rangle
)_{c- \delta \leq G \leq c+ \delta}$ in case $c \neq a,b$
(respectively, $\Ext ( \tilde{D} , \R \langle \gamma , \delta
\rangle )_{a \leq G \leq a+ \delta}$ if $c=a$, and $\Ext \left(
\tilde{D} , \R \langle \gamma , \delta \rangle \right)_{b- \delta \leq G
\leq b}$ if $c=b$) is non-empty and semi-algebraically connected. We prove
the statement only in case $c \neq a,b$; the proof in the other two cases
being very similar.

Let $x,y$ be any two points in $\Ext( \tilde{D} , \R \langle \gamma
, \delta \rangle)_{c- \delta \leq G \leq c+ \delta}$. 

We show that
there exists a semi-algebraic path connecting $x$ to $y$ lying within
$\Ext ( \tilde{D} , \R \langle \gamma , \delta \rangle)_{c-
\delta \leq G \leq c+ \delta}$. Since, $\tilde{D}$ itself is
semi-algebraically connected, there exists a semi-algebraic path, $\gamma
: [0,1] \rightarrow \Ext( \tilde{D} , \R \langle \gamma , \delta
\rangle)$, with $\gamma (0) =x, \gamma (1) =y$, and $\gamma (t) \in
\Ext ( \tilde{D} , \R \langle \gamma , \delta \rangle) ,0 \leq
t \leq 1$. If $\gamma (t) \in \Ext ( \tilde{D} , \R \langle \gamma ,
\delta \rangle)_{c- \delta \leq G \leq c+ \delta}$ for all~$t \in
[0,1]$, we are done. Otherwise, the semi-algebraic path $\gamma$ is the
union of a finite number of closed connected pieces $\gamma_{i}$ lying
either in $\Ext ( \tilde{D} , \R \langle \gamma , \delta \rangle
)_{a \leq G \leq c- \delta} $, $\Ext ( \tilde{D} , \R
\langle \gamma , \delta \rangle)_{c- \delta \leq G \leq c+
\delta}$, or $\Ext ( \tilde{D} , \R \langle \gamma , \delta \rangle
)_{c+ \delta \leq G \leq b}$.

By Lemma \ref{lemmafora} the semi-algebraically connected components of
\[
\Ext( \tilde{D} , \R \langle \gamma , \delta \rangle)_{G=c-
\delta}
\] 
(resp. 
$
\Ext( \tilde{D} , \R \langle \gamma , \delta
\rangle)_{G=c+ \delta}
$
) are in 1-1 correspondence with the
semi-algebraically connected components of 
\[
\Ext( \tilde{D} , \R
\langle \gamma , \delta \rangle)_{a \leq G \leq c- \delta}\]
 (resp.
$\Ext (\tilde{D} , \R \langle \gamma , \delta \rangle)_{c+
\delta \leq G \leq b}$) containing them. In particular, notice that since
$D$ is non-empty being a semi-algebraically connected component of the
semi-algebraic set $S_{a \leq G \leq b}$, $\tilde{D}$ is also non-empty,
and hence at least one of the sets 
\[
\Ext ( \tilde{D} , \R \langle
\gamma , \delta \rangle)_{a \leq G \leq c- \delta},
\]
\[
\Ext (
\tilde{D} , \R \langle \gamma , \delta \rangle)_{c- \delta \leq G
\leq c+ \delta},
\]
\[
\Ext ( \tilde{D} , \R \langle \gamma , \delta
\rangle)_{c+ \delta \leq G \leq b}
\]
is not empty. 
Thus, 
$\Ext
( \tilde{D} , \R \langle \gamma , \delta \rangle)_{c- \delta
\leq G \leq c+ \delta} \neq \emptyset.
$ 
But, since 
\[
\lim_{\gamma}
\left( \Ext( \tilde{D} , \R \langle \gamma , \delta \rangle
)_{c- \delta \leq G \leq c+ \delta} \right) =D_{c},
\] 
this proves
that $D_{c}$ is not empty. 

We can now replace using Proposition
\ref{prop:morsegeneral} each of the $\gamma_{i}$ lying in 
\[
\Ext(
\tilde{D} , \R \langle \gamma , \delta \rangle)_{a \leq G \leq c-
\delta}
\]
(resp. $\Ext ( \tilde{D} , \R \langle \gamma , \delta
\rangle)_{c+ \delta \leq G \leq b}$) with endpoints in 
$
\Ext
( \tilde{D} , \R \langle \gamma , \delta \rangle)_{G=c-
\delta} 
$
(resp. $\Ext ( \tilde{D} , \R \langle \gamma , \delta
\rangle )_{G=c+ \delta}$) by another segment with the same endpoints
but lying completely in 
$\Ext ( \tilde{D} , \R \langle \gamma ,
\delta \rangle)_{G=c- \delta}
$ 
(resp. $\Ext ( \tilde{D} , \R
\langle \gamma , \delta \rangle )_{G=c+ \delta}$). We thus obtain a
new semi-algebraic path $\gamma'$ connecting $x$ to $y$ and lying inside
$
\Ext ( \tilde{D} , \R \langle \gamma , \delta \rangle)_{c-
\delta \leq G \leq c+ \delta}.
$
 This proves that 
 $
 \Ext ( \tilde{D} ,
\R \langle \gamma , \delta \rangle)_{c- \delta \leq G \leq c+
\delta}
$
 is semi-algebraically connected and hence so is $\tilde{D}_{c} =
\lim_{\delta} \tilde{D}_{c- \delta \leq G \leq c+ \delta}$ (by Proposition
12.43 in {\cite{BPRbook2}}).
\end{enumerate}

\end{proof}

\section{Acknowledgments}
The authors are grateful to Purdue
University, Universit{\'e} de Rennes 1, the Institute of Mathematical
Sciences, National University of Singapore, the RIP program in Oberwolfach,
and IPAM, UCLA for hosting them on several occasions during which time this
work was completed.

\bibliographystyle{abbrv}
\bibliography{master}

\end{document}